\documentclass[11pt]{amsart} %loads amsmath, amsfonts, amsthm

\usepackage{amssymb,mathtools,hyperref,tensor,enumitem,verbatim}
\usepackage[margin=1in]{geometry}
\usepackage[all,cmtip]{xy}
\usepackage{todonotes}
\usepackage{cite}

\newcommand{\sm}[1]{\left[ \begin{smallmatrix} #1 \end{smallmatrix} \right]}
\newcommand{\bmat}[1]{\begin{bmatrix} #1 \end{bmatrix}}
\newcommand{\subgrp}[1]{\langle #1 \rangle}
\newcommand{\set}[1]{\left\{ #1 \right\}}

\newcommand{\bs}[1]{\boldsymbol{#1}}
\newcommand{\wh}[1]{\widehat{#1}}

\newcommand{\ol}[1]{\overline{#1}}

\renewcommand{\mod}{\; \mathrm{mod} \;}

\DeclareMathOperator{\coker}{coker}
\DeclareMathOperator{\diag}{diag}
\DeclareMathOperator{\Ext}{Ext}
\DeclareMathOperator{\bbExt}{\mathbb{E}xt}
\DeclareMathOperator{\opH}{H}
\DeclareMathOperator{\Hom}{Hom}
\DeclareMathOperator{\id}{id}
\DeclareMathOperator{\im}{im}

\DeclareMathOperator{\RHom}{RHom}
\DeclareMathOperator{\sgn}{sgn}
\DeclareMathOperator{\Tot}{Tot}

\newcommand{\rmI}{\mathrm{I}}
\newcommand{\rmII}{\mathrm{II}}
\newcommand{\Kz}{K\!z}

\newcommand{\gotimes}{\tensor[^g]{\otimes}{}}
\newcommand{\Hbul}{\opH^\bullet}
\newcommand{\Hs}{\opH_{[s]}}
\newcommand{\Hone}{\opH_{[1]}}
\newcommand{\Hstar}{\opH_*}
\newcommand{\Hstarbul}{\Hstar^\bullet}

\newcommand{\zero}{\ol{0}}
\newcommand{\one}{\ol{1}}
\newcommand{\ev}{\textup{ev}}
\newcommand{\op}{\textup{op}}

\newcommand{\ve}{\varepsilon}

\newcommand{\G}{\mathbb{G}}
\newcommand{\N}{\mathbb{N}}
\newcommand{\Z}{\mathbb{Z}}

\newcommand{\calO}{\mathcal{O}}
\newcommand{\calP}{\mathcal{P}}

\newcommand{\calV}{\mathcal{V}}

\newcommand{\bsa}{\bs{A}}
\newcommand{\bsb}{\bs{B}}
\newcommand{\bsc}{\bs{c}}
\newcommand{\bse}{\bs{e}}
\newcommand{\bsg}{\bs{\Gamma}}
\newcommand{\bsi}{\bs{I}}

\newcommand{\bsionezero}{\bsi_0^{(1)}}

\newcommand{\bsir}{\bsi^{(r)}}
\newcommand{\bsijell}{\bsi_\ell^{(j)}}
\newcommand{\bsijzero}{\bsi_0^{(j)}}

\newcommand{\bsirell}{\bsi_\ell^{(r)}}
\newcommand{\bsirone}{\bsi_1^{(r)}}
\newcommand{\bsirzero}{\bsi_0^{(r)}}
\newcommand{\bsl}{\bs{\Lambda}}
\newcommand{\bso}{\bs{\Omega}}
\newcommand{\bsp}{\bs{\calP}}
\newcommand{\bsPi}{\bs{\Pi}}
\newcommand{\bspi}{\bs{\Pi}}

\newcommand{\bss}{\bs{S}}

\newcommand{\bsv}{\bs{\calV}}

\newcommand{\bsvzero}{\bsv_{\zero}}

\newcommand{\bscrpi}{\bsc_r^{\bspi}}
\newcommand{\bserpi}{\bse_r^{\bspi}}

\newcommand{\bsbst}{\bsb_{[s,t]}}
\newcommand{\bsbszero}{\bsb_{[s,0]}}
\newcommand{\bsbbar}{\ol{\bsb}}

\newcommand{\fS}{\mathfrak{S}}
\newcommand{\fsvec}{\mathfrak{svec}}

\newcommand{\Uone}{U_{\one}}
\newcommand{\Uzero}{U_{\zero}}
\newcommand{\Vone}{V_{\one}}
\newcommand{\Vzero}{V_{\zero}}

\newcommand{\dol}{\ol{d}}
\newcommand{\Jbar}{\ol{J}}
\newcommand{\Qbar}{\ol{Q}}
\newcommand{\Tbar}{\ol{T}}
\newcommand{\Vbar}{\ol{V}}
\newcommand{\Wbar}{\ol{W}}
\newcommand{\partialbar}{\ol{\partial}}
\newcommand{\alphabar}{\ol{\alpha}}
\newcommand{\rhobar}{\ol{\rho}}
\newcommand{\vebar}{\ol{\ve}}

\newcommand{\chat}{\wh{c}}
\newcommand{\whe}{\wh{e}}
\newcommand{\whbse}{\wh{\bse}}

\input cyracc.def
\font\tencyr=wncyr10
\def\cyr{\tencyr\cyracc}
\newcommand{\Sha}{{\mbox{\cyr Sh}}}
\newcommand{\sha}{{\mbox{\cyr sh}}}

\newcommand{\Shabar}{\ol{\Sha}}
\newcommand{\shabar}{\ol{\sha}}

\numberwithin{equation}{subsection}

\newtheorem{theorem}{Theorem}[subsection]
\newtheorem{proposition}[theorem]{Proposition}
\newtheorem{corollary}[theorem]{Corollary}
\newtheorem{lemma}[theorem]{Lemma}

\theoremstyle{definition}
\newtheorem{convention}[theorem]{Convention}
\newtheorem{definition}[theorem]{Definition}

\newtheorem{remark}[theorem]{Remark}

\title[Superized Troesch complexes and cohomology]{Superized Troesch complexes and cohomology \\ for strict polynomial superfunctors}

\author{Christopher M.\ Drupieski}
\address{Department of Mathematical Sciences,
		DePaul University,
		Chicago, IL 60614, USA}
\email{c.drupieski@depaul.edu}

\author{Jonathan R. Kujawa}
\address{Department of Mathematics \\
		University of Oklahoma \\
		Norman, OK 73019, USA}
\email{kujawa@math.ou.edu}

\thanks{The first author was supported in part by Simons Collaboration Grant for Mathematicians No.\ 426905. The second author was supported in part by Simons Collaboration Grant for Mathematicians No.\ 525043.}

\date{\today}


\allowdisplaybreaks

\begin{document}

\begin{abstract}
We adapt a construction due to Troesch to the category of strict polynomial superfunctors in order to construct complexes of injective objects whose cohomology is isomorphic to Frobenius twists of the (super)symmetric power functors. We apply these complexes to construct injective resolutions of the even and odd Frobenius twist functors, to investigate the structure of the Yoneda algebra of the Frobenius twist functor, and to compute other extension groups between strict polynomial superfunctors.
\end{abstract}

\maketitle

\setcounter{tocdepth}{3}

%\tableofcontents

\section{Introduction}

\subsection{Background}

Let $k$ be a field of positive characteristic $p$, and let $r$ be a positive integer. The category $\calP = \calP_k$ of strict poly\-nomial functors over $k$ was defined by Friedlander and Suslin \cite{Friedlander:1997}, as part of their proof that the cohomology ring of a finite $k$-group scheme is a finitely-generated $k$-algebra. One of their main objects of study was the extension algebra $\Ext_{\calP}^\bullet(I^{(r)},I^{(r)})$ for the $r$-th Frobenius twist functor $I^{(r)}$. In general, Friedlander and Suslin were unable to give explicit injective resolutions for objects in $\calP$. But for $p=2$ they showed that a result of Franjou, Lannes, and Schwartz \cite{Franjou:1994} could be generalized to construct, for each $n \geq 1$, an explicit injective resolution of the twisted symmetric power functor $S^{n(r)} = S^n \circ I^{(r)}$. In the case $n=1$, this injective resolution permitted an easy calculation of $\Ext_{\calP}^\bullet(I^{(r)},I^{(r)})$.

Later, Troesch \cite{Troesch:2005} showed that Friedlander and Suslin's construction could be generalized to odd characteristics, but at the cost that the underlying object of the construction was no longer a chain complex in the ordinary sense. Rather, it is a $p$-complex, i.e., a graded object $C^\bullet$ equipped with a map $d: C^i \to C^{i+1}$ such that $d^p = 0$ instead of the usual $d^2 = 0$. Contracting these $p$-complexes, one gets injective resolutions in $\calP$ of the twisted symmetric powers. The power and utility of Troesch's $p$-complexes were subsequently put on display in several papers by Touz\'{e}. In \cite{Touze:2010a}, Touz\'{e} used Troesch's $p$-complexes as an ingredient in exhibiting certain universal bifunctor cohomology classes, which were then used in work with van der Kallen \cite{Touze:2010b} to prove cohomological finite generation for arbitrary reductive groups. And in \cite{Touze:2012}, Touz\'{e} applied Troesch's $p$-complexes to give efficient recalculations of results by Franjou, Friedlander, Scorichenko, and Suslin \cite{Franjou:1999} and by Cha\l upnik \cite{Chal-upnik:2005}, as well as new results that were inaccessible via previous methods.

\subsection{Main results}

Suppose $k$ is a field of characteristic $p \geq 3$. In this paper we show how Troesch's $p$-complex construction generalizes to the category $\bsp$ of strict polynomial super\-functors over $k$. Strict polynomial superfunctors are the natural generalization to vector superspaces of the `ordinary' strict polynomial functors defined by Friedlander and Suslin. The category $\bsp$ was defined by Axtell \cite{Axtell:2013}, and was used by the first author to help prove cohomological finite generation for finite supergroup schemes \cite{Drupieski:2016}. We show that applying Troesch's construction in the category $\bsp$ yields a $p$-complex $\bsb_{p^rn}(r)$ of injective objects whose cohomology is the twisted (super)symmetric power $\bss^{n(r)}$ (Theorem \ref{theorem:B(r)-cohomology}). In contrast to the classical situation, however, the contraction $T(\bss^n,r)$ of $\bsb_{p^rn}(r)$ is not an injective resolution, but has nonzero cohomology in degrees $\ell \cdot (p^r-1)$ for all $0 \leq \ell \leq n$ (Corollary \ref{cor:T(Sn,r)-cohomology}).

Though they are less well-behaved than their classical counterparts, the superized Troesch complexes and their contractions are still useful for calculations. In the case $n=1$, we show in Section \ref{subsection:injective-resolutions} how the contracted complexes $T(\bsi,r) = T(\bss^1,r)$ and $\Tbar(\bsi,r) = \bsPi \circ T(\bsi,r) \circ \bsPi$ can be spliced together to build injective resolutions for the even and odd Frobenius twist functors
	\[
	\bsirzero: V \mapsto \Vzero^{(r)} \quad \text{and} \quad \bsirone: V \mapsto \Vone^{(r)}.
	\]
Here $\bsPi$ is (one of) the parity change functor(s) in $\bsp$. In the case $r=1$, the injective resolutions are made completely explicit through a specific choice of splicing map exhibited in Section \ref{subsec:epsilon-prime}. In Section \ref{subsec:ExtP(Ir,Ir)} we show how the injective resolutions can be used to give a more direct recalculation of the vector space structure of the extension algebra $\Ext_{\bsp}^\bullet(\bsir,\bsir)$, which was a main focus of the first author's work in \cite{Drupieski:2016}. And in Section \ref{subsec:ExtP(Ir,Ir)-multiplication} we show how the injective resolutions can be used to give relatively quick direct proofs of some of the multiplicative relations in $\Ext_{\bsp}^\bullet(\bsir,\bsir)$ that were previously inaccessible in \cite{Drupieski:2016}.

Finally, in Section \ref{section:FFSS} we look at the utility of the superized Troesch complexes for doing cohomology calculations in the spirit of work by Franjou, Friedlander, Scorichenko, and Suslin (FFSS). In \cite{Franjou:1999}, FFSS used inductive arguments based on hypercohomology spectral sequences to compute extension groups of the form $\Ext_{\calP}^\bullet(X^{*(r)},Y^{*(s)})$ when $X^*$ and $Y^*$ are classical exponential functors.\footnote{A few cases not covered by \cite{Franjou:1999} were later computed using similar techniques by Cha\l upnik in \cite{Chal-upnik:2008}; see also \cite{Touze:2014}.} One of the key ingredients to their inductive approach was the De Rham complex functor. In \cite{Drupieski:2016}, the first author showed that there is a natural super-analogue of the De Rham complex functor, but its cohomology is spread over a wider range of degrees than in the classical case (see \cite[Remark 4.1.3]{Drupieski:2016}), which presents significant complications for trying to use it to imitate the FFSS approach. On the other hand, the contraction $T(\bss^n,r)$ of the superized Troesch complex $\bsb_{p^rn}(r)$ exhibits behavior much closer to the classical De Rham complex (even better behavior, in some respects, since $T(\bss^n,r)$ consists of injective objects). In Section \ref{subsec:proof-j=1} we show how $T(\bss^n,1)$ can be used to kick off an induction argument similar to that employed by FFSS. For example, we show for $j \geq 1$ that the natural cup product maps
	\begin{align*}
	\Ext_{\bsp}^\bullet(\bsirone,S_0^{p^{r-j}(j)})^{\otimes d} &\to \Ext_{\bsp}^\bullet(\Gamma_1^{d(r)},S_0^{dp^{r-j}(j)}), \\
	\Ext_{\bsp}^\bullet(\bsirone,\Lambda_0^{p^{r-j}(j)})^{\otimes d} &\to \Ext_{\bsp}^\bullet(\Gamma_1^{d(r)},\Lambda_0^{dp^{r-j}(j)}),
	\end{align*}
factor to induce isomorphisms of graded vector spaces
	\begin{align*}
	S^d( \Ext_{\bsp}^\bullet(\bsirone,S_0^{p^{r-j}(j)}) ) &\cong \Ext_{\bsp}^\bullet(\Gamma_1^{d(r)},S_0^{dp^{r-j}(j)}), \\
	\Lambda^d( \Ext_{\bsp}^\bullet(\bsirone,\Lambda_0^{p^{r-j}(j)}) ) &\cong \Ext_{\bsp}^\bullet(\Gamma_1^{d(r)},\Lambda_0^{dp^{r-j}(j)});
	\end{align*}
see Theorem \ref{theorem:FFSS-4.5}. Here $\Gamma_1^{d(r)} = \Gamma^d \circ \bsirone$, $S_0^{n(j)} = S^n \circ \bsijzero$, and $\Lambda_0^{n(j)} = \Lambda^n \circ \bsijzero$ are classical exponential functors pre-composed with either the even or odd Frobenius twist functor. These isomorphisms are similar in form to those in the classical situation \cite[Theorem 4.5]{Franjou:1999}, although the graded vector spaces over which the symmetric and exterior powers are taken are now infinite-dimensional rather than finite-dimensional. In \cite[\S5]{Franjou:1999}, FFSS also calculated all extension groups in $\calP$ from a Frobenius twist of a `more projective' exponential functor to a Frobenius twist of a `less projective' exponential functor; we leave it as an unstated theorem that these other calculations also admit natural generalizations to superfunctors (Remark \ref{rem:unstated-theorem}).

\subsection{Future directions}

The calculations in Section \ref{section:FFSS} provide compelling evidence that many cohomology calculations in $\calP$ admit natural generalizations to $\bsp$. But the category $\bsp$ also admits interesting complications over its classical counterpart. For example, extension groups in $\bsp$ frequently appear to be much larger than their counterparts in $\calP$, and pre-composition with the Frobenius twist functor $\bsir: V \mapsto V^{(r)}$ does not, in general, induce an isomorphism on $\Hom$-groups or an injection on $\Ext$-groups in $\bsp$. % Indeed, even basic questions related to functor cohomology remain open in the super setting. For example, as far as we are aware, the rational cohomology for the supergroup $\operatorname{GL}(m|n)$ is not yet known.
It would be interesting to more fully develop the theory and computational methods of functor cohomology in this setting (e.g., for classical supergroups as in \cite{Touze:2012} and additive supercategories as in \cite{djament2021functor}).

In a somewhat different direction, for any superalgebra $A$ there is a category of strict polynomial superfunctors corresponding to the generalized Schur superalgebra $S^{A}(n,d)$ which appears in \cite{Evseev:2017,Kleshchev:2020}. See \cite{Touze:2014b} and the references therein for when $A$ is a commutative ring and where functor cohomology has applications to classical invariants from ring theory. It would be interesting to consider functor cohomology for these categories as well.

\subsection{Structure of the paper}

In Section \ref{sec:preliminaries} we review basic definitions and constructions related to strict polynomial superfunctors; for a more detailed account, the reader can consult \cite{Drupieski:2016}. In Section \ref{sec:superTroesch} we give the construction of the superized Troesch complexes, beginning in Section \ref{subsec:p-complexes} with a review of $p$-complexes, and in Section \ref{subsec:tensor-product-p-complexes} with a discussion of the K\"{u}nneth Theorem for tensor products of normal $p$-complexes (Theorem \ref{thm:Kunneth}). Most of the work for the construction comes in Sections \ref{subsec:r=1} and \ref{subsec:r-greater-1}, where we verify the cohomology of the resulting $p$-complexes. This step is more involved than in the classical (non-super) situation, because the $p$-complexes have nonzero cohomology in positive degrees. In Section \ref{section:Ext(Ir,Ir)} we look at applications of the superized Troesch complexes to the extension algebra $\Ext_{\bsp}^\bullet(\bsir,\bsir)$, and in Section \ref{section:FFSS} we look at applications to the methods of Franjou, Friedlander, Scorichenko, and Suslin.

\subsection{Conventions}

We generally follow the conventions of \cite{Drupieski:2016}. Except when indicated, $k$ will denote a field of characteristic $p \geq 3$, all vector spaces will be $k$-vector spaces, and all unadorned tensor products will denote tensor products over $k$. Given a $k$-vector space $V$, let $V^* = \Hom_k(V,k)$ be its $k$-linear dual, and let $V^{(r)} = k \otimes_{\varphi} V$ be its $r$-th Frobenius twist, the $k$-vector space obtained via base change along the $p^r$-power map $\varphi: \lambda \mapsto \lambda^{p^r}$. Given $v \in V$, let $v^{(r)} = 1 \otimes_{\varphi} v \in V^{(r)}$.

Set $\Z_2 = \Z/2\Z = \{ \zero,\one \}$. Following the literature, we use the prefix `super' to indicate that an object is $\Z_2$-graded. We denote the decomposition of a vector superspace into its $\Z_2$-homogeneous components by $V = \Vzero \oplus \Vone$, calling $\Vzero$ and $\Vone$ the even and odd subspaces of $V$, respectively, writing $\ol{v} \in \Z_2$ to denote the superdegree of a homogeneous element $v \in \Vzero \cup \Vone$. When written without additional adornment, we consider the field $k$ to be a superspace concentrated in even super\-degree. Whenever we state a formula in which homogeneous degrees of elements are specified, we mean that the formula is true as written for homogeneous elements, and that it extends linearly to non-homogeneous elements. We write $\cong$ to denote an even (i.e., degree-preserving) isomorphism, and use the symbol $\simeq$ for odd (i.e., degree-reversing) isomorphisms.

A \emph{graded superspace} (superalgebra, etc.) is a $\Z \times \Z_2$-graded vector space. Given a graded super\-space $V$ and a homogeneous element $v \in V$, we write $\deg(v)$ for the $\Z$-degree of $v$. If $V = \bigoplus_{j \in \Z} V^j$ is a graded superspace, we write $V\subgrp{i}$ for the graded superspace defined by $V\subgrp{i}^j = V^{j-i}$. We say that a graded superalgebra is \emph{graded-commutative} if for all homogeneous $a,b \in A$, one has
	\[
	ab = (-1)^{\ol{a}\cdot \ol{b} + \deg(a) \cdot \deg(b)} ba.
	\]
Given graded superalgebras $A$ and $B$, the \emph{graded tensor product} of $A$ and $B$, denoted $A \gotimes B$, is the graded superalgebra whose underlying superspace is $A \otimes B$, in which the $\Z$-grading is defined by $\deg(a \otimes b) = \deg(a) + \deg(b)$, and whose product is defined by
	\[
(a \otimes b)(c \otimes d) = (-1)^{\ol{b} \cdot \ol{c} + \deg(b) \cdot \deg(c)} (ac \otimes bd).
	\]

Let $\N = \set{0,1,2,3,\ldots}$ denote the set of non-negative integers.

\section{Preliminaries on strict polynomial superfunctors} \label{sec:preliminaries}

\subsection{The category of strict polynomial superfunctors} \label{subsec:definitions}

Let $\fsvec$ be the category whose objects are $k$-super\-spaces and whose morphisms are arbitrary $k$-linear maps between superspaces. Let $\bsv$ be the full subcategory of finite-dimensional objects in $\fsvec$.  There are natural $\Z_{2}$-gradings on sets of morphisms in $\fsvec$ and $\bsv$. %Let $\fsvec_\ev$ and $\bsv_\ev$ denote the underlying even (abelian) subcategories of $\fsvec$ and $\bsv$, having the same objects but only the even linear maps as morphisms.
Given $V,W \in \bsv$, let $T: V \otimes W \rightarrow W \otimes V$ be the supertwist map, $v \otimes w \mapsto (-1)^{\ol{v} \cdot \ol{w}} w \otimes v$. For each $n \in \N$, there exists a unique right action of the symmetric group $\fS_n$ on $V^{\otimes n}$ such that the transposition $(i,i+1) \in \fS_n$ acts as $(1_V)^{\otimes(i-1)} \otimes T \otimes (1_V)^{\otimes(n-i-1)}$. In general, for $v_1,\ldots,v_n \in V$ and $\sigma \in \fS_n$, one has
	\[
	(v_1 \otimes \cdots \otimes v_n) \cdot \sigma = (-1)^\mu \cdot v_{\sigma(1)} \otimes \cdots \otimes v_{\sigma(n)}, \quad \text{where} \quad \mu = \sum_{\substack{i < j \\\sigma(i) > \sigma(j)}} \ol{v_{\sigma(i)}} \cdot \ol{v_{\sigma(j)}}.
	\]
This extends to a right $\fS_n$-action on $\Hom_k(V^{\otimes n},W^{\otimes n})$, defined by $(\varphi \cdot \sigma)(z) = \varphi(z \cdot \sigma^{-1}) \cdot \sigma$. Then the standard superspace isomorphism $\Hom_k(V,W)^{\otimes n} \cong \Hom_k(V^{\otimes n},W^{\otimes n})$ restricts to
	\begin{equation} \label{eq:Sn-invariants}
	\bsg^n \Hom_k(V,W) := (\Hom_k(V,W)^{\otimes n})^{\fS_n} \cong \Hom_{k\fS_n}(V^{\otimes n},W^{\otimes n}).
	\end{equation}
Given $n \in \N$, define $\bsg^n \bsv$ to be the category whose objects are the same as those in $\bsv$, whose morphisms are defined by $\Hom_{\bsg^n \bsv}(V,W) = \bsg^n \Hom_k(V,W)$, and in which the composition of morphisms is defined via \eqref{eq:Sn-invariants} and the composition of $k\fS_n$-module homomorphisms.

\begin{definition} \label{def:superfunctor}
Let $n \in \N$. A \emph{strict polynomial superfunctor of degree $n$} is an even linear functor $F: \bsg^n \bsv \rightarrow \bsv$, i.e., a covariant functor such that for each $V,W \in \bsv$, the corresponding function $F_{V,W} : \bsg^n \Hom_k(V,W) \rightarrow \Hom_k(F(V),F(W))$ is an even linear map. Given degree-$n$ strict polynomial superfunctors $F$ and $G$, a homomorphism $\eta: F \rightarrow G$ of degree $\ol{\eta} \in \Z_{2}$ consists for each $V \in \bsv$ of a linear map $\eta(V) \in \Hom_k(F(V), G(V))$, such that for each $\phi \in \Hom_{\bsg^n \bsv}(V,W)$,
\[
\eta(W) \circ F(\phi) = (-1)^{\ol{\eta} \cdot \ol{\phi}} G(\phi) \circ \eta(V).
\]
We denote by $\bsp_n$ the category whose objects are the strict polynomial superfunctors of degree $n$ and whose morphisms are the homomorphisms between those functors. Set $\bsp = \bigoplus_{n \in \N} \bsp_n$.
\end{definition}

By definition, if $F \in \bsp_m$ and $G \in \bsp_n$ with $m \neq n$, then $\Hom_{\bsp}(F,G) = 0$. The category $\bsp$ is not an abelian category, though the underlying even subcategory $\bsp_\ev$ of $\bsp$, having the same objects as $\bsp$ but only the even homomorphisms, is an abelian category in which kernels and cokernels are computed ``pointwise'' in $\bsv$. For $F,G \in \bsp$, one has $\Hom_{\bsp_{\ev}}(F,G) = \Hom_{\bsp}(F,G)_{\zero}$. More generally, if $\eta \in \Hom_{\bsp}(F,G)$ is homogeneous, then the kernel, cokernel, and image of $\eta$ are also objects in $\bsp$.

\begin{remark} \label{remark:restriction-to-non-super}
Let $\bsvzero$ be the full subcategory of $\bsv$ consisting of just the purely even superspaces. Forgetting the $\Z_2$-grading, $\bsvzero$ identifies with the category $\calV$ of  finite-dimensional $k$-vector spaces. Let $F \in \bsp_n$. Forgetting the superspace structure on $F(U)$ for each $U \in \bsvzero$, it follows that $F|_{\calV} := F|_{\bsvzero}$ is an ordinary (non-super) strict polynomial functor. The map $F \mapsto F|_{\calV}$ thus defines an exact linear functor from $\bsp$ to the category $\calP$ of ordinary strict polynomial functors.%, which we call \emph{restriction from $\bsp$ to $\calP$}.
\end{remark}

\subsection{Basic examples and constructions} \label{subsec:examples}

\subsubsection{Parity change}

The parity change functors $\bspi \in \bsp_1$ and $\Pi \in \bsp_1$ act on objects by reversing the $\Z_2$-grading. On morphisms, $\bspi(\phi) = (-1)^{\ol{\phi}} \phi$ and $\Pi(\phi) = \phi$ as linear maps on the underlying vector spaces. Let $k^{0|1}$ be a one-dimensional purely odd superspace. Then $\bspi$ and $\Pi$ can be realized as $\bspi = k^{0|1} \otimes -$ and $\Pi = -\otimes k^{0|1}$, and the supertwist map induces an isomorphism $\bspi \cong \Pi$. The functor $\Pi$ can also be realized as $\Pi = \Hom_k(k^{0|1},-)$. Given $m,n \in \N$, set $k^{m|n} = k^m \oplus \Pi(k^n)$.

Given $v \in V$, let $\prescript{\pi}{}{v}$ and $v^\pi$ denote $v$ considered as an element of $\bspi(V)$ and $\Pi(V)$, respectively. Let $A$ be a $k$-superalgebra. To extend $\bspi$ and $\Pi$ to the category of left $A$-super\-modules, set $a.({}^\pi v) = (-1)^{\ol{a}} \cdot {}^\pi(a.v)$ and $a. (v^\pi) = (a.v)^\pi$. For right $A$-super\-modules, set $({}^\pi v).a = {}^\pi (v.a)$ and $(v^\pi).a = (-1)^{\ol{a}} \cdot (v.a)^\pi$. Recall that a $k$-linear map $\phi: V \rightarrow W$ is a left $A$-supermodule homomorphism if $\phi(a.v) = (-1)^{\ol{a} \cdot \ol{\phi}} a.\phi(v)$ for all homogeneous $a \in A$ and $v \in V$, and is a right homomorphism if $\phi(v.a) = \phi(v).a$. Then the maps ${}^\pi(-): v \mapsto {}^\pi v$ and $(-)^{\pi}: v \mapsto (-1)^{\ol{v}} v^\pi$ define odd (left or right) isomorphisms $V \simeq \bspi(V)$ and $V \simeq \Pi(V)$, respectively, which lift to odd isomorphisms $\bsi \simeq \bspi$ and $\bsi \simeq \Pi$. More generally, for each $F \in \bsp$, one gets $F = \bsi \circ F \simeq \bspi \circ F$ and $F = \bsi \circ F \simeq \Pi \circ F$.

\subsubsection{Symmetric powers}

The $n$-th symmetric power functor $\bss^n \in \bsp_n$ is defined on objects by
\[
\bss^n(V) = (V^{\otimes n})/\subgrp{z-(z.\sigma): z \in V^{\otimes n}, \sigma \in \fS_n}.
\]
Then $\bss(V) := \bigoplus_{n \in \N} \bss^n(V)$ is the \emph{symmetric superalgebra} on $V$. It is a commutative and cocommutative Hopf superalgebra, with the subspace $\bss^1(V) \cong V$ consisting of primitive elements for the coproduct. As an algebra, $\bss(V) \cong S(\Vzero) \gotimes \Lambda(\Vone)$.

\subsubsection{Exterior powers}

The $n$-th exterior power functor $\bsl^n \in \bsp_n$ is defined on objects by
\[
\bsl^n(V) = (V^{\otimes n})/\subgrp{z - (-1)^\sigma (z.\sigma): z \in V^{\otimes n}, \sigma \in \fS_n}.
\]
Then $\bsl(V) := \bigoplus_{n \in \N} \bsl^n(V)$ is the \emph{exterior superalgebra} on $V$. It is a graded-commutative and graded-cocommutative graded Hopf super\-algebra, with $\bsl^n(V)$ concentrated in $\Z$-degree $n$, and with the subspace $\bsl^1(V) \cong V$ consisting of primitive elements for the coproduct. As an algebra, $\bsl(V) \cong \Lambda(\Vzero) \gotimes S(\Vone)$, where we consider $\Lambda(\Vzero)$ and $S(\Vone)$ as $\Z$-graded via their natural decompositions $\Lambda(\Vzero) = \bigoplus_{n \in \N} \Lambda^n(\Vzero)$ and $S(\Vone) = \bigoplus_{n \in \N} S^n(\Vone)$.

\subsubsection{Divided powers} \label{subsubsec:dividedpowers}

The $n$-th divided power functor $\bsg^n \in \bsp_n$ is defined on objects by
\[
\bsg^n(V) = (V^{\otimes n})^{\fS_n} = \{ z \in V^{\otimes n} : z.\sigma = z \text{ for all $\sigma \in \fS_n$} \}.
\]
If $V$ is purely even, then $\bsg^n(V)$ is equal to $\Gamma^n(V)$, the usual $n$-th divided power of $V$. Let $J \subseteq \fS_{m+n}$ be a set of right coset representatives for the Young subgroup $\fS_m \times \fS_n$ of $\fS_{m+n}$. Then $\sum_{\sigma \in J} \sigma$ defines a natural transformation $\bsg^m \otimes \bsg^n \rightarrow \bsg^{m+n}$ (the shuffle product) that is independent of the choice of $J$. Summing over all $m,n \in \N$, one gets a product on $\bsg(V) := \bigoplus_{n \in \N} \bsg^n(V)$, and we call $\bsg(V)$ the \emph{divided power superalgebra} on $V$. It is a commutative and cocommutative Hopf superalgebra, whose coproduct $\Delta: \bsg(V) \to \bsg(V) \otimes \bsg(V)$ is the sum of the components
	\[
	\Delta_{m,n}: \bsg^{m+n}(V) = (V^{\otimes (m+n)})^{\fS_{m+n}} \hookrightarrow (V^{\otimes (m+n)})^{\fS_m \times \fS_n} = \bsg^m(V) \otimes \bsg^n(V).
	\]
Given a homogeneous vector $v \in V$, set $\gamma_a(v) = v^{\otimes a}$. Then $\gamma_a(v) \in \bsg^a(V)$ provided that $a \leq 1$ if $v$ is odd. Now given a homogeneous basis $v_1,\ldots,v_s$ for $V$, the set of monomials
	\[
	\set{ \gamma_{a_1}(v_1) \cdots \gamma_{a_s}(v_s): a_i \in \N \text{ with } a_i \leq 1 \text{ if $v_i$ is odd} }
	\]
is a basis for $\bsg(V)$. In particular, $\bsg(V)$ is generated as an algebra by the subspace $\Vone \subseteq \bsg^1(V)$ and the elements of the form $\gamma_{p^e}(v)$ for $v \in \Vzero$ and $e \geq 0$. As an algebra, $\bsg(V) \cong \Gamma(\Vzero) \otimes \Lambda(\Vone)$. 

\subsubsection{Alternating powers}

The $n$-th alternating power functor $\bsa^n \in \bsp_n$ is defined on objects by
\[
\bsa^n(V) = \{ z \in V^{\otimes n} : z.\sigma = (-1)^\sigma z \text{ for all $\sigma \in \fS_n$} \}.
\]
Let $J \subseteq \fS_{m+n}$ be a set of right coset representatives for the Young subgroup $\fS_m \times \fS_n$ of $\fS_{m+n}$. Then $\sum_{\sigma \in J} (-1)^\sigma \sigma$ defines a natural transformation $\bsa^m \otimes \bsa^n \rightarrow \bsa^{m+n}$ (the signed shuffle product) that is independent of the choice of $J$. Summing over all $m,n \in \N$, one gets a product on $\bsa(V) := \bigoplus_{n \in \N} \bsa^n(V)$, and we call $\bsa(V)$ the \emph{alternating power superalgebra} on $V$. It is a graded-commutative and graded-cocommutative graded Hopf super\-algebra, with $\bsa^n(V)$ concentrated in $\Z$-degree $n$, whose coproduct $\Delta: \bsa(V) \to \bsa(V) \gotimes \bsa(V)$ is the sum of the components
	\[
	\bsa^{m+n}(V) \cong \Hom_{\fS_{m+n}}(\sgn,V^{\otimes (m+n)}) \hookrightarrow \Hom_{\fS_m \times \fS_n}(\sgn,V^{\otimes (m+n)}) \cong \bsa^m(V) \otimes \bsa^n(V).
	\]
Here $\sgn$ denotes the one-dimensional sign representation of $\fS_{m+n}$. Given a homogeneous vector $v \in V$, set $\gamma_a(v) = v^{\otimes a}$. Then $\gamma_a(v) \in \bsa^a(V)$ provided that $a \leq 1$ if $v$ is even. Now given a homogeneous basis $v_1,\ldots,v_s$ for $V$, the set of monomials
	\[
	\set{ \gamma_{a_1}(v_1) \cdots \gamma_{a_s}(v_s): a_i \in \N \text{ with } a_i \leq 1 \text{ if $v_i$ is even} }
	\]
is a basis for $\bsa(V)$. As an algebra, $\bsa(V) \cong \Lambda(\Vzero) \gotimes \Gamma(\Vone)$.

\subsubsection{Duality} \label{subsubsec:duality}

Given $F \in \bsp_n$, the dual functor $F^\# \in \bsp_n$ is defined on objects by $F^\#(V) = F(V^*)^*$. On morphisms, the action of $F^\#$ is given by the composite map
	\[ %\label{eq:F-dual-on-morphisms}
	\bsg^n \Hom_k(V,W) \to \bsg^n \Hom_k(W^*,V^*) \xrightarrow{F} \Hom_k(F(W^*),F(V^*)) \to \Hom_k(F^\#(V),F^\#(W)),
	\]
where the first and last arrows are induced by sending a linear map to its transpose. Using the standard superspace map $\Phi: V \to V^{**}$, $\Phi(v)(f) = (-1)^{\ol{v} \cdot \ol{f}} f(v)$, which is an isomorphism for $V \in \bsv$, one can check that $\bsi \cong \bsi^\#$ and $F \cong F^{\#\#}$. Given $\eta \in \Hom_{\bsp}(F,G)$, define $\eta^\# \in \Hom_{\bsp}(G^\#,F^\#)$ by $\eta^\#(V) = \eta(V^*)^*$. Then $\eta \mapsto \eta^\#$ defines an isomorphism $\Hom_{\bsp}(F,G) \cong \Hom_{\bsp}(G^\#,F^\#)$, and $\eta^{\#\#}$ identifies via the isomorphisms $F \cong F^{\#\#}$ and $G \cong G^{\#\#}$ with $\eta$. If also $\sigma \in \Hom_{\bsp}(G,H)$, then $(\sigma \circ \eta)^\# = (-1)^{\ol{\sigma} \cdot \ol{\eta}} \eta^\# \circ \sigma^\#$. As discussed in \cite[\S2.6]{Drupieski:2016}, the identifications $\bss^1 = \bsi \cong \bsi^\# = (\bsg^1)^\#$ and $\bsl^1 = \bsi \cong \bsi^\# = (\bsa^1)^\#$ extend multiplicatively to isomorphisms $\bss \cong \bsg^\#$ and $\bsl \cong \bsa^\#$.

\subsubsection{Frobenius twists} \label{subsubsec:Frobenius}

Let $r$ be a positive integer, and let $\varphi: k \rightarrow k$ be the $p^r$-power map $\lambda \mapsto \lambda^{p^r}$. By abuse of notation, we also denote by $\varphi$ the map $\bss(V)^{(r)} = k \otimes_{\varphi} \bss(V) \to \bss(V)$ defined by $\lambda \otimes_{\varphi} z \mapsto \lambda \cdot z^{p^r}$. If $z = z_{\zero} + z_{\one}$ is the decomposition of $z$ into its even and odd parts, then $z_{\zero}$ and $z_{\one}$ commute in the ordinary (non-super) sense in $\bss(V)$, and hence $z^{p^r} = (z_{\zero})^{p^r}$ because $(z_{\one})^2 = 0$ in $\bss(V)$. It follows that $\varphi: \bss(V)^{(r)} \to \bss(V)$ is a $k$-superalgebra homomorphism whose image is contained in the subalgebra $\bss(\Vzero)$ of $\bss(V)$. By duality, there exists for each $V \in \bsv$ a superalgebra homomorphism $\varphi^\#: \bsg(V) \rightarrow \bsg(V)^{(r)}$. On generators,
\begin{equation} \label{eq:dualFrobenius}
\varphi^\#(z) = \begin{cases}
0 & \text{if $z \in \Vone \subseteq \bsg^1(V)$}, \\
\gamma_{p^{e-r}}(v)^{(r)} & \text{if $z = \gamma_{p^e}(v)$ for some $e \in \N$, $v \in \Vzero$, and $e \geq r$,} \\
0 & \text{if $z = \gamma_{p^e}(v)$ for some $e \in \N$, $v \in \Vzero$, and $e < r$.}
\end{cases}
\end{equation}
Then $\varphi^\#$ has image in the subalgebra $\bsg(\Vzero)^{(r)}$ of $\bsg(V)^{(r)}$. Now the $r$-th Frobenius twist functor $\bsir \in \bsp_{p^r}$ is defined on objects by $\bsir(V) = V^{(r)}$, and is defined on morphisms by
\begin{equation} \label{eq:IrVW}
\bsir_{V,W}: \bsg^{p^r} \Hom_k(V,W) \xrightarrow{\varphi^\#} [\bsg^1 \Hom_k(V,W)_{\ol{0}}]^{(r)} = \Hom_k(V^{(r)},W^{(r)})_{\ol{0}}.
\end{equation}
Since \eqref{eq:IrVW} has image in the space of even linear maps, it follows that there exist subfunctors $\bsirzero$ and $\bsirone$ of $\bsir$ such that $\bsirzero(V) = \Vzero^{(r)}$, $\bsirone(V) = \Vone^{(r)}$, and $\bsir = \bsirzero \oplus \bsirone$. Additionally,
	\[
	\bspi \circ \bsirzero \circ \bspi = \bsirone \quad \text{and} \quad \bspi \circ \bsirone \circ \bspi = \bsirzero.
	\]
	
If $F \in \calP_n$ is an ordinary strict polynomial functor, then $F \circ \bsir$ defines a strict polynomial superfunctor of degree $p^r n$, with the action of $F^{(r)}$ on morphisms defined by
	\begin{equation} \label{eq:Fr-on-morphisms}
	\bsg^{p^r n} \Hom_k(V,W) \xrightarrow{\varphi^\#} \bsg^n[\Hom_k(V^{(r)},W^{(r)})_{\ol{0}}] \xrightarrow{F} \Hom_k(F(V^{(r)}),F(W^{(r)})).
	\end{equation}
The second arrow in \eqref{eq:Fr-on-morphisms} is well-defined because for $U$ purely even, $\bsg^n(U)$ is equal to the usual divided power algebra $\Gamma^n(U)$ on $U$. Let $\pi: v \mapsto (-1)^{\ol{v}} v$ be the parity automorphism on $V$. Then the superspace structure on $F(V^{(r)})$ is defined by declaring the even (resp.\ odd) subspace of $F(V^{(r)})$ to be the $+1$ (resp.\ $-1$) eigenspace for the action of $F^{(r)}(\pi^{\otimes p^r n}): F^{(r)}(V) \to F^{(r)}(V)$. Now given either $F \in \bsp$ or $F \in \calP$, set $F^{(r)} = F \circ \bsir$, $F_0^{(r)} = F \circ \bsirzero$, and $F_1^{(r)} = F \circ \bsirone$. Then one has the identifications
	\[
	\bss^{(r)} = S_0^{(r)} \otimes \Lambda_1^{(r)}, \quad \bsl^{(r)} = \Lambda_0^{(r)} \gotimes S_1^{(r)}, \quad \bsa^{(r)} = \Lambda_0^{(r)} \gotimes \Gamma_1^{(r)}, \quad	\bsg^{(r)} = \Gamma_0^{(r)} \otimes \Lambda_1^{(r)}.
	\]
For $r \geq 1$, one gets
	\begin{equation} \label{eq:conjugateiso}
	\left. \begin{aligned}
	S_1^{n(r)} &\cong S_0^{n(r)} \circ \bspi \\
	\Lambda_1^{n(r)} &\cong \Lambda_0^{n(r)} \circ \bspi \\
	\Gamma_1^{n(r)} &\cong \Gamma_0^{n(r)} \circ \bspi
	\end{aligned} \right\rbrace \text{ if $n$ is even, and} \qquad
	\left. \begin{aligned}
	S_1^{n(r)} &\cong \bspi \circ S_0^{n(r)} \circ \bspi \\
	\Lambda_1^{n(r)} &\cong \bspi \circ \Lambda_0^{n(r)} \circ \bspi \\
	\Gamma_1^{n(r)} &\cong \bspi \circ \Gamma_0^{n(r)} \circ \bspi
	\end{aligned} \right\rbrace \text{ if $n$ is odd.} 
	\end{equation}
These identifications also hold with $\bspi$ replaced by $\Pi$. The functors $\bspi$ and $\Pi$ act differently on morphisms, but because $\varphi^\#$ annihilates the odd superdegree generators in $\bsg [ \Hom_k(V,W) ]$, this difference is immaterial in \eqref{eq:conjugateiso}.

\subsubsection{Parameterized functors}

Given $F \in \bsp_n$ and $V \in \bsv$, the \emph{parameterized functors} $F^V \in \bsp_n$ and $F_V \in \bsp_n$ are defined by $F^V = F(\Hom_k(V,-))$ and $F_V = F(V \otimes -)$. These functors satisfy the relations $F^V \cong F_{V^*}$, $(F^V)^\# \cong (F^\#)_V$, $(F^V)^U \cong F^{V \otimes U}$, and $(F_V)_U \cong F_{V \otimes U}$.

\subsubsection{Strict polynomial superfunctors on graded superspaces} \label{subsubsec:gradings}

Let $V = k^{m|n}$ for some $m,n \in \N$. By \cite[Lemma 5.1.1]{Drupieski:2019b}, evaluation on $V$ defines an exact functor from $\bsp_d$ to the category of rational $GL_{m|n}$-super\-modules. Then for each $F \in \bsp_d$, $F(V)$ decomposes as a direct sum of weight spaces,
	\[
	F(V) = \bigoplus_{\substack{d_1,\ldots,d_{m+n} \in \N \\ d_1+\cdots+d_{m+n} = d}} F(V)^{d_1,\ldots,d_{m+n}}
	\]
for the action of the diagonal subgroup scheme $T \cong (\G_m)^{\times (m+n)}$ of $GL_{m|n}$. This decomposition can be realized as follows: Given a field extension $K/k$ and a $k$-vector space $W$, set $W_K = W \otimes_k K$. By scalar extension, $F$ induces an even linear map
	\[
	F_K: [\bsg^d \Hom_k(V,V)]_K \to \Hom_k(F(V),F(V))_K \cong \Hom_K(F(V)_K,F(V)_K).
	\]
If $g = \diag(\lambda_1,\ldots,\lambda_{m+n}) \in T(K)$ is a diagonal matrix with coefficients in $K$, then $g$ defines an element of $\Hom_k(V,V)_K$, and hence $g^{\otimes d} \in \bsg^d [\Hom_k(V,V)_K] \cong [\bsg^n \Hom_k(V,V)]_K$. Then
	\begin{multline*}
	F(V)^{d_1,\ldots,d_{m+n}} = \Big\{ x \in F(V) : \text{ for all $K/ k$ and all } g = \diag(\lambda_1,\ldots,\lambda_{m+n}) \in T(K), \\
	F_K(g^{\otimes d})(x) = \lambda_1^{d_1} \cdots \lambda_{m+n}^{d_{m+n}} \cdot x \in F(V)_K \Big\}.
	\end{multline*}

Suppose that $V$ is a graded superspace, and that $v_1,\ldots,v_{m+n}$ is a homogeneous basis for $V$. Define $\G_m \to T \cong (\G_m)^{\times (m+n)}$ by $x \mapsto (x^{\deg(v_i)})_{1 \leq i \leq m+n}$. Then by restriction, $F(V)$ becomes a rational $\G_m$-module, and hence is endowed with a weight space decomposition $F(V) = \bigoplus_{i \in \Z} F(V)^i$. In this manner, we consider $F(V)$ as a graded superspace.

Now consider the parameterized functor $F(V \otimes -)$. For each $U \in \bsv$, the tensor product $V \otimes U$ inherits a $\Z$-grading via $\deg(v \otimes u) = \deg(v)$. Then $F(V \otimes U)$ is a graded superspace. Moreover, the decomposition $F(V \otimes U) = \bigoplus_{i \in \Z} F(V \otimes U)^i$ is natural with respect to $U$, and hence induces a corresponding decomposition of the functor $F_V = F(V \otimes -)$. Similarly, $\Hom_k(V,U)$ inherits a $\Z$-grading defined by $\Hom_k(V,U)^i = \Hom_k(V^{-i},U)$, which in turn induces a decomposition of the functor $F^V = F(\Hom_k(V,-))$.

\subsubsection{Projectives and injectives}

We say that $P \in \bsp$ is \emph{projective} if the functor $\Hom_{\bsp}(P,-): \bsp_\ev \rightarrow \fsvec_\ev$ is exact, and that $Q \in \bsp$ is \emph{injective} if $\Hom_{\bsp}(-,Q): \bsp_\ev \rightarrow \fsvec_\ev$ is exact. Here $\fsvec_\ev$ is the underlying (abelian) even subcategory of $\fsvec$, having the same objects as $\fsvec$ but only the even linear maps as morphisms. Through appropriate composition with parity change functors, one can check that a functor is projective (resp.\ injective) if and only if it is projective (resp.\ injective) in the abelian subcategory $\bsp_{\ev}$.

Given $V \in \bsv$ and $n \in \N$, set $\bsg^{n,V} = \bsg^n \Hom_k(V,-)$ and $\bss_V^n = \bss^n(V \otimes -)$. Then by Yoneda's Lemma and duality, there exist for each $F \in \bsp_n$ natural isomorphisms
\begin{equation} \label{eq:Yonedalemma}
\Hom_{\bsp_n}(\bsg^{n,V},F) \cong F(V) \quad  \text{and} \quad \Hom_{\bsp_n}(F,\bss_V^n) \cong F^\#(V).
\end{equation}
Since exactness in $\bsp_{\ev}$ is defined `pointwise,' this implies that $\bsg^{n,V}$ is projective and $\bss_V^n$ is injective in $\bsp_n$. In fact, if $V = k^{n|n}$, then $\bsg^{n,V} \oplus (\Pi \circ \bsg^{n,V})$ is a projective generator and $\bss_V^n \oplus (\Pi \circ \bss_V^n)$ is an injective generator for $\bsp_n$ \cite[Theorem 3.1.1]{Drupieski:2016}, so $\bsp_n$ has enough projectives and enough injectives.

%In the special case that $F = \bsg^{n,W}$ for some $W \in \bsv$, the first isomorphism in \eqref{eq:Yonedalemma} becomes
%	\[
%	\Hom_{\bsp_n}(\bsg^{n,V},\bsg^{n,W}) \cong \bsg^{n,W}(V) = \bsg^n \Hom_k(W,V).
%	\]
%Given $\phi \in \bsg^n \Hom_k(W,V)$, the corresponding homomorphism $\eta_\phi: \bsg^{n,V} \to \bsg^{n,W}$ is defined as follows: Given $U \in \bsv$ and $\alpha \in \bsg^{n,V}(U) = \bsg^n \Hom_k(V,U)$, one has $\eta_\phi(\alpha) = (-1)^{\ol{\alpha} \cdot \ol{\phi}}  \cdot \alpha \circ \phi$, where we consider $\alpha$ and $\phi$ as morphisms in the category $\bsg^n \bsv$.

In the special case that $F = \bss_W^n$, the second isomorphism in \eqref{eq:Yonedalemma} becomes
	\begin{equation} \label{eq:SYoneda}
	\Hom_{\bsp}(\bss_W^n,\bss_V^n) \cong \bsg^{n,W}(V) = \bsg^n \Hom_k(W,V).
	\end{equation}
Observe that $\Hom_k(W,V)$ maps into $\Hom_k(W \otimes U, V \otimes U)$ via $\psi \mapsto \psi \otimes 1_U$, where $1_U$ denotes the identity map on $U$. This induces an even linear map $\bsg^n \Hom_k(W,V) \to \bsg^n \Hom_k(W \otimes U, V \otimes U)$, which we also denote by $\phi \mapsto \phi \otimes 1_U$. Then the homomorphism $\eta_\phi: \bss_W^n \to \bss_V^n$ corresponding to $\phi \in \bsg^n \Hom_k(W,V)$ is defined by $\eta_\phi(U) = \bss^n(\phi \otimes 1_U): \bss^n(W \otimes U) \to \bss^n(V \otimes U)$, i.e., $\eta_\phi(U)$ is the linear map that arises from $\phi \otimes 1_U$ via the functoriality of $\bss^n$.

\begin{lemma} \label{lemma:AprojectiveLinjective}
The spaces $\Hom_{\bsp}(\bsg^{n,k^{0|1}},\bsa^n)$ and $\Hom_{\bsp}(\bsl^n,\bss_{k^{0|1}}^n)$ are each one-dimen\-sional, spanned by isomorphisms of parity $\ol{n}$. Consequently, for each $V \in \bsv$, $\bsa^{n,V}$ is isomorphic to the projective functor $\bsg^{n,k^{0|1} \otimes V}$, and $\bsl_V^n$ is isomorphic to the injective functor $\bss_{k^{0|1} \otimes V}^n$.
\end{lemma}

\begin{proof}
We'll prove the claim for $\Hom_{\bsp}(\bsg^{n,k^{0|1}},\bsa^n)$; the other claim follows by duality.

By Yoneda's Lemma, $\Hom_{\bsp}(\bsg^{n,k^{0|1}},\bsa^n) \cong \bsa^n(k^{0|1})$. Let $v$ be a fixed basis vector for $k^{0|1}$. Then $\bsa^n(k^{0|1})$ is spanned by $x :=  v^{\otimes n}$. Under the Yoneda isomorphism, this vector corresponds to the homomorphism $\eta_x: \bsg^{n,k^{0|1}} \to \bsa^n$ that is defined as follows: Given $U \in \bsv$ and $\phi \in \bsg^{n,k^{0|1}}(U) = \Hom_{\bsg^n \bsv}(k^{0|1},U)$, one has $\eta_x(U)(\phi) = (-1)^{\ol{x} \cdot \ol{\phi}} \cdot \bsa^n(\phi)(x)$. This shows that $\ol{\eta_x} = n \cdot \ol{v} = \ol{n}$. Now to prove the claim we just need to show for each $U \in \bsv$ that $\eta_x(U)$ is an isomorphism. Let $u_1,\ldots,u_s$ be a homogeneous basis for $U$, and let $v^* \in (k^{0|1})^*$ be the functional that evaluates to $1$ on $v$. Then $u_1 \otimes v^*,\ldots,u_s \otimes v^*$ is a homogeneous basis for $\Hom_k(k^{0|1},U) \cong U \otimes (k^{0|1})^*$, so the set
	\[
	\set{\gamma_{a_1}(u_1 \otimes v^*) \cdots \gamma_{a_s}(u_s \otimes v^*): a_i \in \N \text{ with } a_i \leq 1 \text{ if $u_i$ is even}}.
	\]
is a basis for $\bsg^{n,k^{0|1}}(U)$, and the set
	\[
	\set{\gamma_{a_1}(u_1) \cdots \gamma_{a_s}(u_s): a_i \in \N \text{ with } a_i \leq 1 \text{ if $u_i$ is even}}.
	\]
is a basis for $\bsa^n(U)$. Now one can check that, modulo a sign that depends on the integers $a_1,\ldots,a_s$ and on the parities of the vectors $u_1,\ldots,u_s$, the basis vector $\gamma_{a_1}(u_1 \otimes v^*) \cdots \gamma_{a_s}(u_s \otimes v^*)$ is mapped by $\eta_x(U)$ to the basis vector $\gamma_{a_1}(u_1) \cdots \gamma_{a_s}(u_s)$. Thus, $\eta_x(U)$ is a superspace isomorphism.
\end{proof}

%\begin{remark}
%For $F \in \bsp$, one has $F \circ \bspi = F_{k^{0|1}}$ and $F \circ \Pi = F^{k^{0|1}}$. Then it follow for $n$ odd that $\bspi \circ \bss^n \circ \bspi \cong \bsl^n$ and $\Pi \circ \bsg^n \circ \Pi \cong \bsa^n$. Since $\bspi \cong \Pi$, the functors $\bspi$ and $\Pi$ can be interchanged in either of these isomorphisms.
%\end{remark}

\subsubsection{Cohomology of strict polynomial functors}

Let $F,G \in \bsp$. Since $\bsp_{\ev}$ contains both enough projectives and enough injectives, the cohomology groups $\Ext_{\bsp}^\bullet(F,G)$ can be defined in the usual way as the right derived functors of either $\Hom_{\bsp}(F,-): \bsp_{ev} \to \fsvec_{\ev}$ or $\Hom_{\bsp}(-,G): (\bsp_{\ev})^{\op} \to \fsvec_{\ev}$. Writing $F = \bigoplus_{n \in \N} F_n$ and $G = \bigoplus_{n \in \N} G_n$ for the decompositions of $F$ and $G$ in $\bsp$, one has $\Ext_{\bsp}^\bullet(F,G) = \prod_{n \in \N} \Ext_{\bsp_n}^\bullet(F_n,G_n)$. In particular, $\Ext_{\bsp}^\bullet(F,G) = 0$ if $F$ and $G$ are homogeneous of distinct polynomial degrees.

Given $F \in \bsp$, set $F^{\bspi} = \bspi \circ F \circ \bspi$. We refer to the operation $F \mapsto F^{\bspi}$ as \emph{conjugation by $\bspi$}. It is an exact even linear operation on $\bsp$ that sends projectives to projectives and injectives to injectives. It thus extends to an even isomorphism on cohomology groups $\Ext_{\bsp}^\bullet(F,G) \cong \Ext_{\bsp}^\bullet(F^{\bspi},G^{\bspi})$, denoted $z \mapsto z^{\bspi}$, which is compatible with Yoneda compositions of extensions in the sense that $(z \circ w)^{\bspi} = z^{\bspi} \circ w^{\bspi}$. Since $\bspi \circ \bspi = \bsi$, then $(z^{\bspi})^{\bspi} = z$.

\section{Superized Troesch complexes} \label{sec:superTroesch}

\subsection{\texorpdfstring{$p$-complexes}{p-complexes}} \label{subsec:p-complexes}

A $p$-complex (of vector superspaces, of cohomological type) is a graded super\-space $C = \bigoplus_{i \in \N} C^i$, concentrated in non-negative integer degrees, together with an even linear map $d: C \to C$ that raises $\Z$-degrees by a fixed positive integer $\alpha$ and such that $d^p = 0$.\footnote{The definition of a $p$-complex given elsewhere in the literature corresponds to the case $\alpha = 1$ of the definition stated here. We give this more general definition to avoid some re-indexing of complexes later. A $p$-complex as we have defined it decomposes as a direct sum of `ordinary' $p$-complexes of the form $C_j = \bigoplus_{i \in \N} C_j^{j+i\alpha}$ for $0 \leq j < \alpha$.} An even linear map $d$ with these properties is called a \emph{$p$-differential}. Equivalently, considering $k[d]/(d^p)$ as a graded superalgebra with $\ol{d} = \zero$ and $\deg(d)= \alpha$, a $p$-complex is a graded supermodule for $k[d]/(d^p)$ whose $\Z$-grading is concentrated in non-negative degrees.

Given a $p$-complex $C$ and an integer $1 \leq s < p$, the cohomology group $\Hs^i(C)$ is defined by
	\[
	\Hs^i(C) = \frac{\ker\left( d^s: C^i \to C^{i+s\alpha}\right)}{\im\left( d^{p-s}: C^{i-(p-s)\alpha} \to C^i \right)}. 
	\]
Given integers $1 \leq s \leq t < p$, one has $\ker(d^s) \subseteq \ker(d^t)$ and $\im(d^{p-s}) \subseteq \im(d^{p-t})$, and thus one gets a canonical morphism $i: \Hs^\bullet(C) \to \opH_{[t]}^\bullet(C)$. Following Tikaradze \cite{Tikaradze:2002}, we say that a $p$-complex $C$ is \emph{normal} if $i: \Hs^\bullet(C) \to \opH_{[t]}^\bullet(C)$ is an isomorphism for all $1 \leq s \leq t < p$. If $C$ is normal, we may simply write $\Hstarbul(C)$ rather than $\Hs^\bullet(C)$. We then say that a $p$-complex $C$ is \emph{$p$-acyclic} if $\Hstarbul(C) = 0$. We say that $C$ is a \emph{$p$-coresolution} of a superspace $V$ if $C$ is normal, $\Hstar^i(C) = 0$ for $i > 0$, and $\Hstar^0(C) \cong V$. In particular, if $C$ is $p$-acyclic, it is a $p$-coresolution of $0$.

\begin{remark} \label{remark:p-acyclic-equivalent}
By a result of Kapranov \cite{Kapranov:1996}, if $\Hs^\bullet(C) = 0$ for some $1 \leq s < p$, then $\Hs^\bullet(C) = 0$ for all $1 \leq s < p$, and hence $C$ is $p$-acyclic. This can be seen by considering $C$ as a module over $k[d]/(d^p)$, for then $\Hs^\bullet(C) = 0$ if and only if $\ker(d^s) = \im(d^{p-s})$. The latter occurs if and only if $C$ is free over $k[d]/(d^p)$, in which case $\ker(d^s) = \im(d^{p-s})$ for all $1 \leq s < p$.
\end{remark}

Given integers $1 \leq s < p$ and $0 \leq t < (p-s)\alpha$, the contracted complex $C_{[s,t]} = C_{[s,t]}^\bullet$ is the ordinary chain complex
	\[
	C_{[s,t]} : C^t \xrightarrow{d^s} C^{t+s\alpha} \xrightarrow{d^{p-s}} C^{t+p\alpha} \xrightarrow{d^s} C^{t+(p+s)\alpha} \xrightarrow{d^{p-s}} C^{t+2p \alpha} \xrightarrow{d^s}\cdots,
	\]
with $C_{[s,t]}^{2i} = C^{t+pi \alpha}$ and $C_{[s,t]}^{2i+1} = C^{t+(pi+s)\alpha}$. Then
	\begin{equation} \label{eq:Cst-to-C}
	\opH^\ell(C_{[s,t]}) =
	\begin{cases}
	\Hs^{t+pi\alpha}(C) & \text{if $\ell = 2i$ is even,} \\
	\opH_{[p-s]}^{t+(pi+s)\alpha}(C) & \text{if $\ell = 2i+1$ is odd.}
	\end{cases}
	\end{equation}
Conversely, let $\ell \in \N$, and write $\ell = \ell_0 + \ell_1 p\alpha$ for some $\ell_0,\ell_1 \in \N$ with $0 \leq \ell_0 < p\alpha$. Then
	\begin{equation} \label{eq:C-to-Cst}
	\Hs^\ell(C) =
	\begin{cases}
	\opH^{2\ell_1}(C_{[s,\ell_0]}) & \text{if $0 \leq \ell_0 < (p-s)\alpha$,} \\
	\opH^{2\ell_1 + 1}(C_{[p-s,\ell_0-(p-s)\alpha]}) & \text{if $(p-s)\alpha \leq \ell_0 < p\alpha$.}
	\end{cases}
	\end{equation}
If $C$ is a $p$-coresolution of $V$, then for all $1 \leq s < p$ and all $1 \leq t < (p-s)\alpha$, the complex $C_{[s,0]}$ is a coresolution of $V$ in the ordinary sense, and $C_{[s,t]}$ is acyclic.

\begin{lemma} \label{lemma:normal-p-complexes}
Let $C$ be a $p$-complex, considered as a graded supermodule for $A = k[d]/(d^p)$. Then $C$ is normal if and only if it is isomorphic as a graded $A$-supermodule to a direct sum of cyclic $A$-modules of the forms $A\subgrp{i}$, $k\subgrp{i}$, $\Pi(A)\subgrp{i}$, and $\Pi(k)\subgrp{i}$ for $i \in \N$. Consequently, if $C$ is normal, then $\Hstar^j(C)$ is isomorphic to the direct sum of the copies of $k\subgrp{j}$ and $\Pi(k)\subgrp{j}$ that occur in the direct sum decomposition of $C$.
\end{lemma}

\begin{proof}
By \cite[Theorem 1]{Webb:1985}, $C$ is isomorphic as a graded $A$-supermodule to a direct sum of cyclic submodules, i.e., modules of the forms $A/(d^j)\subgrp{i}$ and $\Pi(A/(d^j))\subgrp{i}$ for $1 \leq j \leq p$ and $i \in \N$. Note that $A/(d^1) \cong k$ is the trivial $A$-module and $A/(d^p) = A$. The $A$-module decomposition of $C$ defines a direct sum decomposition of $C$ as a $p$-complex, and $C$ is normal if and only if each direct summand is a normal. One immediately checks that $A/(d^j)\subgrp{i}$ and $\Pi(A/(d^j))\subgrp{i}$ are normal if and only if $j = 1$ or $j = p$, so $C$ is normal if and only if it contains only summands of these types.
\end{proof}

\subsection{Tensor products of \texorpdfstring{$p$-complexes}{p-complexes}} \label{subsec:tensor-product-p-complexes}

Let $(C_1,d_1)$ and $(C_2,d_2)$ be two $p$-complexes. Then $C = C_1 \otimes C_2$ is a $p$-complex with $p$-differential $d$ defined by
	\[
	d(x \otimes y) = d_1(x) \otimes y + x \otimes d_2(y).
	\]
If $x \in \ker(d_1)$ and $y \in \ker(d_2)$, then $x \otimes y \in \ker(d)$. If $x \in \im(d_1^{p-1})$ with say $x = d_1^{p-1}(x')$, then
	\[ \textstyle
	d^{p-1}(x' \otimes y) = \sum_{\ell = 0}^{p-1} \binom{p-1}{\ell} d_1^{p-1-\ell}(x') \otimes d_2^\ell(y) = \binom{p-1}{\ell} d_1^{p-1}(x') \otimes y = 0,
	\]
so $x \otimes y \in \im(d^{p-1})$. Similarly, if $y \in \im(d_2^{p-1})$, then $x \otimes y \in \im(d^{p-1})$. Now given cohomology classes $\alpha \in \Hone^i(C_1)$ and $\beta \in \Hone^j(C_2)$, represented by cocycles $x \in \ker(d_1)$ and $y \in \ker(d_2)$, respectively, define $\zeta(\alpha \otimes \beta) \in \Hone^{i+j}(C_1 \otimes C_2)$ to be the cohomology class of $x \otimes y$. Then $\zeta$ extends to a well-defined even linear map
	\begin{equation} \label{eq:Kunnethmap}
	\zeta: \Hone^\bullet(C_1) \otimes \Hone^\bullet(C_2) \to \Hone^\bullet(C_1 \otimes C_2)
	\end{equation}

\begin{theorem}[K\"{u}nneth Theorem for normal $p$-complexes] \label{thm:Kunneth}
Let $(C_1,d_1)$ and $(C_2,d_2)$ be two normal $p$-complexes. Then $C_1 \otimes C_2$ is a normal $p$-complex, whose $n$-th cohomology group $\Hstar^n(C_1 \otimes C_2)$ is isomorphic to the direct sum of the copies of $k\subgrp{i} \otimes k\subgrp{j}$, $\Pi(k)\subgrp{i} \otimes k\subgrp{j}$, $k\subgrp{i} \otimes \Pi(k)\subgrp{j}$, and $\Pi(k)\subgrp{i} \otimes \Pi(k)\subgrp{j}$, for $i+j=n$, that occur in the graded $k[d]/(d^p)$-supermodule decomposition of $C_1 \otimes C_2$. In particular, the map \eqref{eq:Kunnethmap} defines an isomorphism
	\[
	\zeta: \Hstarbul(C_1) \otimes \Hstarbul(C_2) \cong \Hstarbul(C_1 \otimes C_2),
	\]
which we call the K\"{u}nneth isomorphism for normal $p$-complexes.
\end{theorem}

\begin{proof}
This follows from Lemma \ref{lemma:normal-p-complexes} and the general fact for Hopf algebras that the tensor product of two free modules is again free.
\end{proof}

\subsection{Construction of Troesch \texorpdfstring{$p$-complexes}{p-complexes}}

Given $W \in \bsv$ and an even linear map $\phi: W \to W$, let $\bss(\phi): \bss_W \to \bss_W$ be the natural transformation whose degree-$n$ component is defined by
	\begin{equation} \label{eq:S(psi)}
	\bss^n(\phi)(U) = \bss^n(\gamma_n(\phi \otimes 1_U)): \bss^n(W \otimes U) \to \bss^n(W \otimes U).
	\end{equation}
Then $\bss(\phi)(U): \bss(W \otimes U) \to \bss(W \otimes U)$ is the superalgebra homomorphism that is induced via multiplicativity from the even linear map $\phi \otimes 1_U: W \otimes U \to W \otimes U$. Now for $d \in \N$, let $\phi_d \in \Hom_{\bsp}(\bss_W,\bss_W)$ be the homomorphism whose degree-$n$ component $\phi_d: \bss_W^n \to \bss_W^n$ is equal to the following component of the convolution morphism $\id \star \bss(\phi)$:
	\[
	\phi_d: \bss_W^n \xrightarrow{\Delta_{n-d,d}} \bss_W^{n-d} \otimes \bss_W^d \xrightarrow{\id \otimes \bss^d(\phi)} \bss_W^{n-d} \otimes \bss_W^d \xrightarrow{m} \bss_W^n.
	\]
Equivalently, $\phi_d: \bss_W^n \to \bss_W^n$ is the homomorphism that corresponds as in \eqref{eq:SYoneda} to the morphism
	\[
	\gamma_{d}(\phi) \cdot \gamma_{n-d}(1_W) \in \bsg^n \Hom_k(W,W).
	\]
If $n < d$, then $\phi_d: \bss_W^n \to \bss_W^n$ is the zero map, while for $n=d$ one has $\phi_n = \bss^n(\phi) : \bss_W^n \to \bss_W^n$.

Recall from Section \ref{subsubsec:gradings} that if $W$ is a graded superspace, then $\bss_W$ is endowed with a $\Z$-grading, $\bss_W = \bigoplus_{i \in \Z} \bss(W \otimes -)^i$, which makes $\bss(W \otimes U)$ a graded superalgebra for each $U \in \bsv$.

\begin{lemma}[Troesch \cite{Troesch:2005}, Touz\'{e} \cite{Touze:2012}] \label{lemma:convolution-components}
Let $W \in \bsv$ be a graded superspace, and let $\phi: W \to W$ be an even linear map that raises $\Z$-degrees by an integer $\alpha$. Then the morphisms $\phi_d: \bss_W \to \bss_W$ satisfy the following properties:
	\begin{enumerate}
	\item For each $U \in \bsv$, $\phi_d$ maps $\bss(W \otimes U)^\ell$ into $\bss(W \otimes U)^{\ell+d\alpha}$.
%	\item For each $U \in \bsv$, $\phi_1(U)$ is an algebra derivation.
	\item If $\phi^p = 0$, then $(\phi_d)^p = 0$.
	\item \label{item:phi-psi-commute} If $\psi: W \to W$ is another even linear map such that $\phi \circ \psi = \psi \circ \phi$, then for all $d,e \in \N$ one has $\phi_d \circ \psi_e = \psi_e \circ \phi_d$.
	\item \label{item:phiderivation} Let $U \in \bsv$, and let $x,y \in \bss(W \otimes U)$. Then $\phi_d(xy) = \sum_{\ell=0}^d \phi_\ell(x) \phi_{d-\ell}(y)$. In particular, $\phi_1$ is an algebra derivation.
	\item \label{item:phi-on-p-powers} Let $U \in \bsv$, let $x \in \bss(W \otimes U)$, and let $r \in \N$. Then $\phi_d(x^{p^r}) = [\phi_{d/p^r}(x)]^{p^r}$ if $p^r$ divides $d$, and $\phi_d(x^{p^r}) = 0$ otherwise.
	\end{enumerate}
\end{lemma}

\begin{proof}
The first statement is immediate from the definition of $\phi_d$, and the relations in the second and third statements hold at the level of morphisms in $\bsg^n \Hom_k(W,W)$. For the fourth and fifth statements, let $D = \id \star \bss(\phi)$ be the convolution morphism
	\[
	 D: \bss(W \otimes U) \xrightarrow{\Delta} \bss(W \otimes U) \otimes \bss(W \otimes U) \xrightarrow{\id \otimes \bss(\phi)} \bss(W \otimes U) \otimes \bss(W \otimes U) \xrightarrow{m} \bss(W \otimes U).
	\]
It follows from the definitions that $D(z) = \sum_{d \geq 0} \phi_d(z)$. For each fixed $z$ this sum is finite, because $\phi_d$ is zero on $\bss^n(W \otimes N)$ when $d > n$. Using the commutativity and cocommutativity of the Hopf superalgebra $\bss(W \otimes U)$, one can check that $D$ is a superalgebra homomorphism. Now \eqref{item:phi-on-p-powers} follows from the relation $D(x^{p^r}) = [D(x)]^{p^r}$ and the binomial theorem modulo $p$. For  \eqref{item:phiderivation}, one has
	\begin{equation}
	\sum_{d \geq 0} \phi_d(xy) = D(xy) = D(x)D(y) = \sum_{d \geq 0} \sum_{i+j=d} \phi_i(x)\phi_j(y).
	\end{equation}
Further, for $\lambda \in k$ one has $\phi_d( (\lambda \cdot x)(\lambda \cdot y)) = \lambda^{2d} \cdot \phi_d(xy)$ and $\phi_i(\lambda \cdot x) \phi_j(\lambda \cdot y) = \lambda^{i+j} \cdot \phi_i(x)\phi_j(y)$. Then equating eigenspaces (perhaps after extending scalars to a sufficiently large field), it follows for all $d \geq 0$ that $\phi_d(xy) = \sum_{i+j=d} \phi_i(x) \phi_j(y)$, as desired.
\end{proof}

Following Touz\'{e}'s notation \cite[\S9]{Touze:2012}, let $\Sha_1$ be a $p$-dimensional purely even graded superspace with basis $\sha_0,\ldots,\sha_{p-1}$ such that $\deg(\sha_i) = i$, and let $\rho: \Sha_1 \to \Sha_1$ be the linear map such that $\rho(\sha_i) = \sha_{i+1}$ if $0 \leq i \leq p-2$, and $\rho(\sha_{p-1}) = 0$. For $\ell \in \N$, let $\Sha_1^{(\ell)}$ be the $\ell$-th Frobenius twist of $\Sha_1$. By the conventions of Section \ref{subsubsec:gradings}, the Frobenius twist functor $\bsi^{(\ell)}$ multiplies $\Z$-degrees by $p^\ell$, so $\Sha_1^{(\ell)}$ is spanned by the vectors $\sha_0^{(\ell)},\ldots,\sha_{p-1}^{(\ell)}$, with $\deg(\sha_i^{(\ell)}) = p^\ell i$. Let $\rho^{(\ell)}: \Sha_1^{(\ell)} \to \Sha_1^{(\ell)}$ be the $\ell$-th Frobenius twist of $\rho$, i.e., the linear map such that $\rho^{(\ell)}(\sha_i^{(\ell)}) = \sha_{i+1}^{(\ell)}$ if $0 \leq i \leq p-2$, and $\rho^{(\ell)}(\sha_{p-1}^{(\ell)}) = 0$. Finally, for each positive integer $r$, set $\Sha_r = \Sha_1^{(0)} \otimes \Sha_1^{(1)} \otimes \cdots \otimes \Sha_1^{(r-1)}$, where by convention $\Sha_1^{(0)} = \Sha_1$. Then $\Sha_r$ is a $p^r$-dimensional purely even graded superspace that is one-dimensional in each $\Z$-degree from $0$ to $p^r-1$. Given an integer $0 \leq i < p^r$, let $i = i_0 + i_1 p + \cdots + i_{r-1} p^{r-1}$ be the base-$p$ decomposition of $i$, and set
	\begin{equation} \label{eq:shabasis}
	\sha_i = \sha_{i_0}^{(0)} \otimes \sha_{i_1}^{(1)} \otimes \cdots \otimes \sha_{i_{r-1}}^{(r-1)} \in \Sha_r.
	\end{equation}
Then $\sha_0,\sha_1,\ldots,\sha_{p^r-1}$ is a basis for $\Sha_r$, with $\deg(\sha_i) = i$. For $0 \leq s < r$, let
	\[
	\rho_s = 1_{\Sha_1^{(0)}} \otimes \cdots \otimes 1_{\Sha_1^{(s-1)}} \otimes \rho^{(s)} \otimes 1_{\Sha_1^{(s+1)}} \otimes \cdots \otimes 1_{\Sha_1^{(r-1)}},
	\]
i.e., $\rho_s: \Sha_r \to \Sha_r$ is the map that acts via $\rho^{(s)}$ on the factor $\Sha_1^{(s)}$, and acts as the identity on all other tensor factors of $\Sha_r$. In terms of the basis vectors \eqref{eq:shabasis},
	\[
	\rho_s(\sha_i) = \begin{cases} \sha_{i+p^s} & \text{if $0 \leq i_s \leq p-2$,} \\ 0 & \text{if $i_s = p-1$.} \end{cases}
	\]
Then $\rho_s$ raises $\Z$-degrees by $p^s$ and $(\rho_s)^p = 0$. If $0 \leq s,s' < r$, then $\rho_s \circ \rho_{s'} = \rho_{s'} \circ \rho_s$.

Let $n \in \N$. By Section \ref{subsubsec:gradings}, $\bss^n(\Sha_r \otimes -)$ inherits a non-negative $\Z$-grading from $\Sha_r$. Set
	\[
	\bsb_n(r) = \bsb_n(r)^\bullet = \bss^n(\Sha_r \otimes -)^\bullet,
	\]
and set $\bsb(r) = \bigoplus_{n \in \N} \bsb_n(r)$. Given integers $0 \leq s < r$ and $\ell > 0$, let
	\[
	d_\ell^s = (\rho_{r-1-s})_\ell: \bss(\Sha_r \otimes -) \to \bss(\Sha_r \otimes -).
	\]
Then:
	\begin{enumerate}
	\item \label{item:raisesZdegree} For all $\ell > 0$ and $0 \leq s < r$, $d_\ell^s$ raises $\Z$-degrees by $\ell p^{r-1-s}$.
	\item \label{item:pdifferential} For all $\ell > 0$ and $0 \leq s < r$, $(d_\ell^s)^p = 0$.
	\item \label{item:commute} For all $\ell,m > 0$ and $0 \leq s,t < r$, $d_\ell^s \circ d_m^t = d_m^t \circ d_\ell^s$.
	\end{enumerate}
Now set
	\[
	d(r) = d_1^0 + d_p^1 + \cdots + d_{p^{r-1}}^{r-1}.
	\]
Properties \eqref{item:raisesZdegree}--\eqref{item:commute} imply that $d(r)$ makes $\bsb(r)$ into a $p$-complex in which the $p$-differential raises $\Z$-degrees by $p^{r-1}$. Since $d(r)$ preserves polynomial degrees, $\bsb(r)$ is a direct sum of $p$-complexes, $\bsb(r) = \bigoplus_{n \in \N} \bsb_n(r)$. Our goal in Sections \ref{subsec:r=1} and \ref{subsec:r-greater-1} is to prove:

\begin{theorem} \label{theorem:B(r)-cohomology}
Let $r$ be a positive integer. Then $d(r)$ makes $\bsb(r)$ into a normal $p$-complex, and $\Hstarbul(\bsb(r)) \cong \bss^{(r)}$ as strict polynomial superfunctors. Specifically, $\Hstarbul(\bsb_n(r)) = 0$ if $p^r \nmid n$, and $\Hstarbul(\bsb_{p^rn}(r)) \cong \bss^{n(r)}$, with the summand $S_0^{n-\ell(r)} \otimes \Lambda_1^{\ell(r)}$ of $\bss^{n(r)}$ in cohomology degree $\ell \cdot \binom{p^r}{2}$.
\end{theorem}

\subsection{The case \texorpdfstring{$r=1$}{r=1}} \label{subsec:r=1}

In this section set $\bsb = \bsb(1) = \bss(\Sha_1 \otimes -)$, $\bsb_n^\ell = \bsb_n(1)^\ell$, and $d = d(1) = (\rho_0)_1$. Since $d$ is an algebra derivation by Lemma \ref{lemma:convolution-components}\eqref{item:phiderivation}, it follows that $\Hone^\bullet(\bsb)$ inherits a superalgebra structure via the K\"{u}nneth map \eqref{eq:Kunnethmap} and multiplication in $\bss$.

Our goal in Section \ref{subsec:r=1} is to prove the $r=1$ case of Theorem \ref{theorem:B(r)-cohomology}. First we'll show for each $U \in \bsv$ that $\bsb(U)$ is a normal $p$-complex, and that there exists an isomorphism of graded superalgebras $\eta(U): \bss^{(1)}(U) \to \Hstarbul(\bsb(U))$. Then in Section \ref{subsubsec:r=1morphisms} we'll show that this family of isomorphisms lifts to an isomorphism of strict polynomial superfunctors.

\subsubsection{Definition of the isomorphism} \label{subsubsec:define-eta}

Fix $U \in \bsv$, and fix a homogeneous basis $u_1,\ldots,u_m$ for $U$. Then $u_1^{(1)},\ldots,u_m^{(1)}$ is a homogeneous basis for $U^{(1)}$. Define a $k$-linear map $U^{(1)} \to \bsb(U) = \bss(\Sha_1 \otimes U)$ on basis elements by
	\begin{equation} \label{eq:eta-on-generators}
	u_i^{(1)} \mapsto \begin{cases}
	(\sha_0 \otimes u_i)^p & \text{if $u_i \in \Uzero$,} \\
	(\sha_0 \otimes u_i) \cdot (\sha_1 \otimes u_i) \cdots (\sha_{p-1} \otimes u_i) & \text{if $u_i \in \Uone$.}
	\end{cases}
	\end{equation}
This linear map uniquely extends to a homomorphism of graded superalgebras
	\[
	\eta(U): S(\Uzero^{(1)}) \otimes \Lambda(\Uone^{(1)}) \to \bss(\Sha_1 \otimes U) = \bsb(U),
	\]
where we consider $\Uzero^{(1)}$ to be concentrated in $\Z$-degree $0$, and we consider $\Uone^{(1)}$ to be concentrated in $\Z$-degree $\binom{p}{2}$. Since $d$ is an algebra derivation, it follows that the image of $\eta(U)$ consists of cocycles in $\bsb(U)$. Then $\eta(U)$ induces a homomorphism of graded superalgebras
	\begin{equation} \label{eq:eta(U)}
	\eta(U): \bss^{(1)}(U) = S(\Uzero^{(1)}) \otimes \Lambda(\Uone^{(1)}) \to \Hone^\bullet(\bsb(U)),
	\end{equation}
which by abuse of notation we have also denoted $\eta(U)$. Our goal is to show that this map is an isomorphism of graded superalgebras.

\subsubsection{Reduction to the one-dimensional case} \label{subsubsec:reduction-to-1-dim}

Recall that $\bss$ is an exponential superfunctor; see \cite[\S2.5]{Drupieski:2016}. Then $\bss^{(1)}$ and $\bsb$ are also exponential superfunctors, and the exponential formula for $\bsb$ takes the form
	\[
	\bsb_n^\ell(V \oplus W) \cong \bigoplus_{i+j = n} \bigoplus_{s+t = \ell} \bsb_i^s(V) \otimes \bsb_j^t(W).
	\]
Applying this to the decomposition $U = \bigoplus_{i=1}^m k \cdot u_i$, and using the fact that $d$ is a derivation, it follows that $\bsb(U)$ is isomorphic to the tensor product of complexes $\bsb(k \cdot u_1) \otimes \cdots \otimes \bsb(k \cdot u_m)$. Now to show that $\bsb(U)$ is a normal $p$-complex, and to show that \eqref{eq:eta(U)} is an isomorphism of graded superalgebras, it suffices by Theorem \ref{thm:Kunneth} and the definition of $\eta(U)$ to prove these claims in the special case when $U$ is a one-dimensional superspace. The case when $U$ is a one-dimensional purely even superspace is covered by the calculations in \cite[\S3]{Troesch:2005} (since when $U$ is purely even, the complex $\bsb(U)$ reduces to the complex originally defined by Troesch), so we may assume that $U$ is a one-dimensional purely odd superspace. Our calculations are modeled on those in \cite[\S3]{Troesch:2005}.

\subsubsection{Calculations in the one-dimensional purely odd case} \label{subsubsec:1-dim-purely-odd}

In this subsection assume that $U = k \cdot u$ is a one-dimensional purely odd superspace spanned by $u$, and set $w_i = \sha_i \otimes u \in \Sha_1 \otimes U$. The exponential isomorphism for $\bss$, applied to the decomposition $\Sha_1 \otimes U = \bigoplus_{i=0}^{p-1} k \cdot w_i$, gives
	\begin{equation} \label{eq:r=1Biso}
	\bsb_n^\ell(U) \cong \bigoplus_{\substack{a_0+a_1 + \cdots+a_{p-1} = n \\ a_0 \cdot 0 + a_1 \cdot 1 + \cdots + a_{p-1} \cdot (p-1) = \ell}} \Lambda^{a_0}(w_0) \otimes \Lambda^{a_1}(w_1) \otimes \cdots \otimes \Lambda^{a_{p-1}}(u_{p-1}).
	\end{equation}
In particular, $\bsb_n(U) = 0$ for $n > p$, and $\bsb_p(U)$ is equal to the single summand
	\[
	\Lambda^1(w_0) \otimes \Lambda^1(w_1) \otimes \cdots \otimes \Lambda^1(w_{p-1}),
	\]
located in $\Z$-degree $\ell = \binom{p}{2}$. This implies that $\bsb_n(U)$ is a normal $p$-complex for $n \geq p$, and that \eqref{eq:eta(U)} is an isomorphism when restricted to the components of polynomial degree $n \geq p$. Then to finish the calculations for the one-dimensional purely odd case, it suffices by Remark \ref{remark:p-acyclic-equivalent} to show for $1 \leq n < p$ that $\Hone^\bullet(\bsb_n(U)) = 0$. So assume for the rest of this subsection that $1 \leq n < p$. %Our calculations are modeled on those in \cite[\S3]{Troesch:2005}.

Let $C = \Lambda(w_0,w_1,\ldots,w_{p-1})$, identified with $\bsb(U)$ via \eqref{eq:r=1Biso}, and let $d_C : C^\ell \to C^{\ell+1}$ be the corresponding differential. Then $d_C$ is the unique algebra derivation such that $d_C(w_i) = w_{i+1}$ if $0 \leq i \leq p-2$, and $d_C(w_{p-1}) = 0$. Set $D_n = k[x_1]/(x_1^p) \otimes \cdots \otimes k[x_n]/(x_n^p)$. (We do not define a superspace structure on $D_n$.) Dropping the tensor product symbols, the set of monomials $\{x_1^{b_1} \cdots x_n^{b_n} : 0 \leq b_j < p \}$ is a basis for $D_n$. Let $D_n^\ell$ be the subspace of $D_n$ spanned by the monomials $x_1^{b_1} \cdots x_n^{b_n}$ such that $\sum_{j=1}^n b_j = \ell$, and define $d_D: D_n^\ell \to D_n^{\ell+1}$ by
	\[
	d_D(x_1^{b_1} \cdots x_n^{b_n}) = \sum_{i=1}^n x_1^{b_1} \cdots x_i^{b_i+1} \cdots x_n^{b_n}.
	\]
Then $(D_n^\bullet,d_D) \cong (D_1^\bullet,d_D)^{\otimes n}$. It is immediately verified that $D_1^\bullet$ is $p$-acyclic (the differential defines isomorphisms $D_1^\ell \cong D_1^{\ell+1}$ for $0 \leq \ell \leq p-2$), so then $D_n^\bullet$ is $p$-acyclic by Theorem \ref{thm:Kunneth}.

Define a $k$-linear map $\varphi: D_n^\ell \to C_n^\ell$ by
	\[
	\varphi(x_1^{b_1} x_2^{b_2} \cdots x_n^{b_n}) = w_{b_1} w_{b_2} \cdots w_{b_n},
	\]
and define a $k$-linear map $\psi: C_n^\ell \to D_n^\ell$ by
	\[
	\psi(w_{i_1} w_{i_2} \cdots w_{i_n}) = x_1^{i_1} x_2^{i_2} \cdots x_n^{i_n} \quad \text{if } i_1 < i_2 < \cdots < i_n.
	\]
We'll show that $\Hone^\bullet(\bsb_n(U)) = \Hone^\bullet(C_n) = 0$ by relating the $p$-complexes $C_n^\bullet$ and $D_n^\bullet$ via the maps $\varphi$ and $\psi$. First, the following properties are straightforward to verify:
	\begin{enumerate}
	\item \label{item:varphi-chian-map} $d_C \circ \varphi = \varphi \circ d_D$, so $\varphi$ is a homomorphism of $p$-complexes.
	\item \label{item:Troesch-3.3.3} $\varphi \circ \psi = \id_C: C \to C$, so $\varphi$ is a surjection.
	\item \label{item:Troesch-3.3.4} $\psi \circ \varphi \circ d_D \circ \psi = \psi \circ d_C$.
	\end{enumerate}
Then by \eqref{item:Troesch-3.3.3} and \eqref{item:Troesch-3.3.4},
	\begin{enumerate}[resume]
	\item \label{item:Troesch-3.3.6} $d_C = \varphi \circ \psi \circ d_C = \varphi \circ \psi \circ \varphi \circ d_D \circ \psi = \varphi \circ d_D \circ \psi$.
	\end{enumerate}
Next, the symmetric group $\fS_n$ acts on the basis monomials for $D_n$ by the formula
	\[
	\sigma \cdot x_1^{b_1} \cdots x_n^{b_n} = (-1)^{\sigma} x_1^{b_{\sigma^{-1}(1)}} \cdots x_n^{b_{\sigma^{-1}(n)}},
	\]
where $(-1)^{\sigma}$ denotes the sign of the permutation $\sigma$. Define a $k$-linear map $s: D_n^\ell \to D_n^\ell$ by
	\[
	s(x_1^{b_1} \cdots x_n^{b_n}) = \frac{1}{n!} \sum_{\sigma \in \fS_n} \sigma \cdot x_1^{b_1} \cdots x_n^{b_n}.
	\]
This formula makes sense by the assumption that $1 \leq n < p$. Then
	\begin{enumerate}[resume]
	\item \label{item:Troesch-3.5.2} $\varphi(\sigma \cdot y) = \varphi(y)$ for all $y \in D_n^\ell$, and hence $\varphi = \varphi \circ s$.
	\item \label{item:Troesch-3.5.5} $d_D(\sigma \cdot y) = \sigma \cdot d_D(y)$ for all $y \in D_n^\ell$, and hence $d_D \circ s = s \circ d_D$.
	\end{enumerate}

The symmetric group action permutes the elements of the set $S := \{ \pm x_1^{b_1} \cdots x_n^{b_n} : 0 \leq b_j < p \}$. If $b_1,\ldots,b_n$ are all distinct, then $x_1^{b_1} \cdots x_n^{b_n}$ and $-x_1^{b_1} \cdots x_n^{b_n}$ are in different $\fS_n$-oribts, while if two of $b_1,\ldots,b_n$ are equal, then $x_1^{b_1} \cdots x_n^{b_n}$ and $-x_1^{b_1} \cdots x_n^{b_n}$ are in the same $\fS_n$-orbit. Let
	\[
	\calO = \{ x_1^{b_1} \cdots x_n^{b_n} : 0 \leq b_1 \leq b_2 \leq \cdots \leq b_n < p\},
	\]
let $\calO_1,\ldots,\calO_t$ be the distinct $\fS_n$-orbits of the elements of $\calO$, and let
	\[
	\calO' = \{ x_1^{b_1} \cdots x_n^{b_n} : 0 \leq b_1 < b_2 < \cdots < b_n < p\}.
	\]
Up to a factor of $\pm 1$, every basis monomial for $D_n$ is conjugate to a unique element of $\calO$.

Let $y \in D_n$. Then $y = \sum_{i=1}^t \sum_{y_{ij} \in \calO_i} \lambda_{ij} \cdot y_{ij}$ for some uniquely determined scalars $\lambda_{ij} \in k$. Set $\lambda_i = \sum_{y_{ij} \in \calO_i} \lambda_{ij}$. Then $\varphi(y) = \sum_{i=1}^t \lambda_i \cdot \varphi(y_i)$, where $y_i$ is the unique element of $\calO \cap \calO_i$. Since the set $\set{ \varphi(z) : z \in \calO'}$ is linearly independent in $C_n$, and since $\varphi(z) = 0$ for all $z \in \calO \backslash \calO'$, then $\varphi(y) = 0$ if and only if $\lambda_i = 0$ for all $y_i \in \calO'$. Next, the linear map $s$ is constant on $\fS_n$-orbits, and $s(y_i) = 0$ if $y_i \in \calO \backslash \calO'$, so
	\[
	s(y) = \sum_{i=1}^t \sum_{y_{ij} \in \calO_i} \lambda_{ij} \cdot s(y_{ij}) = \sum_{i=1}^t \lambda_i \cdot s(y_i) = \sum_{\substack{1 \leq i \leq t \\ y_i \in \calO'}} \lambda_i \cdot s(y_i).
	\]
The elements $\set{s(y_i): y_i \in \calO'}$ are linearly independent, so this implies that $s(y) = 0$ if and only if $\lambda_i = 0$ for all $y_i \in \calO'$. Thus we conclude that
	\begin{enumerate}[resume]
	\item \label{item:Troesch-3.5.4} $\varphi(y) = 0$ if and only if $s(y) = 0$.
	\end{enumerate}
Combined with \eqref{item:Troesch-3.5.5}, this implies:
	\begin{enumerate}[resume]
	\item \label{item:Troesch-3.5.6} If $(\varphi \circ d_D)(y) = 0$, then $(d_D \circ s)(y) = 0$.
	\end{enumerate}

Now we apply the preceding observations to show that $\Hone^\bullet(C_n) = 0$. Let $x \in C_n^\ell$, and suppose that $d_C(x) = 0$. Set $y = (s \circ \psi)(x)$. Then by \eqref{item:Troesch-3.5.2} and \eqref{item:Troesch-3.3.3},
	\[
	\varphi(y) = (\varphi \circ s \circ \psi)(x) = ((\varphi \circ s) \circ \psi) \psi(x) = (\varphi \circ \psi)(x) = x.
	\]
Next, since $d_C(x) = 0$, then $(\varphi \circ d_D \circ \psi)(x) = 0$ by \eqref{item:Troesch-3.3.6}, and hence $d_D(y) = (d_D \circ s \circ \psi)(x) = 0$ by \eqref{item:Troesch-3.5.6}. The complex $D_n^\bullet$ is $p$-acyclic, so there exists $y' \in D_n^{\ell - p + 1}$ such that $(d_D)^{p-1}(y') = y$. Set $x' = \varphi(y') \in C_n^{\ell - p + 1}$. Then by \eqref{item:varphi-chian-map},
	\[
	(d_C)^{p-1}(x') = (d_C)^{p-1} \circ \varphi(y') = \varphi \circ (d_D)^{p-1}(y') = \varphi(y) = x.
	\]
So $x$ is a $(p-1)$-coboundary in $C_n^\ell$. Since $\ell$ was arbitrary, this shows that $\Hone^\bullet(C_n) = 0$.

\subsubsection{Compatibility with morphisms} \label{subsubsec:r=1morphisms}

In Sections \ref{subsubsec:reduction-to-1-dim} and \ref{subsubsec:1-dim-purely-odd} we showed for each $U \in \bsv$ that $\bsb(U)$ is a normal $p$-complex, and we showed that the map $\eta(U): \bss^{(1)}(U) \to \Hstarbul(\bsb(U))$, which was defined in Section \ref{subsubsec:define-eta} in terms of a fixed choice $u_1,\ldots,u_m$ of homogeneous basis for $U$, is an isomorphism of graded superalgebras. In this section we'll show that $\eta(U)$ lifts to an isomorphism of strict polynomial superfunctors $\eta: \bss^{(1)} \to \Hstarbul(\bsb)$. To do this, we must show for each $U,V \in \bsv$, each $\phi \in \bsg \Hom_k(U,V)$, and each $z \in \bss^{(1)}(U)$ that
	\begin{equation} \label{eq:eta-naturality}
	[\eta(V) \circ \bss^{(1)}(\phi)](z) = [\Hstarbul(\bsb)(\phi) \circ \eta(U)](z).
	\end{equation}
Here $\eta(V)$ is defined as in Section \ref{subsubsec:define-eta} in terms of a fixed choice $v_1,\ldots,v_n$ of homogeneous basis for $V$. Since the multiplication morphisms for $\bss^{(1)}$ and $\Hstarbul(\bsb)$ are natural transformations, it suffices to verify \eqref{eq:eta-naturality} as $z$ ranges over a set of generators for the algebra $\bss^{(1)}(U) = S(\Uzero^{(1)}) \otimes \Lambda(\Uone^{(1)})$, i.e., as $z$ ranges over the homogeneous basis $u_1^{(1)},\ldots,u_m^{(1)}$ for $U^{(1)}$. By linearity, it also suffices to verify \eqref{eq:eta-naturality} as $\phi$ ranges over a basis for $\bsg^p \Hom_k(U,V)$. %We will find it convenient not to consider $\Hbul(\bsb)(\phi)$ directly, but to consider the function $\bsb(\phi)$ that induces $\Hbul(\bsb)(\phi)$.

For each $1 \leq i \leq n$ and $1 \leq j \leq m$, let $e_{i,j} \in \Hom_k(U,V)$ be the homogeneous linear map satisfying $e_{i,j}(u_\ell) = \delta_{j,\ell} \cdot v_i$. Then the $e_{i,j}$ form a homogeneous basis for $\Hom_k(U,V)$, and $\bsg^p \Hom_k(U,V)$ admits a basis consisting of all monomials $\prod_{i,j} \gamma_{a_{i,j}}(e_{i,j})$ (the products taken, say, in the lexicographic order) with $a_{i,j} \in \N$, $\sum_{i,j} a_{i,j} = p$, and $a_{i,j} \leq 1$ if $e_{i,j}$ is of odd superdegree. So assume that $\phi = \prod_{i,j} \gamma_{a_{i,j}}(e_{i,j})$ is one of these basis monomials. Let $H$ be the Young subgroup $\prod_{i,j} \fS_{a_{i,j}}$ of $\fS_p$, and let $J \subset \fS_p$ be a set of right coset representatives for $H$. Then as in \cite[IV.5.3]{Bourbaki:2003}, we can write $\phi = \prod_{i,j} \gamma_{a_{i,j}}(e_{i,j})$ in the form
\begin{equation} \label{eq:phiorbitsum} \textstyle
\phi = \sum_{\sigma \in J} \left( \bigotimes_{i,j} [(e_{i,j})^{\otimes a_{i,j}}] \right) \cdot \sigma,
\end{equation}
where the factors in the tensor product are taken in the lexicographic order.

First consider the expression $[\eta(V) \circ \bss^{(1)}(\phi)](u_\ell^{(1)})$. By the definition \eqref{eq:IrVW}, one has $\bss^{(1)}(\phi) = 0$ unless $\phi = \gamma_p(e_{i,j})$ for some indices $i$ and $j$ such that $e_{i,j}$ is of even superdegree. In this case,
	\begin{align*}
	[\eta(V) \circ \bss^{(1)}(\gamma_p(e_{i,j}))](u_\ell^{(1)}) &= [\eta(V) \circ e_{i,j}^{(1)}](u_\ell^{(1)}) \\
	&= \eta(V) ( \delta_{j,\ell} \cdot v_i^{(1)} ) \\
	&= \begin{cases}
	\delta_{j,\ell} \cdot (\sha_0 \otimes v_i)^p & \text{ if  $\ol{u_\ell} = \ol{v_i} = \zero$,} \\
	\delta_{j,\ell} \cdot (\sha_0 \otimes v_i) \cdots (\sha_{p-1} \otimes v_i) & \text{ if $\ol{u_\ell} = \ol{v_i} = \one$.}
	\end{cases}
	\end{align*}
Next consider the expression
	\[
	[\Hstarbul(\bsb)(\phi) \circ \eta(U)](u_\ell^{(1)}) = \begin{cases}
	\Hstarbul(\bsb)(\phi)\left( (\sha_0 \otimes u_\ell)^p \right) & \text{if $\ol{u_\ell} = \zero$,} \\
	\Hstarbul(\bsb)(\phi)\left( (\sha_0 \otimes u_\ell) \cdots (\sha_{p-1} \otimes u_\ell) \right) & \text{if $\ol{u_\ell} = \one$.}
	\end{cases}
	\]
From \eqref{eq:phiorbitsum} it follows that $[\Hstarbul(\bsb)(\phi) \circ \eta(U)](u_\ell^{(1)}) = 0$ unless $a_{i,j} = 0$ for all $j \neq \ell$. So suppose $a_{i,j} = 0$ for all $j \neq \ell$, in which case
	\[ \textstyle
	\phi = \sum_{\sigma \in J}\left( (e_{1,\ell}^{\otimes a_{1,\ell}}) \otimes \cdots \otimes (e_{n,\ell}^{\otimes a_{n,\ell}}) \right) \cdot \sigma.
	\]
Applying the exponential isomorphism for $\bsb$ to the decomposition $V = \bigoplus_{i=1}^n k \cdot v_i$, one gets $\bsb(V) \cong \bsb(k \cdot v_1) \otimes \cdots \otimes \bsb(k \cdot v_n)$. Then $[\bsb(\phi) \circ \eta(U)](u_\ell^{(1)}) \in \bsb_{a_{1,\ell}}(k \cdot v_1) \otimes \cdots \otimes \bsb_{a_{n,\ell}}(k \cdot v_n)$. The differential $d$ commutes with the action of $\bsb(\phi)$ (roughly, because $\rho$ acts on the first factor of $\Sha_1 \otimes U$ but $\phi$ acts on the second), so $[\bsb(\phi) \circ \eta(U)](u_\ell^{(1)})$ is a cocycle in $\bsb(V)$. By Theorem \ref{thm:Kunneth}, this implies that $[\bsb(\phi) \circ \eta(U)](u_\ell^{(1)})$ defines a cohomology class in $\Hstarbul(\bsb_{a_{1,\ell}}(k \cdot v_1)) \otimes \cdots \otimes \Hstarbul(\bsb_{a_{n,\ell}}(k \cdot v_n))$. Suppose for the moment that $0 < a_{i,\ell} < p$ for some $i$. From the analysis of the one-dimensional purely even case in \cite[\S 3]{Troesch:2005}, and the analysis of the one-dimensional purely odd case in Section \ref{subsubsec:1-dim-purely-odd}, we know that $\Hstarbul(\bsb_a(k \cdot v_i)) = 0$ if $0 < a < p$. Thus, $[\bsb(\phi) \circ \eta(U)](u_\ell^{(1)})$ is a $(p-1)$-coboundary in $\bsb(V)$, and hence $[\Hstarbul(\bsb)(\phi) \circ \eta(U)](u_\ell^{(1)}) = 0$ if $0 < a_{i,\ell} < p$ for any $i$. On the other hand, suppose $\phi = e_{i,j}^{\otimes p} = \gamma_p(e_{i,j})$ for some indices $i$ and $j$, in which case $e_{i,j}$ is necessarily even. Then
	\begin{align*}
	[\bsb(\phi) \circ \eta(U)](u_\ell^{(1)}) &=
	\begin{cases}
	\bsb(e_{i,j}^{\otimes p})\left( (\sha_0 \otimes u_\ell)^p \right) & \text{if $\ol{u_\ell} = \zero$,} \\
	\bsb(e_{i,j}^{\otimes p})\left( (\sha_0 \otimes u_\ell) \cdots (\sha_{p-1} \otimes u_\ell) \right) & \text{if $\ol{u_\ell} = \one$.}
	\end{cases} \\
	&= \begin{cases}
	\delta_{j,\ell} \cdot (\sha_0 \otimes v_i)^p & \text{if $\ol{u_\ell} = \zero$,} \\
	\delta_{j,\ell} \cdot (\sha_0 \otimes v_i) \cdots (\sha_{p-1} \otimes v_i) & \text{if $\ol{u_\ell} = \one$.}
	\end{cases}
	\end{align*}
Combining these observations, \eqref{eq:eta-naturality} then follows.

\subsection{The case \texorpdfstring{$r > 1$}{r>1}} \label{subsec:r-greater-1}

Our goal in Section \ref{subsec:r-greater-1} is to prove the $r > 1$ case of Theorem \ref{theorem:B(r)-cohomology}. The proof proceeds in several steps, modeled on the arguments in \cite[\S4]{Troesch:2005}: First, we explain how $\bsb(r)$ admits the structure of a $p$-bicomplex $\bsb^{\bullet,\bullet}$ with horizontal differential $d'$ and vertical differential $d''$. Next, there is no general theory of spectral sequences for $p$-bicomplexes that we know of, but proceeding as if there were, we calculate
	\[
	\text{``$E_2$''} := \Hstarbul(\Hstarbul(\bsb^{\bullet,\bullet},d'),d''),
	\]
i.e., we calculate the cohomology of the $p$-bicomplex $\bsb^{\bullet,\bullet}$ first along rows and then along columns. In particular, we show that $(\bsb^{\bullet,\bullet},d')$ and $(\Hstarbul(\bsb^{\bullet,\bullet},d'),d'')$ are each normal $p$-complexes. Arguing by induction on $r$, we show that ``$E_2$'' is isomorphic as a strict polynomial superfunctor to $\bss^{(r)}$, and that the isomorphism is induced by a family of superalgebra homomorphisms $\eta_r(U): \bss^{(r)}(U) \to \bsb(r)(U)$. Note, however, that $d(r)$ is not an algebra derivation when $r > 1$, so $\Hstarbul(\bsb(r))$ does not automatically inherit the structure of a graded superalgebra. Finally, we look at the contractions $\bsbst$ of the $p$-complex $\bsb(r)$. The row-wise filtration of $\bsb^{\bullet,\bullet}$ induces filtrations on each $\bsbst$, which in turn give rise to (genuine) spectral sequences. Analyzing these, we determine for all $1 \leq s < p$ that $\Hs^n(\bsb(r))$ is isomorphic to the component of total cohomological degree $n$ in  ``$E_2$''. Since this holds independent of $s$, this will imply that $\bsb(r)$ is normal.

%For the rest of Section \ref{subsec:r-greater-1}, let $r > 1$ be fixed.

\subsubsection{Definition of the isomorphism} \label{subsubsec:define-etar}

Fix $U \in \bsv$, and fix a homogeneous basis $u_1,\ldots,u_m$ for $U$. Then $u_1^{(r)},\ldots,u_m^{(r)}$ is a homogeneous basis for $U^{(r)}$. Define a $k$-linear map $U^{(r)} \to \bss(\Sha_r \otimes U)$ on basis elements by
	\begin{equation} \label{eq:etar-on-generators}
	u_i^{(r)} \mapsto \begin{cases}
	(\sha_0 \otimes u_i)^{p^r} & \text{if $u_i \in \Uzero$,} \\
	(\sha_0 \otimes u_i) \cdot (\sha_1 \otimes u_i) \cdots (\sha_{p^r-1} \otimes u_i) & \text{if $u_i \in \Uone$.}
	\end{cases}
	\end{equation}
This linear map uniquely extends to a homomorphism of graded superalgebras
	\[
	\eta_r(U): \bss^{(r)}(U) = S(\Uzero^{(r)}) \otimes \Lambda(\Uone^{(r)}) \to \bss(\Sha_r \otimes U) = \bsb(r)(U),
	\]
where we consider $\Uzero^{(r)}$ to be concentrated in $\Z$-degree $0$, and we consider $\Uone^{(r)}$ to be concentrated in $\Z$-degree $\binom{p^r}{2}$. Using parts \eqref{item:phiderivation} and \eqref{item:phi-on-p-powers} of Lemma \ref{lemma:convolution-components}, it is straightforward to check that the image of $\eta_r(U)$ consists of cocycles in $\bsb(r)(U)$. Then $\eta_r(U)$ induces a map
	\[
	\eta_r(U): \bss^{(r)}(U) = S(\Uzero^{(r)}) \otimes \Lambda(\Uone^{(r)}) \to \Hs^\bullet(\bsb(r)(U)),
	\]
which by abuse of notation we also denote $\eta_r(U)$.

\subsubsection{\texorpdfstring{$\bsb(r)$ as a $p$-bicomplex}{B(r) as a p-bicomplex}}

Recall that a basis for $\Sha_r$ is given by the vectors
	\[
	\sha_i = \sha_{i_0}^{(0)} \otimes \sha_{i_1}^{(1)} \otimes \cdots \otimes \sha_{i_{r-1}}^{(r-1)} \quad (0 \leq i < p^r),
	\]
where $i = i_0 + i_1 p + \cdots + i_{r-1}p^{r-1}$ is the base-$p$ decomposition of $p$. Set $\deg'(\sha_i) = i_1 p + \cdots i_{r-1} p^{r-1}$, and set $\deg''(\sha_i) = i_0$. Then $\deg'$ and $\deg''$ each define (non-negative) $\Z$-gradings on $\Sha_r$, which in turn induce $\Z$-gradings on $\bsb(r) = \bss(\Sha_r \otimes -)$. For each $U \in \bsv$, set
	\[
	\bsb_n^{i,j}(U) = \set{ z \in \bsb_n(r)(U) : \deg'(z) = i \text{ and } \deg''(z) = j},
	\]
and set $\bsb^{i,j} = \bigoplus_{n \in \N} \bsb_n^{i,j}$. Note that $\bsb^{i,j} = 0$ if $p \nmid i$.

Next recall that $d(r) = d_1^0 + d_p^1 + \cdots + d_{p^{r-1}}^{r-1}$. Set $d' = d_1^0+ \cdots + d_{p^{r-2}}^{r-2}$ and set $d'' = d_{p^{r-1}}^{r-1}$. Then $d'$ and $d''$ define even linear maps $d': \bsb^{i,j} \to \bsb^{i+p^{r-1},j}$ and $d'': \bsb^{i,j} \to \bsb^{i,j+p^{r-1}}$. Lemma \ref{lemma:convolution-components} implies that $d' \circ d'' = d'' \circ d'$ and $(d')^p = 0 = (d'')^p$. Then $d'$ and $d''$ endow $\bsb^{\bullet,\bullet}$ with the structure of a $p$-bicomplex such that the total $p$-complex $(\Tot(\bsb^{\bullet,\bullet}),d'+d'')$ is simply $(\bsb(r),d(r))$.

\subsubsection{Calculating \texorpdfstring{``$E_2$''}{"E2"}} \label{subsubsec:E2-page}

In this subsection we calculate $\text{``$E_2$''} := \Hstarbul( \Hstarbul (\bsb^{\bullet,\bullet},d'),d'')$, i.e., we calculate the cohomology of the $p$-bicomplex $\bsb^{\bullet,\bullet}$ first along rows and then along columns. Along the way we show that $(\bsb^{\bullet,\bullet},d')$ and $(\Hstarbul(\bsb^{\bullet,\bullet},d'),d'')$ are each normal $p$-complexes.

Recall that $\Sha_r = \Sha_1^{(0)} \otimes (\Sha_1^{(1)} \otimes \cdots \otimes \Sha_1^{(r-1)}) = \Sha_1 \otimes \Sha_{r-1}^{(1)} \cong \Sha_{r-1}^{(1)} \otimes \Sha_1$. Then for each $U \in \bsv$, one has $\bsb(r)(U) = \bss(\Sha_r \otimes U) \cong \bss(\Sha_{r-1}^{(1)} \otimes (\Sha_1 \otimes U))$. Ignoring the Frobenius twist on $\Sha_{r-1}$, this is equal to $\bsb(r-1)( \Sha_1 \otimes U)$. More precisely,
	\begin{equation} \label{eq:B(r)-restriction-identification}
	\bsb^{pi,\bullet}(U) \cong \bsb(r-1)^i(\Sha_1 \otimes U),
	\end{equation}
and via this identification, the horizontal differential $d': \bsb^{pi,\bullet}(U) \to \bsb^{pi+p^{r-1},\bullet}(U)$ identifies with the differential $d(r-1): \bsb(r-1)^i(\Sha_1 \otimes U) \to \bsb(r-1)^{i+p^{r-2}}(\Sha_1 \otimes U)$ for $\bsb(r-1)$, evaluated on $\Sha_1 \otimes U$. Then, modulo the rescaling of the horizontal $\Z$-degree by a factor of $p$, one has $(\bsb^{\bullet,\bullet},d') \cong (\bsb(r-1),d(r-1)) \circ (\Sha_1 \otimes -)$ as $p$-complexes of strict polynomial super\-functors. Since pre-composition with the functor $\Sha_1 \otimes -$ is an exact operation, this implies by induction on $r$ that $(\bsb^{\bullet,\bullet},d')$ is a normal $p$-complex, and
	\begin{equation} \label{eq:cohomology-along-rows}
	\Hstarbul(\bsb^{\bullet,\bullet},d') \cong \bss^{(r-1)} \circ (\Sha_1 \otimes -)
	\end{equation}
as strict polynomial superfunctors. Specifically, \eqref{eq:cohomology-along-rows} is induced by the family of superalgebra homomorphisms
	\[
	\eta_{r-1}(\Sha_1 \otimes -): \bss^{(r-1)}(\Sha_1 \otimes -) \to \bss(\Sha_{r-1}^{(1)} \otimes (\Sha_1 \otimes -)) \cong \bss(\Sha_r \otimes -)
	\]
that are defined on generators as in \eqref{eq:etar-on-generators}, except that the basis $\sha_0, \sha_1, \ldots, \sha_{p^{r-1}-1}$ for $\Sha_{r-1}$ is replaced by $\sha_0^{(1)},\sha_1^{(1)},\ldots,\sha_{p^{r-1}-1}^{(1)}$. Taking $\{\sha_i \otimes u_j: 0 \leq i < p,\, 1 \leq j \leq m \}$ as our homogeneous basis for $\Sha_1 \otimes U$, the map $\eta_{r-1}(\Sha_1 \otimes U)$ is defined on generators by
	\[
	(\sha_i \otimes u_j)^{(r-1)} \mapsto \begin{cases}
	(\sha_i \otimes \sha_0^{(1)} \otimes u_j)^{p^{r-1}} & \text{if $u_j \in \Uzero$,} \\
	(\sha_i \otimes \sha_0^{(1)} \otimes u_j) \cdot (\sha_i \otimes \sha_1^{(1)} \otimes u_j) \cdots (\sha_i \otimes \sha_{p^{r-1}-1}^{(1)} \otimes u_j) & \text{if $u_j \in \Uone$.}
	\end{cases}
	\]
Note that if $0 \leq i < p$ and $0 \leq \ell < p^{r-1}$, then $\sha_i \otimes \sha_\ell^{(1)} = \sha_{i+p \ell} \in \Sha_r$.

Let $u_{i,j}$ denote the image of $(\sha_i \otimes u_j)^{(r-1)}$ under $\eta_{r-1}(\Sha_1 \otimes U)$. Then $\Hstarbul(\bsb^{\bullet,\bullet}(U),d')$ is spanned by the cohomology classes of monomials of the form $u_{i_1,j_1} \cdots u_{i_n,j_n}$. Since $\deg'(u_{i,j}) \in p \cdot \binom{p^{r-1}}{2} \Z$ and $\deg''(u_{i,j}) \in p^{r-1} \Z$, this implies that
	\begin{equation} \label{eq:E1-page-nonzero} \textstyle
	\Hstar^a(\bsb^{\bullet,b},d') \neq 0 \qquad \text{only if} \qquad a \in p \cdot \binom{p^{r-1}}{2} \Z \quad \text{and} \quad b \in p^{r-1} \Z.
	\end{equation}
Next, as a superspace one has $\bss^{(r-1)}(\Sha_1 \otimes U) = \bss(\Sha_1^{(r-1)} \otimes U^{(r-1)})$. Ignoring the Frobenius twist on $\Sha_1$, this identifies with $\bsb(1)(U^{(r-1)})$. In terms of \eqref{eq:cohomology-along-rows}, the identification becomes
	\begin{equation} \label{eq:cohomology-along-rows-B(1)}
	\Hstarbul(\bsb^{\bullet,p^{r-1}\ell},d') \cong \bsb(1)^\ell \circ \bsi^{(r-1)}.
	\end{equation}
Via this identification, the cohomology class of $u_{i,j}$ corresponds to $\sha_i \otimes u_j^{(r-1)} \in \bss^1(\Sha_1 \otimes U^{(r-1)})$.

Now consider the induced action of the vertical differential $d''$ on $\Hstarbul(\bsb^{\bullet,\bullet},d')$. On terms of polynomial degree $p^{r-1}n$ in $\bsb(r)(U) = \bss(\Sha_r \otimes U)$, $d'' = d_{p^{r-1}}^{r-1}$ is the morphism that corresponds via \eqref{eq:SYoneda} to the morphism
	\[
	\gamma_{p^{r-1}}(\rho_0) \cdot \gamma_{p^{r-1}(n-1)}(1) \in \bsg^{p^{r-1}n} \Hom_k(\Sha_r,\Sha_r).
	\]
Making the identification \eqref{eq:B(r)-restriction-identification}, this is the morphism $\bsb(r-1)(\Sha_1 \otimes U) \to \bsb(r-1)(\Sha_1 \otimes U)$ that arises via functoriality from the morphism
	\begin{equation} \label{eq:B(r-1)-morphism}
	\gamma_{p^{r-1}}(\rho \otimes 1_U) \cdot \gamma_{p^{r-1}(n-1)}(1) \in \bsg^{p^{r-1}n} \Hom_k(\Sha_1 \otimes U,\Sha_1 \otimes U).
	\end{equation}
Passing through the isomorphism \eqref{eq:cohomology-along-rows}, the functorial action of $\bss^{(r-1)} = \bss \circ \bsi^{(r-1)}$ on \eqref{eq:B(r-1)-morphism} is
	\begin{align*}
	\bss^{(r-1)}\left( \gamma_{p^{r-1}}(\rho \otimes 1_U) \cdot \gamma_{p^{r-1}(n-1)}(1) \right) &= \bss \left( \gamma_1((\rho \otimes 1_U)^{(r-1)}) \cdot \gamma_{n-1}(1) \right) \\
	&=\bss \left( \gamma_1(\rho^{(r-1)} \otimes 1_{U^{(r-1)}}) \cdot \gamma_{n-1}(1) \right).
	\end{align*}
Ignoring the Frobenius twist on $\rho$, this is the definition of the differential $d(1)$, applied to terms of polynomial degree $n$ in $\bsb(1)(U^{(r-1)})$. Then we deduce that, via the identification \eqref{eq:cohomology-along-rows-B(1)}, the action of $d''$ on $\Hstarbul(\bsb^{\bullet,\bullet},d')$ corresponds to the differential $d(1)$ on $\bsb(1)$. Then by the case $r=1$ in Section \ref{subsec:r=1}, it follows that $d''$ makes $\Hstarbul(\bsb^{\bullet,\bullet},d')$ into a normal $p$-complex, and
	\begin{equation} \label{eq:cohomology-along-columns}
	\Hstarbul(\Hstarbul(\bsb^{\bullet,\bullet},d'),d'') \cong \Hstarbul(\bsb(1)^\bullet,d(1)) \circ \bsi^{(r-1)} \cong \bss^{(1)} \circ \bsi^{(r-1)} = \bss^{(r)}.
	\end{equation}
Specifically, \eqref{eq:cohomology-along-columns} is induced by the superalgebra homomorphism 
	\[
	\eta_1(U^{(r-1)}) : \bss(U^{(r)}) = \bss^{(1)}(U^{(r-1)}) \to \Hstarbul(\bsb^{\bullet,\bullet},d')
	\]
that is defined on generators by
	\[
	u_j^{(r)} \mapsto \begin{cases}
	\text{the cohomology class of} \quad (u_{0,j})^p & \text{if $u_j \in \Uzero$,} \\
	\text{the cohomology class of} \quad u_{0,j} \cdot u_{1,j} \cdots u_{p-1,j} & \text{if $u_j \in \Uone$.}
	\end{cases}
	\]

If $u_j$ is even, then $(u_{0,j})^p = [(\sha_0 \otimes u_j)^{p^{r-1}}]^p = (\sha_0 \otimes u_j)^{p^r}$ in $\bss(\Sha_r \otimes U)$. If $u_j$ is odd, then $u_{i,j} = (\sha_i \otimes u_j)(\sha_{i+p} \otimes u_j)(\sha_{i+2p} \otimes u_j) \cdots (\sha_{i+p(p^{r-1}-1)} \otimes u_j)$, so the product $u_{0,j} u_{1,j} \cdots u_{p-1,j}$ includes each of the terms $\sha_\ell \otimes u_j$ for $0 \leq \ell < p^r$ exactly once. Since $\sha_\ell \otimes u_j$ and $\sha_{\ell'} \otimes u_j$ anti-commute in $\bss(\Sha_r \otimes U)$, this implies that, up to a sign that depends only the values of $p$ and $r$, the product $u_{0,j} u_{1,j} \cdots u_{p-1,j}$ is equal to $(\sha_0 \otimes u_j) \cdot (\sha_1 \otimes u_j) \cdots (\sha_{p^r-1} \otimes u_j)$. Then it follows that the map $\eta_r(U)$ defined in Section \ref{subsubsec:define-etar} also induces an isomorphism $\bss^{(r)} \cong \Hstarbul( \Hstarbul( \bsb^{\bullet,\bullet},d'),d'')$.

Since the image of $\eta_r(U)$ is concentrated in polynomial degrees divisible by $p^r$, it follows that $\Hstarbul( \Hstarbul( \bsb_n^{\bullet,\bullet},d'),d'') = 0$ if $p^r \nmid n$, and $\Hstarbul( \Hstarbul( \bsb_{p^rn}^{\bullet,\bullet},d'),d'') \cong \bss^{n(r)}$. Since $(u_{0,j})^p$ is in bidegree $(\deg',\deg'') = (0,0)$ and $u_{0,j} u_{1,j} \cdots u_{p-1,j}$ is in bidegree $(\deg',\deg'') = ( p^2 \binom{p^{r-1}}{2}, p^{r-1} \binom{p}{2})$, one sees that the summand $S_0^{n-\ell(r)} \otimes \Lambda_1^{\ell(r)}$ of $\bss^{n(r)}$ occurs in $\Hstar^a(\Hstar^b(\bsb_{p^r n}^{\bullet,\bullet},d'),d'')$, where $a = \ell \cdot p^{r-1} \binom{p}{2}$ and $b = \ell \cdot p^2 \binom{p^{r-1}}{2}$. The total cohomological degree of $S_0^{n-\ell(r)} \otimes \Lambda_1^{\ell(r)}$ is then equal to
	\[ \textstyle
	a+b = \ell \cdot p^{r-1} \binom{p}{2} + \ell \cdot p^2 \binom{p^{r-1}}{2} = \ell \cdot \binom{p^r}{2}.
	\]

\subsubsection{Spectral sequences arising from contracted complexes}

For the rest of this section, set $\bsb = \bsb(r) = \Tot(\bsb^{\bullet,\bullet})$. Our goal in the rest of this section is show for all $1 \leq s < p$ that $\Hs^m(\bsb)$ is equal to the component of total cohomological degree $m$ in $\text{``$E_2$''} := \Hstarbul( \Hstarbul( \bsb^{\bullet,\bullet},d'),d'')$.

Given integers $1 \leq s < p$ and $0 \leq t < (p-s)p^{r-1}$, define $\phi = \phi_{r,s,t}: \N \to \N$ by
	\[
	\phi(\ell) = 
	\begin{cases}
	t+m p^r & \text{if $\ell = 2m$ is even,} \\
	t+sp^{r-1}+mp^r & \text{if $\ell=2m+1$ is odd.}
	\end{cases}
	\]
Then $\bsbst^\ell = \bsb^{\phi(\ell)}$. Let $\bsb = F^0 \bsb \supseteq F^1 \bsb \supseteq F^2 \bsb \supseteq \cdots$ be the decreasing row-wise filtration on $\bsb$, defined by $F^i \bsb^m = \bigoplus_{\ell \geq i} \bsb^{m-\ell,\ell}$. Since $d''(F^i \bsb^m) \subseteq F^{i+1} \bsb^{m+1}$, the differential on the associated graded $p$-complex $\bigoplus_{i \in \N} F^i \bsb^{\bullet} / F^{i+1} \bsb^{\bullet} \cong \bigoplus_{i \in \N} \bsb^{\bullet-i,i}$ identifies with the horizontal differential $d'$ on $\bsb^{\bullet,\bullet}$. The filtration on $\bsb^{\bullet,\bullet}$ induces a decreasing filtration on the contracted complex $\bsbst$, with $F^0 \bsbst = \bsbst$ and $F^\ell \bsbst^m = 0$ if $\ell > \phi(m)$. The filtration is finite in each cohomological degree, so it gives rise to a spectral sequence converging to $\Hbul(\bsbst)$, with
	\[
	E_0^{i,j} = F^i \bsbst^{i+j} / F^{i+1} \bsbst^{i+j} = F^i \bsb^{\phi(i+j)} / F^{i+1} \bsb^{\phi(i+j)} \cong \bsb^{\phi(i+j)-i,i}.
	\]
The differential $d_0: E_0^{i,j} \to E_0^{i,j+1}$ then identifies with $(d')^s$ if $i+j$ is even, or with $(d')^{p-s}$ if $i+j$ is odd. Either way, since $(\bsb^{\bullet,i},d')$ is a normal $p$-complex by Section \ref{subsubsec:E2-page}, we conclude that
	\[
	E_1^{i,j} \cong \Hstar^{\phi(i+j)-i}(\bsb^{\bullet,i},d').
	\]
Using the notation introduced before \eqref{eq:E1-page-nonzero}, $E_1^{i,j}$ is spanned by the classes of the monomials $u_{i_1,j_1} \cdots u_{i_n,j_n}$ such that $\deg'(u_{i_1,j_1} \cdots u_{i_n,j_n}) = \phi(i+j)-i$ and $\deg''(u_{i_1,j_1} \cdots u_{i_n,j_n}) = i$. Since these monomials are all in the kernel of the horizontal differential $d'$, and since the differentials on each page of the spectral sequence are induced by the differential of the underlying complex $\bsbst$, it follows that, from the $E_1$-page onward, the differentials $d_c: E_c^{i,j} \to E_c^{i+c,j+1-c}$ of the spectral sequence are all induced by either $(d'')^s$ (when $i+j$ is even) or $(d'')^{p-s}$ (when $i+j$ is odd).

Let $c \geq 1$ be an integer such that $d_c: E_c^{i,j} \to E_c^{i+c,j+1-c}$ is nontrivial. Then $E_1^{i,j}$ and $E_1^{i+c,j+1-c}$ must both be nonzero, so by \eqref{eq:E1-page-nonzero},
	\[ \textstyle
	\phi(i+j)-i = \ell \cdot p \binom{p^{r-1}}{2} \quad \text{and} \quad \phi(i+j+1)-i-c = \ell' \cdot p \binom{p^{r-1}}{2}
	\]
for some integers $\ell$ and $\ell'$. Since $i$ and $i+c$ must also be multiples of $p^{r-1}$, then $c = \chat p^{r-1}$ for some integer $\chat \geq 1$. First suppose $i+j$ is even. Then
	\begin{align*}
	\textstyle 	\ell' \cdot p \binom{p^{r-1}}{2} &= \phi(i+j+1)-i-c \\
	&= \phi(i+j)+ sp^{r-1} - i - c \\
	&= \textstyle \ell \cdot p \binom{p^{r-1}}{2} + sp^{r-1} - \chat p^{r-1},
	\end{align*}
and hence $s - \chat = (\ell' - \ell) \cdot p \cdot \frac{p^{r-1}-1}{2}$ is a multiple of $p$. (Note that $\frac{p^{r-1}-1}{2}$ is an integer by the standing assumption throughout this paper that $p$ is odd.) Since $1 \leq s < p$, and since $\chat \geq 1$, this implies that the two smallest possible values for $\chat$ are $s$ and $s+p$, and hence the two smallest possible values for $c$ are $s p^{r-1}$ and $s p^{r-1} + p^r$. Next suppose $i+j$ is odd. Then
	\begin{align*}
	\textstyle \ell' \cdot p \binom{p^{r-1}}{2} &= \phi(i+j+1)-i-c \\
	&= \phi(i+j)+ (p-s)p^{r-1} - i - c  \\
	&= \textstyle \ell \cdot p \binom{p^{r-1}}{2} + (p-s)p^{r-1} - \chat p^{r-1},
	\end{align*}
and hence $(p-s)-\chat = (\ell' - \ell) \cdot p \cdot \frac{p^{r-1}-1}{2}$ is a multiple of $p$. Since $1 \leq s < p$, and since $\chat \geq 1$, this implies that the two smallest possible values for $\chat$ are $(p-s)$ and $(p-s)+p$, and hence the two smallest possible values for $c$ are $(p-s)p^{r-1}$ and $(p-s)p^{r-1} + p^r$.

The observations of the previous paragraph imply that for $c$ in the range $1 \leq c \leq p^r$, the only pages $E_c$ of the spectral sequence that have nontrivial differentials are $E_{sp^{r-1}}$ and $E_{(p-s)p^{r-1}}$. On the $E_{sp^{r-1}}$-page, the nontrivial differentials originate at terms $E_{sp^{r-1}}^{i,j}$ with $i+j$ even, while on the $E_{(p-s)p^{r-1}}$-page they originate at terms $E_{(p-s)p^{r-1}}^{i,j}$ with $i+j$ odd. From this information, and from the earlier observation that $d_c$ is induced by either $(d'')^s$ or $(d'')^{p-s}$, we can calculate the $E_{p^r}$-page. First suppose $s < p-s$ and $i+j$ is even. So $E_c = E_1$ for $1 \leq c \leq sp^{r-1}$, $E_c = E_{(p-s)p^{r-1}}$ for $sp^{r-1} < c \leq (p-s)p^{r-1}$, and $E_c = E_{p^r}$ for $(p-s)p^{r-1} < c \leq p^r$. Then $E_{sp^{r-1}+1}^{i,j} = E_{(p-s)p^{r-1}}^{i,j}$ is equal to the kernel of the map
	\[
	E_1^{i,j} = E_{sp^{r-1}}^{i,j} \xrightarrow{d_{sp^{r-1}}} E_{sp^{r-1}}^{i+sp^{r-1},j+1-sp^{r-1}} = E_1^{i+sp^{r-1},j+1-sp^{r-1}},
	\]
which identifies with the map
	\[
	\Hstar^{\phi(i+j)-i}(\bsb^{\bullet,i},d') \to \Hstar^{\phi(i+j)-i}(\bsb^{\bullet,i+sp^{r-1}},d'),
	\]
induced by $(d'')^s$. Similarly,
	\[
	E_{sp^{r-1}+1}^{i-(p-s)p^{r-1},j-1+(p-s)p^{r-1}} = E_{(p-s)p^{r-1}}^{i-(p-s)p^{r-1},j-1+(p-s)p^{r-1}}
	\]
is equal to the cokernel of the map
	\[
	E_1^{i-p^r,j-2+p^r} = E_{sp^{r-1}}^{i-p^r,j-2+p^r} \xrightarrow{d_{sp^{r-1}}} E_{sp^{r-1}}^{i-(p-s)p^{r-1},j-1+(p-s)p^{r-1}} = E_1^{i-(p-s)p^{r-1},j-1+(p-s)p^{r-1}},
	\]
which identifies with the map
	\[
	\Hstar^{\phi(i+j)-i}(\bsb^{\bullet,i-p^r},d') \to \Hstar^{\phi(i+j)-i}(\bsb^{\bullet,i-(p-s)p^{r-1}},d'),
	\]
induced by $(d'')^s$. Finally, $E_{(p-s)p^{r-1}+1}^{i,j} = E_{p^r}^{i,j}$ is equal to the cokernel of the map
	\[
	E_{(p-s)p^{r-1}}^{i-(p-s)p^{r-1},j-1+(p-s)p^{r-1}} \xrightarrow{d_{(p-s)p^{r-1}}} E_{(p-s)p^{r-1}}^{i,j},
	\]
which is induced by $(d'')^{p-s}$. Then we conclude that $E_{p^r}^{i,j} \cong \Hs^i( \Hstar^{\phi(i+j)-i}(\bsb^{\bullet,\bullet},d'),d'')$. However, since $(\Hstarbul(\bsb^{\bullet,\bullet},d'),d'')$ is a normal $p$-complex by Section \ref{subsubsec:E2-page}, we can simply write
	\[
	E_{p^r}^{i,j} \cong \Hstar^i( \Hstar^{\phi(i+j)-i}(\bsb^{\bullet,\bullet},d'),d'').
	\]
This conclusion also holds if either $s > p-s$ or if $i+j$ is odd, via entirely analogous reasoning.

We've now shown for all $i$ and $j$ that $E_{p^r}^{i,j} \cong \Hstar^i( \Hstar^{\phi(i+j)-i}(\bsb^{\bullet,\bullet},d'),d'')$. Then by the observations at the end of Section \ref{subsubsec:E2-page}, $E_{p^r}^{i,j} \neq 0$ only if $i = \ell \cdot p^{r-1} \binom{p}{2}$ and $\phi(i+j)-i = \ell \cdot p^2 \binom{p^{r-1}}{2}$ for some integer $\ell$. Then $\phi(i+j) = \ell \cdot \binom{p^r}{2} = \ell \cdot p^r \cdot \frac{p^r-1}{2}$. Since $1 \leq s < p$ and $0 \leq t < (p-s)p^{r-1}$, it follows that $\phi(i+j) = \phi_{r,s,t}(i+j)$ is equal to a multiple of $p^r$ only if $t = 0$ and $i+j$ is even. Then the $E_{p^r}$-page is concentrated in even total degrees, which implies that there can be no further nontrivial differentials in the spectral sequence, and hence $E_{p^r} = E_\infty$. Moreover, for each total cohomological degree $m$, there is at most one pair of indices $i$ and $j$ such that $i+j=m$ and $E_\infty^{i,j} \neq 0$. Namely, if $\phi(m) = \ell \cdot \binom{p^r}{2}$, then $i = \ell \cdot p^{r-1} \binom{p}{2}$ and $j = m-i$. This implies that $\Hbul(\bsbst)$ is \emph{equal} to the $E_\infty$-page of the spectral sequence, and not just equivalent modulo a filtration.

From the preceding observations we deduce that $\Hbul(\bsbst) = 0$ for $t \neq 0$, and that $\opH^m(\bsbszero) \neq 0$ only if $m$ is even and of the form $m = \ell \cdot (p^r-1) = 2 ( \ell \cdot \frac{p^r-1}{2} )$ for some integer $\ell$. Then by \eqref{eq:C-to-Cst}, we deduce that $\Hs^m(\bsb) \neq 0$ only if $m = \ell \cdot \binom{p^r}{2}$ for some integer $\ell$. For $m$ of this form one has
	\[
	\Hs^{\ell \cdot \binom{p^r}{2}}(\bsb) = \opH^{\ell \cdot (p^r-1)}(\bsbszero) = E_\infty^{i,j} \cong \Hstar^{\ell \cdot p^{r-1}\binom{p}{2}}(\Hstar^{\ell \cdot p^2 \binom{p^{r-1}}{2}}(\bsb^{\bullet,\bullet},d'),d''),
	\]
where $i = \ell \cdot p^{r-1}\binom{p}{2}$ and $j = m - i$. Restricting to the component of a single polynomial degree in $\bsb$, one gets $\Hs^\bullet(\bsb_n) = 0$ if $p^r \nmid n$, and
	\[
	\Hs^{\ell \cdot \binom{p^r}{2}}(\bsb_{p^rn}) \cong \Hstar^{\ell \cdot p^{r-1}\binom{p}{2}}(\Hstar^{\ell \cdot p^2 \binom{p^{r-1}}{2}}(\bsb_{p^rn}^{\bullet,\bullet},d'),d'') \cong S_0^{n-\ell(r)} \otimes \Lambda_1^{\ell(r)},
	\]
the last isomorphism induced by the family of superalgebra homomorphisms defined in Section \ref{subsubsec:define-etar}. Since this holds for all $1 \leq s < p$, this implies that $\bsb = \bsb(r)$ is a normal $p$-complex, and completes the proof of Theorem \ref{theorem:B(r)-cohomology}.

\section{\texorpdfstring{$\Ext_{\bsp}^\bullet(\bsir,\bsir)$}{Ext(Ir,Ir)}} \label{section:Ext(Ir,Ir)}

\subsection{Constructing injective resolutions of \texorpdfstring{$\bsirzero$ and $\bsirone$}{Izero(r) and Ione(r)}} \label{subsection:injective-resolutions}

In this section we describe how the $p$-complex $\bsb(r)$ can be used to construct injective resolutions of $\bsirzero$ and $\bsirone$. First, given $n \in \N$, let $T(\bss^n,r)$ be the contraction $\bsb_{p^rn}(r)_{[1,0]}$ of $\bsb_{p^rn}(r)$. Then
	\[
	T(\bss^n,r)^{2i} = \bsb_{p^rn}(r)^{p^r i} \quad \text{and} \quad T(\bss^n,r)^{2i+1} = \bsb_{p^rn}(r)^{p^r i + p^{r-1}}.
	\]
Write $\partial = \partial_\ell: T(\bss^n,r)^{\ell} \to T(\bss^n,r)^{\ell+1}$ for the differential on $T(\bss^n,r)$. Then $\partial_\ell = d(r)$ if $\ell$ is even, and $\partial_\ell = [d(r)]^{p-1}$ if $\ell$ is odd. The complex $T(\bss^n,r)$ consists of injective objects in $\bsp_{p^r n}$ because each $T(\bss^n,r)^\ell$ is a direct summand of the injective object $\bss^{p^rn}(\Sha_r \otimes -) \in \bsp_{p^r n}$. Since $\bss^{p^r n}(\Sha_r \otimes -)^\ell = 0$ for $\ell > p^r n \cdot (p^r-1)$, then $T(\bss^n,r)^\ell = 0$ for $\ell > 2n \cdot (p^r-1)$. The following is an immediate corollary of Theorem \ref{theorem:B(r)-cohomology}:

\begin{corollary} \label{cor:T(Sn,r)-cohomology}
The cohomology of the cochain complex $T(\bss^n,r)$ is even-isomorphic to $\bss^{n(r)}$, with the summand $S_0^{n-\ell(r)} \otimes \Lambda_1^{\ell(r)}$ of $\bss^{n(r)}$ located in cohomology degree $\ell \cdot (p^r-1)$.
\end{corollary}

Now set $T = T(\bsi,r) = T(\bss^1,r)$. By Corollary \ref{cor:T(Sn,r)-cohomology}, $T$ has nonzero cohomology in exactly two degrees: $\opH^0(T) \cong \bsirzero$ and $\opH^{p^r-1}(T) \cong \bsirone$. Set $(\Tbar,\partialbar) = \bspi \circ (T,\partial) \circ \bspi$. Then $\Tbar$ is also a complex of injectives, with $\opH^0(\Tbar) \cong \bsirone$ and $\opH^{p^r-1}(\Tbar) \cong \bsirzero$. Set $K = \ker(\partial_{p^r-1})$. Let $j: K \hookrightarrow T^{p^r-1}$ be the inclusion, and let $\pi: K \twoheadrightarrow \bsirone$ be the projection, whose kernel is $\im(\partial_{p^r-2})$. Let $\iota: \bsirone \hookrightarrow \Tbar^0$ be the inclusion. Now
	\[
	K \xrightarrow{j} T^{p^r-1} \xrightarrow{\partial} T^{p^r} \xrightarrow{\partial} \cdots  \xrightarrow{\partial} T^{2p^r-2} \to 0
	\]
is an acyclic complex (in fact, an injective resolution of $K$), and each term in the complex $\Tbar$ is injective, so the projection $\pi: K \twoheadrightarrow \bsirone$ extends to a commutative diagram
	\begin{equation} \label{eq:chain-map-epsilon-prime}
	\vcenter{\xymatrix{
	K \ar@{^(->}[r]^-{j} \ar@{->>}[d]^{\pi} & T^{p^r-1} \ar@{->}[r]^{\partial} \ar@{->}[d]^{\ve_{p^r-1}'} & T^{p^r} \ar@{->}[r]^-{\partial} \ar@{->}[d]^{\ve_{p^r}'} & \cdots \ar@{->}[r]^-{\partial} & T^{2p^r-2} \ar@{->}[r] \ar@{->}[d]^{\ve_{2p^r-2}'} & 0 \ar@{->}[d] \\
	\bsirone \ar@{^(->}[r]^-{\iota} & \Tbar^0 \ar@{->}[r]^{\partialbar} & \Tbar^1 \ar@{->}[r]^-{\partialbar} & \cdots \ar@{->}[r]^-{\partialbar} & \Tbar^{p^r-1} \ar@{->}[r]^{\partialbar} & \Tbar^{p^r}
	}}
	\end{equation}
in which each $\ve_\ell': T^\ell \to \Tbar^{\ell- p^r+1}$ is an even homomorphism; later in Section \ref{subsec:epsilon-prime} we will give an explicit formula for a choice of $\ve'$ in the case $r=1$. Set $\ve_\ell = (-1)^\ell \cdot \ve_\ell'$, and set $\ve_\ell = 0$ if either $\ell < p^r-1$ or $\ell > 2p^r-2$. Then for all $\ell \in \N$ one has
	\[
	\partialbar_{\ell-p^r+1} \circ \ve_\ell = -\ve_{\ell+1} \circ \partial_\ell : T^\ell \to \Tbar^{\ell-p^r+2}.
	\]

Set $Q = T \oplus \Tbar\subgrp{p^r}$. So for all $\ell \in \N$ one has $Q^\ell = T^\ell \oplus \Tbar^{\ell - p^r}$, where by definition $\Tbar^\ell = 0$ if $\ell < 0$, and $T^\ell = \Tbar^\ell = 0$ if $\ell \geq 2p^r-1$. Now define $\delta = \delta_\ell: Q^\ell \to Q^{\ell+1}$ by
	\[
	\delta_\ell(x,y) = \left( \partial_\ell(x),\ve_\ell(x) + \partialbar_{\ell-p^r}(y) \right)) \quad \text{for $x \in T^\ell$ and $y \in \Tbar^{\ell-p^r}$.}
	\]
We claim that $\delta^2 = 0$, and then that
	\[
	\opH^i(Q,\delta) \cong \begin{cases}
	0 & \text{if $i \notin \set{0, 2p^r-1}$,} \\
	\bsirzero & \text{if $i \in \set{0, 2p^r-1}$.}
	\end{cases}
	\]
Thinking of the elements of $Q^\ell$ as column vectors, one has
	\[
	\delta^2 = \bmat{\partial & 0 \\ \ve & \partialbar} \bmat{\partial & 0 \\ \ve & \partialbar} = \bmat{\partial^2 & 0 \\ \ve \circ \partial + \partialbar \circ \ve & \partialbar^2 } = \bmat{0 & 0 \\ 0 & 0},
	\]
so $\delta$ makes $Q$ into a complex. The claim about $\opH^i(Q,\delta)$ is evident if either $i \leq p^r-2$ or if $i \geq 2p^r$, since then $(Q,\delta)$ is identifies with either $(T,\partial)$ or with a shifted copy of $(\Tbar,\partialbar)$. So it suffices to show that $\opH^i(Q,\delta) = 0$ for $p^r-1 \leq i \leq 2p^r-2$, and to show that $\opH^{2p^r-1}(Q,\delta) \cong \bsirzero$.

First let $i = p^r-1$. Let $x \in Q^{p^r-1} = T^{p^r-1}$, and suppose that $\delta_{p^r-1}(x) = (\partial_{p^r-1}(x),\ve_{p^r-1}(x)) = (0,0)$. Then $x \in \ker(\partial_{p^r-1}) = K$, so $\ve_{p^r-1}(x) = (\ve_{p^r-1} \circ j)(x) = (\iota \circ \pi)(x)$. Since $\iota$ is an injection, this implies that $x \in \ker(\pi) = \im(\partial_{p^r-2})$. Then $\opH^{p^r-1}(Q,\delta) = 0$.

Next let $i = p^r$. Let $(x,y) \in Q^{p^r} = T^{p^r} \oplus \Tbar^0$, and suppose that $\delta_{p^r}(x,y) = (\partial_{p^r}(x), \ve_{p^r}(x) + \partialbar_0(y)) = (0,0)$. Then $x \in \ker(\partial_{p^r}) = \im(\partial_{p^r-1})$, so $x = \partial_{p^r-1}(u)$ for some $u \in T^{p^r-1}$. Then $\partialbar_0(y) = -\ve_{p^r}(x) = -\ve_{p^r}(\partial_{p^r-1}(u)) = \partialbar_0(\ve_{p^r-1}(u))$, so $y- \ve_{p^r-1}(u) \in \ker(\partialbar_0) = \im(\iota) = \im(\iota \circ \pi)$. Let $u' \in K$ such that $(\iota \circ \pi)(u') = y - \ve_{p^r-1}(u)$. Then
	\begin{align*}
	\delta_{p^r-1}(u+j(u')) &= \left( \partial_{p^r-1}(u+j(u')) , \ve_{p^r-1}(u+j(u')) \right) \\
	&= \left( \partial_{p^r-1}(u) + (\partial_{p^r-1} \circ j)(u'), \ve_{p^r-1}(u) + (\ve_{p^r-1} \circ j)(u') \right) \\
	&= \left( x + 0, \ve_{p^r-1}(u) + (\iota \circ \pi)(u')\right) \\
	&= \left( x, \ve_{p^r-1}(u) + y-\ve_{p^r-1}(u) \right) = \left( x,y\right),
	\end{align*}
so $(x,y) \in \im(\delta_{p^r-1})$, and hence $\opH^{p^r}(Q,\delta) = 0$.

Now let $p^r+1 \leq i < 2p^r-2$. Let $(x,y) \in Q^i = T^i \oplus \Tbar^{i-p^r}$, and suppose that $\delta_i(x,y) = (\partial_i(x),\ve_i(x) + \partial_{i-p^r}(y)) = (0,0)$. Then $x \in \ker(\partial_i) = \im(\partial_{i-1})$, so $x = \partial_{i-1}(u)$ for some $u \in T^{i-1}$. Then $\partial_{i-p^r}(y) = -\ve_i(x) = -\ve_i(\partial_{i-1}(u)) = \partialbar_{i-p^r}(\ve_{i-1}(u))$, so $y - \ve_{i-1}(u) \in \ker(\partialbar_{i-p^r}) = \im(\partialbar_{i-p^r-1})$. Let $u' \in \Tbar^{i-p^r-1}$ such that $\partialbar_{i-p^r-1}(u') = y-\ve_{i-1}(u)$. Then
	\begin{align*}
	\delta_{i-1}(u,u') &= \left( \partial_{i-1}(u),\ve_{i-1}(u) + \partialbar_{i-p^r-1}(u') \right) \\
	&= \left( x , \ve_{i-1}(u) + y-\ve_{i-1}(u) \right) = \left( x, y \right),
	\end{align*}
so $(x,y) \in \im(\delta_{i-1})$, and hence $\opH^i(Q,\delta) = 0$.

Finally, let $i = 2p^r-1$. Then $Q^{2p^r-1} = \Tbar^{p^r-1}$, $Q^{2p^r} = \Tbar^{p^r}$, and $\delta_{2p^r-1}$ identifies with $\partialbar_{p^r-1}$, so $\ker(\delta_{2p^r-1}) = \ker(\partialbar_{p^r-1})$. On the other hand, $\delta_{2p^r-2}$ is the map
	\[
	T^{2p^r-2} \oplus \Tbar^{p^r-2} \to \Tbar^{p^r-1}, \quad (x,y) \mapsto \ve_{2p^r-2}(x) + \partialbar_{p^r-2}(y).
	\]
Note that $T^{2p^r-2} = \im(\partial_{2p^r-3})$, so $x = \partial_{2p^r-3}(u)$ for some $u \in T^{2p^r-3}$. Then
	\[
	\ve_{2p^r-2}(x) = \ve_{2p^r-2}(\partial_{2p^r-3}(u)) = -\partialbar_{p^r-2}(\ve_{2p^r-3}(u)),
	\]
and hence $\im(\delta_{2p^r-2}) \subseteq \im(\partialbar_{p^r-2})$. Since $\delta_{2p^r-2}(0,y) = \partialbar_{p^r-2}(y)$, we deduce that $\im(\delta_{2p^r-2}) = \im(\partialbar_{p^r-2})$. Then $\opH^{2p^r-1}(Q,\delta) = \ker(\partialbar_{p^r-1}) / \im(\partialbar_{p^r-2}) \cong \bsirzero$.

We've shown that we can splice together copies of $T$ and $\Tbar\subgrp{p^r}$ to construct a complex $Q$ of injective objects in $\bsp_{p^r}$ whose cohomology is $\bsirzero$ in degrees $0$ and $2p^r-1$, and is zero otherwise. The complex $Q$ can be visualized as follows (the numbers in the top row record the cohomological degrees of the terms in that column):
	\[
	\vcenter{\xymatrix @C=1.2pc @R=.5pc {
	\rotatebox[origin=c]{270}{$0$} & & \rotatebox[origin=c]{270}{$p^r-1$} & \rotatebox[origin=c]{270}{$p^r$} & \rotatebox[origin=c]{270}{$p^r+1$} & & \rotatebox[origin=c]{270}{$2p^r-2$} & \rotatebox[origin=c]{270}{$2p^r-1$} & & \rotatebox[origin=c]{270}{$3p^r-2$} \\
	T^0 \ar@{->}[r]^-{\partial} & \cdots \ar@{->}[r]^-{\partial} & T^{p^r-1} \ar@{->}[r]^-{\partial} \ar@{->}[rdd]^{\ve} & T^{p^r} \ar@{->}[r]^-{\partial} \ar@{->}[rdd]^{\ve} & T^{p^r+1} \ar@{->}[r]^-{\partial} \ar@{->}[rdd]^{\ve} & \mathrlap{\phantom{T}}\cdots \ar@{->}[r]^-{\partial} \ar@{->}[rdd]^{\ve} & T^{2p^r-2} \ar@{->}[rdd]^{\ve} & & \\
	& & & \oplus & \oplus & & \oplus & & & \\
	& & & \Tbar^0 \ar@{->}[r]^-{\partialbar} & \Tbar^1 \ar@{->}[r]^-{\partialbar} & \mathrlap{\phantom{\Tbar}}\cdots \ar@{->}[r]^-{\partialbar} & \Tbar^{p^r-2} \ar@{->}[r]^-{\partialbar} & \Tbar^{p^r-1} \ar@{->}[r]^-{\partialbar} & \cdots \ar@{->}[r]^-{\partialbar} & \Tbar^{2p^r-2}
	}}
	\]
	
Now we'll splice together shifts of $Q$ to construct an injective resolution $J(r)$ of the functor $\bsirzero$. First, as a graded superfunctor, $J(r) = \bigoplus_{n \geq 0} Q\subgrp{2np^r}$. Next, we splice the successive copies of $Q$ together in a manner similar to how we spliced together $T$ and $\Tbar$. Specifically, set $\vebar_\ell = \bspi \circ \ve_\ell \circ \bspi$. Then $Q\subgrp{2np^r}$ is spliced to $Q\subgrp{2(n-1)p^r}$ as shown in the following diagram:
	\[
	\vcenter{\xymatrix @C=1.2pc @R=.5pc {
	\mathrlap{\phantom{T}}\cdots \ar@{->}[r]^-{\partial} \ar@{->}[rdd]^{\ve} & T^{2p^r-2} \ar@{->}[rdd]^{\ve} & & T^0 \ar@{->}[r]^-{\partial} & T^1 \ar@{->}[r]^-{\partial} & \cdots \ar@{->}[r]^-{\partial} & T^{p^r-2} \ar@{->}[r]^-{\partial} & T^{p^r-1} \ar@{->}[r]^-{\partial} \ar@{->}[rdd]^{\ve} & T^{p^r} \ar@{->}[r]^-{\partial} \ar@{->}[rdd]^{\ve} & \cdots \\
	& \oplus & & \oplus & \oplus & & \oplus & & \oplus & \\
	\mathrlap{\phantom{\Tbar}}\cdots \ar@{->}[r]^-{\partialbar} & \Tbar^{p^r-2} \ar@{->}[r]^-{\partialbar} & \Tbar^{p^r-1} \ar@{->}[r]^-{\partialbar} \ar@{->}[ruu]^{\vebar} & \Tbar^{p^r} \ar@{->}[r]^-{\partialbar} \ar@{->}[ruu]^{\vebar} & \Tbar^{p^r+1} \ar@{->}[r]^-{\partialbar} \ar@{->}[ruu]^{\vebar} & \mathrlap{\phantom{\Tbar^{p^r+1}}}\cdots \ar@{->}[r]^-{\partialbar} \ar@{->}[ruu]^{\vebar} & \Tbar^{2p^r-2} \ar@{->}[ruu]^{\vebar} & & \Tbar^0 \ar@{->}[r]^-{\partialbar} & \mathrlap{\phantom{\Tbar^1}}\cdots
	}}
	\]
In other words, the splicing between the terminal segment $\Tbar^{p^r-1} \to \cdots \to \Tbar^{2p^r-2}$ of $Q\subgrp{2(n-1)p^r}$ and the initial segment $T^0 \to \cdots \to T^{p^r-1}$ of $Q\subgrp{2np^r}$ is exactly the same as the maps in the relevant range of the complex $\Qbar := \bspi \circ Q \circ \bspi$.

With the successive copies of $Q\subgrp{2np^r}$ spliced together in the manner described above, the earlier calculations for $Q$ imply that $J(r)$ is exact except in cohomological degree $0$, and hence $J(r)$ is an injective resolution of $\bsirzero$. Set $\Jbar(r) = \bspi \circ J(r) \circ \bspi$. Then $\Jbar(r)$ is an injective resolution of $\bsirone$. Ignoring the differentials, one has, as strict polynomial superfunctors,
	\begin{align}
	J(r) &= \bigoplus_{n \geq 0} \left( T(\bsi,r)\subgrp{2np^r} \oplus \Tbar(\bsi,r)\subgrp{(2n+1)p^r} \right), \quad \text{and} \label{eq:J(r)-decomposition} \\
	\Jbar(r) &= \bigoplus_{n \geq 0} \left( \Tbar(\bsi,r)\subgrp{2np^r} \oplus T(\bsi,r)\subgrp{(2n+1)p^r} \right). \label{eq:Jbar(r)-decomposition}
	\end{align}

\subsection{A formula for \texorpdfstring{$\ve'$}{epsilon'} when \texorpdfstring{$r=1$}{r=1}} \label{subsec:epsilon-prime}

In this section we describe, for the case $r=1$, an explicit chain map $\ve'$ that makes the diagram \eqref{eq:chain-map-epsilon-prime} commute. We do this by constructing a homomorphism at the level of $p$-complexes. With this choice for $\ve'$, the injective resolution $J(1)$ of $\bsionezero$ is made completely explicit. First we set some notation that is valid for any $r \geq 1$.

Set $(\bsbbar(r),\dol(r)) = \bspi \circ (\bsb(r),d(r)) \circ \bspi$, so that $\Tbar(\bsi,r)$ is the contraction $\bsbbar_{p^r}(r)_{[1,0]}$ of $\bsbbar_{p^r}(r)$. Set $\Shabar_r = \Pi(\Sha_r)$, and for $0 \leq i < p^r$, set $\shabar_i = (\sha_i)^\pi \in \Pi(\Sha_r)$. Then
	\[
	\bsbbar_{p^r}(r) = \bspi \circ \bss^{p^r}(\Sha_r \otimes -) \circ (k^{0|1} \otimes -) = \bspi \circ \bss^{p^r}( (\Sha_r \otimes k^{0|1}) \otimes -) = \bspi \circ \bss^{p^r}(\Shabar_r \otimes -).
	\]
From \eqref{eq:SYoneda}, one gets
	\begin{equation} \label{eq:HomB-to-Bbar}
	\begin{split}
	\Hom_{\bsp}(\bsb_{p^r}(r),\bsbbar_{p^r}(r)) &= \Hom_{\bsp}(\bss^{p^r}(\Sha_r \otimes-), \bspi \circ \bss^{p^r}(\Shabar_r \otimes -)) \\
	&\cong \bspi \circ \Hom_{\bsp}(\bss^{p^r}(\Sha_r \otimes-), \bss^{p^r}(\Shabar_r \otimes -)) \\
	&\cong \bspi \circ \bsg^{p^r} \Hom_k(\Sha_r,\Shabar_r).
	\end{split}
	\end{equation}
Observe that since $\Hom_k(\Sha_r,\Shabar_r)$ is a purely odd superspace, $\bsg^{p^r} \Hom_k(\Sha_r,\Shabar_r)$ identifies with $\Lambda^{p^r}(\Hom_k(\Sha_r,\Shabar_r))$. Given $f_1,f_2,\ldots,f_{p^r} \in \Hom_k(\Sha_r,\Shabar_r)$, we may write $f_1 \cdot f_2 \cdots f_{p^r}$ to denote the product $\gamma_1(f_1) \cdot \gamma_1(f_2) \cdots \gamma_1(f_{p^r}) \in \bsg^{p^r} \Hom_k(\Sha_r,\Shabar_r)$. Given integers $0 \leq i ,j < p^r$, let $\alpha_{ij} \in \Hom_k(\Sha_r,\Shabar_r)$ be the `matrix unit' defined by $\alpha_{ij}(\sha_\ell) = \delta_{j\ell} \cdot \shabar_i$. Then $\Hom_{\bsp}(\bsb_{p^r}(r)^j,\bsbbar_{p^r}(r)^i)$ is spanned by all expressions of the form
	\[
	{}^\pi(\alpha_{i_1,j_1} \cdot \alpha_{i_2,j_2} \cdots \alpha_{i_{p^r},j_{p^r}}) \quad \text{such that} \quad i_1 + \cdots + i_{p^r} = i \quad \text{and} \quad j_1 + \cdots + j_{p^r} = j.
	\]
In particular, $\Hom_{\bsp}(\bsb_{p^r}(r)^j,\bsbbar_{p^r}(r)^0)$ is spanned by the expressions ${}^\pi(\alpha_{0,j_1} \cdot \alpha_{0,j_2} \cdots \alpha_{0,j_{p^r}})$ with $j_1,j_2,\ldots,j_{p^r}$ all distinct, so
	\begin{equation} \label{eq:HomBBbarzero} \textstyle
	\Hom_{\bsp}(\bsb_{p^r}(r)^j,\bsbbar_{p^r}(r)^0) = 0 \quad \text{if} \quad j < 0 + 1 + 2+ \cdots + (p^r-1) = \binom{p^r}{2}.
	\end{equation}

Now we describe a particular choice for $\ve'(1)$. Given an integer $0 \leq j < p$, let $\varphi_j \in \Hom_k(\Sha_1,\Shabar_1)$ be defined by $\varphi_j(\sha_\ell) = (-1)^j \cdot \binom{\ell}{j} \cdot \shabar_{\ell-j}$, where by definition $\shabar_i = 0$ if $i < 0$. Of course, the binomial coefficient $\binom{\ell}{j}$ is also equal to zero if $j > \ell$. In terms of matrix units, one has
	\[ \textstyle
	\varphi_j = (-1)^j \cdot \sum_{\ell=j}^{p-1} \binom{\ell}{j} \cdot \alpha_{\ell-j,\ell} = (-1)^j \cdot \sum_{i=0}^{p-1-j} \binom{i+j}{i} \cdot \alpha_{i,i+j}.
	\]
Using Pascal's formula, it is straightforward to check that $\rhobar \circ \varphi_j - \varphi_j \circ \rho = \varphi_{j-1}$, where by definition $\varphi_{-1} = 0$. Now set
	\[
	\ve'(1) = {}^\pi(\varphi_0 \cdot \varphi_1 \cdots \varphi_{p-1}) \in \bspi \circ \bsg^p \Hom_k(\Sha_1,\Shabar_1) \cong \bspi \circ \Lambda^p \Hom_k(\Sha_1,\Shabar_1).
	\]
Since $\varphi_j$ maps $(\Sha_1)^\ell$ into $(\Shabar_1)^{\ell-j}$, it follows that $\ve'(1)$ maps $\bsb_p(1)^\ell$ into $\bsbbar_p(1)^{\ell - \binom{p}{2}}$. In the case $r=1$, composition with $\dol(1)$ and $d(1)$ take the particularly simple forms
	\begin{align*}
	\dol(1) \circ \ve'(1) &= \textstyle \sum_{i=0}^{p-1} {}^\pi(\varphi_0 \cdots \varphi_{i-1} \cdot [\rhobar \circ \varphi_i] \cdot \varphi_{i+1} \cdots \varphi_{p-1}), \quad \text{and} \\
	\ve'(1) \circ d(1) &= \textstyle \sum_{i=0}^{p-1} {}^\pi(\varphi_0 \cdots \varphi_{i-1} \cdot [\varphi_i \circ \rho] \cdot \varphi_{i+1} \cdots \varphi_{p-1}).
	\end{align*}
Then
	\[ \textstyle
	\dol(1) \circ \ve'(1) - \ve'(1) \circ d(1) = \sum_{i=0}^{p-1} {}^\pi(\varphi_0 \cdots \varphi_{i-1} \cdot [\varphi_{i-1}] \cdot \varphi_{i+1} \cdots \varphi_{p-1}) = 0
	\]
because the $i=0$ summand contains $0$ as a factor, and all other summands contain a repeated factor. So $\ve'(1)$ defines an even homomorphism of $p$-complexes $\ve'(1): \bsb_p(1) \to \bsbbar_p(1)\subgrp{\binom{p}{2}}$.

The restriction of $\ve'(1)$ to $\bsb_p(1)^{\binom{p}{2}}$ is the single monomial
	\begin{equation} \label{eq:ve-prime-p-1}
	\ve_{p-1}' := \prescript{\pi}{}{\left[ (-1)^0 \alpha_{0,0} \cdot (-1)^1 \alpha_{0,1} \cdots (-1)^{p-1} \alpha_{0,p-1} \right]}.
	\end{equation}
One gets $\ve_{p-1}' \circ d(1) = 0$ by \eqref{eq:HomBBbarzero}, and hence $\ve_{p-1}'$ also vanishes on  $\im(d(1)^{p-1}) = \im(\partial_{p^r-2}) = \ker(\pi: K \twoheadrightarrow \bsirone)$. On the other hand, given $U \in \bsv$ and $u \in \Uone$, one has
	\begin{equation} \label{eq:ve-prime-ell=pr-choose-2}
	\begin{split}
	\ve_{p-1}' \big((\sha_0 \otimes u) (\sha_1 \otimes u) &\cdots (\sha_{p-1} \otimes u) \big) \\
	&= \prescript{\pi}{}{\left[ (\alpha_{0,0}(\sha_0) \otimes u) (\alpha_{0,1}(\sha_1) \otimes u) \cdots (\alpha_{0,p^r-1}(\sha_{p-1}) \otimes u) \right]} \\
	&= \prescript{\pi}{}{\left[ (\shabar_0 \otimes u)^p \right]} \\
	&= \prescript{\pi}{}{\left[ (\sha_0 \otimes {}^\pi u)^p \right]} \in \bsbbar_p(1)^0(U).
	\end{split}
	\end{equation}
The signs $(-1)^i$ all cancel out at the first equals sign of \eqref{eq:ve-prime-ell=pr-choose-2}, because as each odd linear map $\alpha_{0,i}$ passes over the product $(\sha_0 \otimes u) \cdots (\sha_{i-1} \otimes u)$ of $i$ odd factors, it introduces its own sign of $(-1)^i$. Then from the discussion of Section \ref{subsubsec:define-eta}, it follows that $\ve_{p-1}'$ makes the left-most square of \eqref{eq:chain-map-epsilon-prime} commute, so $\ve'(1)$ restricts to a chain map making \eqref{eq:chain-map-epsilon-prime} commute.

For later reference, the restriction of $\ve'(1)$ to $\bsb_p(1)^{p(p-1)}$ is the single monomial
	\[
	\ve_{2p-2}' := \textstyle \prescript{\pi}{}{\left[ (-1)^0 \binom{p-1}{p-1} \alpha_{p-1,p-1} \cdot (-1)^1 \binom{p-1}{p-2} \alpha_{p-2,p-1} \cdots (-1)^{p-1} \binom{p-1}{0} \alpha_{0,p-1} \right]}.
	\]
Since $\binom{p-1}{i} \equiv (-1)^i \mod p$, this simplifies to
	\begin{equation} \label{eq:ve-prime-2p-2}
	\ve_{2p-2}' = \prescript{\pi}{}{\left[ \alpha_{p-1,p-1} \cdot \alpha_{p-2,p-1} \cdots \alpha_{0,p-1} \right]}.
	\end{equation}
	
\begin{remark}
For any integer $r \geq 1$, one can show that the chain map $\ve'$ in \eqref{eq:chain-map-epsilon-prime} can be realized as the restriction of an even homomorphism of $p$-complexes $\ve'(r) : \bsb_{p^r}(r) \to \bsbbar_{p^r}(r)\subgrp{\binom{p^r}{2}}$. Since the differentials on $\bsb(r)$ and $\bsbbar(r)$ each increase the $\Z$-degree by $p^{r-1}$, and since the objects in \eqref{eq:chain-map-epsilon-prime} correspond to terms in $\bsb_{p^r}(r)$ and $\bsbbar_{p^r}(r)$ of $\Z$-degrees divisible by $p^{r-1}$, one may assume that $\ve'(r)_\ell : \bsb_{p^r}(r)^\ell \to \bsbbar_{p^r}(r)^{\ell - \binom{p^r}{2}}$ is zero if $\ell < \binom{p^r}{2}$ or if $\ell \not\equiv 0 \mod p^{r-1}$. For $\ell = \binom{p^r}{2}$, let
	\[
	\ve'(r)_{\binom{p^r}{2}} = {}^\pi[ (-1)^0 \alpha_{0,0} \cdot (-1)^1 \alpha_{0,1} \cdots (-1)^{p^r-1} \alpha_{0,p^r-1}].
	\]
An argument like that given above for $r=1$ shows that this makes the left-hand square of \eqref{eq:chain-map-epsilon-prime} commute. Next let $\ell \geq \binom{p^r}{2}$, and assume by way of induction that $\ve'(r)_\ell$ has been defined for all $\ell' \leq \ell$, and for $\ell'$ in this range one has $\dol(r) \circ \ve'(r)_{\ell' - p^{r-1}} = \ve'(r)_{\ell'} \circ d(r)$. Now one checks that $\dol(r) \circ \ve'(r)_\ell$ vanishes on the kernel of the differential $d(r) : \bsb_{p^r}(r)^\ell \to \bsb_{p^r}(r)^{\ell + p^{r-1}}$. For $\ell = \binom{p^r}{2}$, this kernel is $K$, and one checks that $(\dol(r) \circ \ve'(r)_\ell)(K) = 0$ by the commutativity of the left-hand square of \eqref{eq:chain-map-epsilon-prime}. For $\ell > \binom{p^r}{2}$, this kernel is $\im(d(r)^{p-1})$, by Theorem \ref{theorem:B(r)-cohomology}, and
	\[
	\dol(r) \circ \ve'(r)_\ell \circ d(r)^{p-1} = \dol(r) \circ \dol(r)^{p-1} \circ \ve'(r)_{\ell - (p-1)p^{r-1}} = 0
	\]
because $\dol(r)^p = 0$. Now since $\bsbbar_{p^r}(r)^{\ell - \binom{p^r}{2}+p^{r-1}}$ is injective, there exists an even homomorphism
	\[
	\ve'(r)_{\ell+p^{r-1}}: \bsb_{p^r}(r)^{\ell+p^{r-1}} \to \bsbbar_{p^r}(r)^{\ell-\binom{p^r}{2} + p^{r-1}}
	\]
such that $\ve'(r)_{\ell+p^{r-1}} \circ d(r) = \dol(r) \circ \ve'(r)_\ell$.
\end{remark}

\subsection{The vector space structure of \texorpdfstring{$\Ext_{\bsp}^\bullet(\bsir,\bsir)$}{ExtP(Ir,Ir)}} \label{subsec:ExtP(Ir,Ir)}

In this section we show how the injective resolution $J(r)$ permits a quick calculation of $\Ext_{\bsp}^\bullet(\bsir,\bsir)$ as a graded superspace.

The decomposition $\bsir = \bsirzero \oplus \bsirone$ of the functor $\bsir$ leads to the matrix ring decomposition
	\begin{equation} \label{eq:matrixring}
	\renewcommand*{\arraystretch}{1.5}
	\Ext_{\bsp}^\bullet(\bsir,\bsir) = \begin{pmatrix}
	\Ext_{\bsp}^\bullet(\bsi_0^{(r)},\bsi_0^{(r)}) & \Ext_{\bsp}^\bullet(\bsi_1^{(r)},\bsi_0^{(r)}) \\
	\Ext_{\bsp}^\bullet(\bsi_0^{(r)},\bsi_1^{(r)}) & \Ext_{\bsp}^\bullet(\bsi_1^{(r)},\bsi_1^{(r)})
	\end{pmatrix}
	\end{equation}
Conjugation by $\bspi$ defines an algebra involution on $\Ext_{\bsp}^\bullet(\bsir,\bsir)$, and via this isomorphism, one has $\Ext_{\bsp}^\bullet(\bsirzero,\bsirzero) \cong \Ext_{\bsp}^\bullet(\bsirone,\bsirone)$ and $\Ext_{\bsp}^\bullet(\bsirone,\bsirzero) \cong \Ext_{\bsp}^\bullet(\bsirzero,\bsirone)$. So it suffices to calculate the terms in the first row of \eqref{eq:matrixring}. For $r \in \N$, let $E_r = \bigoplus_{0 \leq i < p^r} k \subgrp{2i}$, i.e., $E_r$ is the purely even graded superspace that is equal to $k$ in $\Z$-degrees $2i$ for $0 \leq i < p^r$.

\begin{theorem} \label{thm:Ext(Ir,Ir)-vector-space}
Let $r$ be a positive integer. Then, as graded superspaces,
	\begin{align*}
	\Ext_{\bsp}^\bullet(\bsirone,\bsirone) \cong \Ext_{\bsp}^\bullet(\bsirzero,\bsirzero) \cong \Hom_{\bsp}(\bsirzero,J(r)^\bullet) &\cong \bigoplus_{n \geq 0} E_r\subgrp{2np^r}, \\
	\Ext_{\bsp}^\bullet(\bsirzero,\bsirone) \cong \Ext_{\bsp}^\bullet(\bsirone,\bsirzero) \cong \Hom_{\bsp}^\bullet(\bsirone,J(r)^\bullet) &\cong \bigoplus_{n \geq 0} E_r\subgrp{(2n+1)p^r}. %, \\
%	\Ext_{\bsp}^\bullet(\bsirone,\bsirone) \cong \Hom_{\bsp}(\bsirone,\Jbar(r)^\bullet) &\cong \bigoplus_{n \geq 0} E_r\subgrp{2np^r}, \\
%	\Ext_{\bsp}^\bullet(\bsirzero,\bsirone) \cong \Hom_{\bsp}^\bullet(\bsirzero,\Jbar(r)^\bullet) &\cong \bigoplus_{n \geq 0} E_r\subgrp{(2n+1)p^r}.
	\end{align*}
\end{theorem}

\begin{proof}
First consider $\Ext_{\bsp}^\bullet(\bsirzero,\bsirzero)$, which can be computed as the cohomology of the complex $\Hom_{\bsp}(\bsirzero,J(r))$. Ignoring the differential, $J(r)$ is a direct sum of shifts of $T(\bsi,r)$ and $\Tbar(\bsi,r)$, which are contractions of the $p$-complexes $\bsb_{p^r}(r)$ and $\bsbbar_{p^r}(r)$. As a graded superspace,
	\begin{equation} \label{eq:Hom-iso-Shar-twist}
	\Hom_{\bsp}(\bsirzero,\bsb_{p^r}(r)) = \Hom_{\bsp}(\bsirzero,\bss^{p^r}(\Sha_r \otimes -)) \cong (\bsirzero)^\#(\Sha_r) \cong \bsirzero(\Sha_r) = \Sha_r^{(r)}.
	\end{equation}
Since $\Sha_r^{(r)}$ is concentrated in $\Z$-degrees divisible by $p^r$, this implies that
	\begin{equation} \label{eq:Hom-iso-Er}
	\Hom_{\bsp}(\bsirzero,T(\bsi,r)) \cong E_r.
	\end{equation}
So $\Hom_{\bsp}(\bsirzero,T(\bsi,r))$ is concentrated in even cohomological degrees. On the other hand,
	\begin{align*}
	\Hom_{\bsp}(\bsirzero,\bsbbar_{p^r}(r)) &= \Hom_{\bsp}(\bsirzero,\bspi \circ \bss^{p^r}(\Shabar_r \otimes -)) \\
	&\cong \bspi \circ \Hom_{\bsp}(\bsirzero,\bss^{p^r}(\Shabar_r \otimes -)) & \\
	&\cong \bspi \circ (\bsirzero)^\# \left( \Shabar_r \right) \\
	&\cong \bspi \circ \bsirzero \left( \Shabar_r \right) = 0,
	\end{align*}
the last equality holding because $\Shabar_r$ is a purely odd superspace. Then $\Hom_{\bsp}(\bsirzero,\Tbar(\bsi,r)) = 0$. Then $\Hom_{\bsp}(\bsirzero,J(r))$ is concentrated in even cohomological degrees, and hence
	\[
	\Ext_{\bsp}^\bullet(\bsirzero,\bsirzero) = \Hom_{\bsp}(\bsirzero,J(r)) = \bigoplus_{n \geq 0} \Hom_{\bsp}(\bsirzero, T(\bsi,r)\subgrp{2np^r}) \cong \bigoplus_{n \geq 0} E_r\subgrp{2np^r}.
	\]
For $\Ext_{\bsp}^\bullet(\bsirone,\bsirzero)$, one similarly observes that
	\begin{align*}
	\Hom_{\bsp}(\bsirone,\bsb_{p^r}(r)) &= \Hom_{\bsp}(\bsirone,\bss^{p^r}(\Sha_r \otimes -)) \cong (\bsirone)^\#(\Sha_r) \cong \bsirone(\Sha_r) = 0, & \text{and} \\
	\Hom_{\bsp}(\bsirone,\bsbbar_{p^r}(r)) &\cong \bspi \circ \Hom_{\bsp}(\bsirone,\bss^{p^r}(\Shabar_r \otimes -)) \cong \bspi \circ (\bsirone)^\# \left( \Shabar_r \right) \cong \Sha_r^{(r)}.
	\end{align*}
So $\Hom_{\bsp}(\bsirone,T(\bsi,r)) = 0$, and $\Hom_{\bsp}(\bsirone,\Tbar(\bsi,r)) \cong E_r$. Then $\Hom_{\bsp}(\bsirone,J(r))$ is concentrated in odd cohomo\-logical degrees, and
	\begin{align*}
	\Ext_{\bsp}^\bullet(\bsirone,\bsirzero) = \Hom_{\bsp}(\bsirone,J(r)) &= \bigoplus_{n \geq 0} \Hom_{\bsp}(\bsirone, \Tbar(\bsi,r)\subgrp{2(n+1)p^r}) \\
	&\cong \bigoplus_{n \geq 0} E_r\subgrp{(2n+1)p^r}.
\qedhere	\end{align*}
%The arguments for the other isomorphisms in the theorem are entirely similar. Alternatively, they can be deduced from the results calculated already by applying the map in cohomology induced by the conjugation operation $F \mapsto \bspi \circ F \circ \bspi$; see \cite[\S 3.3.3]{Drupieski:2016}.
\end{proof}

Given an integer $0 \leq j < p^r$, let $\bse_r(j): \bsirzero \to J(r)^{2j}$ be the composite morphism
	\begin{equation} \label{eq:e_r(j)-definition}
	\bsirzero \to \bss^{p^r} \to \bss^{p^r}(\Sha_r \otimes -)^{p^r j} = T(\bsi,r)^{2j} \subseteq J(r)^{2j}
	\end{equation}
in which the first arrow is induced by the $p^r$-power map, $u^{(r)} \mapsto u^{p^r}$, and the second arrow is induced by the vector space map $U \to \Sha_r \otimes U$, $u \mapsto \sha_j \otimes u$. For arbitrary $j \in \N$, write $j = j_0 + j_1 p^r$ with $0 \leq j_0 < p^r$ and $j_1 \geq 0$, and then define $\bse_r(j) : \bsirzero \to J(r)^{2j}$ to be the composite
	\[
	\bse_r(j) : \bsirzero \to T(\bsi,r)^{2j_0}\subgrp{2j_1 p^r} \subseteq J(r)^{2j}
	\]
in which the first arrow is simply equal to $\bse_r(j_0)$ when you forget the shift in the cohomological grading. Set $\bserpi(j) = \bspi \circ \bse_r(j) \circ \bspi : \bsirone \to \Jbar(r)^{2j}$.

Next recall the decompositions \eqref{eq:J(r)-decomposition} and \eqref{eq:Jbar(r)-decomposition} of $J(r)$ and $\Jbar(r)$. Let $\bsc_r: \Jbar(r)^\bullet \to J(r)^{\bullet+p^r}$ be the morphism that maps summands of $\Jbar(r)$ identically onto summands of $J(r)$ as follows:
	\[
	\Tbar(\bsi,r)\subgrp{2np^r} \mapsto \Tbar(\bsi,r)\subgrp{(2n+1)p^r} \quad \text{and} \quad 
	T(\bsi,r)\subgrp{(2n+1)p^r} \mapsto T(\bsi,r)\subgrp{2(n+1)p^r}.
	\]
Then $\bsc_r$ is a chain map. Set $\bscrpi = \bspi \circ \bsc_r \circ \bspi : J(r)^\bullet \to \Jbar(r)^{\bullet+p^r}$. For $j = j_0 + j_1 p^r$ as above,
	\begin{equation} \label{eq:erj-factorization-with-c}
	\bse_r(j) = (\bsc_r \circ \bscrpi)^{j_1} \circ \bse_r(j_0) \quad \text{and} \quad \bserpi(j) = (\bscrpi \circ \bsc_r)^{j_1} \circ \bserpi(j_0).
	\end{equation}
The following result now follows immediately from Theorem \ref{thm:Ext(Ir,Ir)-vector-space} by dimension comparison.

\begin{proposition}
Let $r$ be a positive integer.
	\begin{enumerate}
	\item $\set{ \bse_r(j) : j \in \N}$ is a basis for $\Hom_{\bsp}(\bsirzero,J(r)) = \Ext_{\bsp}^\bullet(\bsirzero,\bsirzero)$.
	\item $\set{ \bserpi(j) : j \in \N}$ is a basis for $\Hom_{\bsp}(\bsirone,\Jbar(r)) = \Ext_{\bsp}^\bullet(\bsirone,\bsirone)$.
	\item $\set{\bsc_r \circ \bserpi(j) : j \in \N}$ is a basis for $\Hom_{\bsp}(\bsirone,J(r)) = \Ext_{\bsp}^\bullet(\bsirone,\bsirzero)$.
	\item $\set{\bscrpi \circ \bse_r(j) : j \in \N}$ is a basis for $\Hom_{\bsp}(\bsirzero,\Jbar(r)) = \Ext_{\bsp}^\bullet(\bsirzero,\bsirone)$.
	\end{enumerate}
\end{proposition}

\begin{remark}
By dimension comparison, one gets for $1 \leq i \leq r$ that the cohomology classes $\bse_r(p^{i-1})$ and $\bsc_r$ defined here are nonzero scalar multiples of the classes denoted $\bse_i'$ and $\bsc_r$ in \cite{Drupieski:2016}. For $j \in \N$, the class $\bse_r(j)$ is a nonzero scalar multiple of the class of the same name in \cite{Drupieski:2019b}.
\end{remark}

\subsection{Multiplicative structure of \texorpdfstring{$\Ext_{\bsp}^\bullet(\bsir,\bsir)$}{ExtP(Ir,Ir)}} \label{subsec:ExtP(Ir,Ir)-multiplication}

In this section we make some observations on the multiplicative structure of the Yoneda algebra $\Ext_{\bsp}^\bullet(\bsir,\bsir)$. Along the way we recall some facts from the classical (non-super) situation about the Yoneda algebra $\Ext_{\calP}^\bullet(I^{(r)},I^{(r)})$, whose multiplicative structure was determined by Friedlander and Suslin \cite[Theorem 4.10]{Friedlander:1997}.

First, from \cite[Theorem 4.7.1]{Drupieski:2016}, it follows for $0 \leq i < r$ that $\bsc_r \circ \bse_r^{\bspi}(p^i) = \pm \bse_r(p^i) \circ \bsc_r$. It was only later in \cite[Theorem 5.3.5]{Drupieski:2019b} that these structure constants were pinned down to $+1$, as a consequence of calculating how the generators of $\Ext_{\bsp}^\bullet(\bsir,\bsir)$ restrict to certain subgroups of $GL_{m|n}$. But now using the resolutions $J(r)$ and $\Jbar(r)$, we can give a much more direct proof:

\begin{proposition}
Let $r,i \in \N$ with $0 \leq i < r$. Then in the cohomology ring $\Ext_{\bsp}^\bullet(\bsir,\bsir)$,
	\[
	\bsc_r \circ \bse_r^{\bspi}(p^i) = \bse_r(p^i) \circ \bsc_r \quad \text{and} \quad \bscrpi \circ \bse_r(p^i) = \bserpi(p^i) \circ \bscrpi
	\]
\end{proposition}

\begin{proof}
Let $\whbse_r(p^i): J(r) \to J(r)\subgrp{2p^i}$ be a chain map lifting $\bse_r(p^i): \bsirzero \to J(r)^{2p^i}$. We will show that $\whbse_r(p^i)$ can be chosen so that $\bsc_r \circ \bse_r^{\bspi}(p^i) = \whbse_r(p^i) \circ \bsc_r \circ \bse_r^{\bspi}(0)$ as morphisms $\bsirone \to J(r)^{p^r + 2p^i}$. To construct $\whbse_r(p^i)$, one must first choose morphisms making the following diagram commute:
	\begin{equation} \label{eq:er(pi)-diagram}
	\vcenter{\xymatrix @R=2pc {
	\bsirzero \ar@{->}[r]^{\bse_r(0)} \ar@{->}[dr]_{\sm{\bse_r(p^i) \\ 0}} & J(r)^0 \ar@{->}[r]^-{\delta} \ar@{->}[d]^{\beta_0} & J(r)^1 \ar@{->}[r]^-{\delta} \ar@{->}[d]^{\beta_1} & \cdots \ar@{->}[r]^-{\delta} & J(r)^{p^r-1} \ar@{->}[r]^-{\delta} \ar@{->}[d]^{\beta_{p^r-1}} & J(r)^{p^r} \ar@{->}[d]^{\beta_{p^r}} \\
	& J(r)^{2p^i} \ar@{->}[r]_-{\delta} & J(r)^{2p^i+1} \ar@{->}[r]_-{\delta} & \cdots \ar@{->}[r]_-{\delta} & J(r)^{p^r + 2p^i - 1} \ar@{->}[r]_-{\delta} & J(r)^{p^r+2p^i}.
	}}
	\end{equation}
Since $p \geq 3$ and $0 \leq i < r$, then $p^r + 2p^i \leq 2p^r -1$ (with equality only if $p=3$ and $r=1$). Then for $0 \leq \ell \leq p^r + 2p^i$, one has $J(r)^\ell = T^\ell \oplus \Tbar^{\ell - p^r}$, where $\Tbar^i = 0$ if $i < 0$, and $T^\ell = 0$ if $\ell > 2p^r-2$ (the latter relevant only if $p=3$ and $r=1$). Then the differentials in \eqref{eq:er(pi)-diagram} are all of the form
	\[
	\delta = \bmat{\partial & 0 \\ \ve & \partialbar} : \left[ \begin{array}{l} T^\ell \\ \Tbar^{\ell - p^r} \end{array} \right] \to \left[ \begin{array}{l} T^{\ell+1} \\ \Tbar^{\ell - p^r+1} \end{array} \right],
	\]
and the unknown morphisms $\beta_0,\beta_1,\ldots,\beta_{p^r}$ are all of the form
	\[
	\beta_\ell = \bmat{a_\ell & b_\ell \\ c_\ell & d_\ell}: \left[ \begin{array}{l} T^\ell \\ \Tbar^{\ell - p^r} \end{array} \right] \to \left[ \begin{array}{l} T^{\ell+2p^i} \\ \Tbar^{\ell - p^r+2p^i} \end{array} \right].
	\]
Arguing by induction, one can show for $0 \leq \ell \leq p^r$ that $\beta_\ell$ can be chosen so that $a_\ell$ is the restriction of the algebra homomorphism $\bss(\rho_i \otimes 1): \bss(\Sha_r \otimes -) \to \bss(\Sha_r \otimes -)$, $b_\ell$ is zero, and $d_\ell$ is the restriction of $\bss(\rho_i \otimes 1)^{\bspi} := \bspi \circ \bss(\rho_i \otimes 1) \circ \bspi$. The fact that $\bss(\rho_i \otimes 1)$ commutes with the $p$-differential $d(r)$ on $\bsb(r)$ follows from Lemma \ref{lemma:convolution-components}\eqref{item:phi-psi-commute}. Now $\whbse_r(p^i) \circ \bsc_r \circ \bse_r^{\bspi}(0) : \bsirone \to J(r)^{p^r+2p^i}$ is the composite morphism
	\[
	\bsirone \xrightarrow{\bse_r^{\bspi}(0)} \Tbar^0 \xrightarrow{\sm{0 \\ \id}} \left[ \begin{array}{l} T^{p^r} \\ \Tbar^0 \end{array} \right] \xrightarrow{\sm{\bss(\rho_i \otimes 1) & 0 \\ c_{p^r} & \bss(\rho_i \otimes 1)^{\bspi}}} \left[ \begin{array}{l} T^{p^r+2p^i} \\ \Tbar^{2p^i} \end{array} \right],
	\]
which is immediately seen to be equal to the composite morphism
	\[
	\bsc_r \circ \bse_r^{\bspi}(p^i): \bsirone \xrightarrow{\bse_r^{\bspi}(p^i)} \Tbar^{2p^i} \xrightarrow{\sm{0 \\ \id}} \left[ \begin{array}{l} T^{p^r+2p^i} \\ \Tbar^{2p^i} \end{array} \right].
	\]
So $\bsc_r \circ \bse_r^{\bspi}(p^i) = \bse_r(p^i) \circ \bsc_r$ in the ring $\Ext_{\bsp}^\bullet(\bsir,\bsir)$. The relation $\bscrpi \circ \bse_r(p^i) = \bserpi(p^i) \circ \bscrpi$ then follows via the conjugation action of $\bspi$.
\end{proof}

Next, recall from Remark \ref{remark:restriction-to-non-super} the exact linear operation $F \mapsto F|_{\calV}$ of restriction from $\bsp$ to the category $\calP$ of strict polynomial functors. The $p$-complex $\bsb_{p^r}(r)$ restricts to the $p$-complex $B_{p^r}(r)$ constructed by Troesch, and the contraction $T(\bsi,r) = \bsb_{p^r}(r)_{[1,0]}$ restricts to the complex $T(I,r)$ of \cite{Touze:2012}, which is an injective resolution in $\calP$ of $I^{(r)}$. The canonical projection map $J(r) \twoheadrightarrow T(\bsi,r)\subgrp{0}$ then restricts to a chain map $J(r)|_{\calV} \to T(I,r)$. The results of \cite[\S4]{Touze:2012} show that
	\[
	E_r \cong \Hom_{\calP}(I^{(r)},T(I,r)^\bullet) = \Ext_{\calP}^\bullet(I^{(r)},I^{(r)}).
	\]
Specifically, for $0 \leq j < p^r$, the space $\Hom_{\calP}(I^{(r)},T(I,r)^{2j})$ is spanned by the morphism $e_r(j): I^{(r)} \to T(I,r)^{2j}$ that is defined in precisely the same manner as in \eqref{eq:e_r(j)-definition}. So for $0 \leq j < p^r$, $\bse_r(j)$ restricts to $e_r(j)$. This immediately implies:

\begin{proposition} \label{prop:restriction-iso-small-degree}
For $0 \leq j < 2p^r$, restriction from $\bsp$ to $\calP$ defines an isomorphism
	\begin{equation} \label{eq:restriction-iso}
	\Ext_{\bsp}^j(\bsirzero,\bsirzero) \cong \Ext_{\calP}^j(I^{(r)},I^{(r)}).
	\end{equation}
\end{proposition}

Lemma \ref{lemma:convolution-components} implies for $0 \leq i < r$ that the algebra homomorphism $S(\rho_i \otimes 1): S(\Sha_r \otimes -) \to S(\Sha_r \otimes -)$ defines a map of $p$-complexes $B_{p^r}(r)^{\bullet} \to B_{p^r}(r)^{\bullet+p^r \cdot p^i}$, which then restricts to a chain map $\whe_r(p^i) : T(I,r) \to T(I,r)\subgrp{2p^i}$ lifting the morphism $e_r(p^i): I^{(r)} \to T(I,r)^{2p^i}$. Yoneda products in $\Ext_{\calP}^\bullet(I^{(r)},I^{(r)})$ can be calculated by composing the corresponding lifted chain maps, and in this manner one can deduce the following identities in $\Ext_{\calP}^\bullet(I^{(r)},I^{(r)})$, where by abuse of notation we identify the homomorphism $e_r(j)$ with the corresponding cohomology class:
	\begin{itemize}
	\item The cohomology classes $e_r(p^0),e_r(p^1),\ldots,e_r(p^{r-1})$ commute in $\Ext_{\calP}^\bullet(I^{(r)},I^{(r)})$.
	\item Given $0 \leq j < p^r$, let $j = \sum_{\ell=0}^{r-1} j_\ell \cdot p^\ell$ be its base-$p$ decomposition. Then
		\[
		e_r(j) = e_r(p^0)^{j_0} \cdot e_r(p^1)^{j_1} \cdots e_r(p^{r-1})^{j_{r-1}}.
		\]
	\item For each $0 \leq i < r$, the $p$-fold Yoneda product $e_r(p^i)^p$ is equal to zero.
	\end{itemize}
The restriction map $\Ext_{\bsp}^\bullet(\bsirzero,\bsirzero) \to \Ext_{\calP}^\bullet(I^{(r)},I^{(r)})$ is an algebra homomorphism, so applying Proposition \ref{prop:restriction-iso-small-degree}, one immediately deduces:

\begin{proposition} \label{prop:ExtI0I0-basic-relations}
The following identities hold in $\Ext_{\bsp}^\bullet(\bsirzero,\bsirzero)$, where by abuse of notation we identify the morphism $\bse_r(j)$ with the corresponding cohomology class:
	\begin{enumerate}
	\item The cohomology classes $\bse_r(p^0),\bse_r(p^1),\ldots,\bse_r(p^{r-1})$ commute in $\Ext_{\bsp}^\bullet(\bsirzero,\bsirzero)$.
	\item Given $0 \leq j < p^r$, let $j = \sum_{\ell=0}^{r-1} j_\ell \cdot p^\ell$ be its base-$p$ decomposition. Then
		\[
		\bse_r(j) = \bse_r(p^0)^{j_0} \cdot \bse_r(p^1)^{j_1} \cdots \bse_r(p^{r-1})^{j_{r-1}}.
		\]
	\item \label{item:p-power-zero} For each $0 \leq i \leq r-2$, the $p$-fold Yoneda product $\bse_r(p^i)^p$ is equal to zero.
	\end{enumerate}
Replacing $\bse_r$ with $\bserpi$, the same identities hold in $\Ext_{\bsp}^\bullet(\bsirone,\bsirone)$.
\end{proposition}

The last assertion of the proposition follows via the conjugation action of $\bspi$.

In part \eqref{item:p-power-zero} of the proposition we get $\bse_r(p^i)^p = 0$ only for $0 \leq i \leq r-2$, because $\bse_r(p^{r-1})^p \in \Ext_{\bsp}^{2p^r}(\bsirzero,\bsirzero)$, and \eqref{eq:restriction-iso} is not an isomorphism in degree $2p^r$. In fact, the discussion preceding \cite[Proposition 4.4.4]{Drupieski:2016} implies the following:

\begin{proposition} \label{prop:bser-pth-power}
$\bse_r(p^{r-1})^p = \mu \cdot \bsc_r \circ \bscrpi$ and $\bserpi(p^{r-1})^p = \mu \cdot \bscrpi \circ \bsc_r$ for some $0 \neq \mu \in k$.
\end{proposition}

The full justification of Proposition \ref{prop:bser-pth-power} as given in \cite{Drupieski:2016} is circuitous, relying on calculating the restriction of $\bse_r(p^{r-1})^p$ to a subgroup of $GL_{m|n}$. But for $r=1$, we can now use the injective resolution $J(1)$, constructed using the map $\ve'(1)$ of Section \ref{subsec:epsilon-prime}, to give a direct proof.

\begin{proposition}
$\bse_1(1)^p = (-1)^{p(p-1)/2} \cdot \bsc_1 \circ \bsc_1^{\bspi}$ in $\Ext_{\bsp}^\bullet(\bsi^{(1)},\bsi^{(1)})$.
\end{proposition}

\begin{proof}
From \eqref{eq:erj-factorization-with-c} we have the equality of morphisms $\bse_1(p) = \bsc_1 \circ \bsc_1^{\bspi} \circ \bse_1(0) : \bsi_0^{(1)} \to J(1)^{2p}$, and we know by Proposition \ref{prop:ExtI0I0-basic-relations} that $\bse_1(1)^{p-1} = \bse_1(p-1)$ as classes in $\Ext_{\bsp}^\bullet(\bsi^{(1)},\bsi^{(1)})$. Note that $\bse_1(0)$  corresponds to the identity element of the ring $\Ext_{\bsp}^\bullet(\bsionezero,\bsionezero)$; it lifts to the identity map of complexes $J(1) \to J(1)$. Let $\whbse_1(p-1) : J(1) \to J(1)\subgrp{p-1}$ be a chain map lifting the morphism $\bse_1(p-1) : \bsi_0^{(1)} \to J(1)^{2(p-1)}$. To prove that $\bse_1(1)^p = (-1)^{p(p-1)/2} \cdot \bsc_1 \circ \bsc_1^{\bspi}$, we'll show that $\whbse_1(p-1)$ can be chosen so that $\whbse_1(p-1) \circ \bse_1(1) = (-1)^{p(p-1)/2} \cdot \bse_1(p)$ as morphisms $\bsionezero \to J(1)^{2p}$.

To begin to construct the chain map $\whbse_1(p-1)$, one must choose homomorphisms that make the following diagram commute:
	\[
	\vcenter{\xymatrix @C=5.5pc @R=2pc {
	\bsionezero \ar@{->}[r]^{\bse_1(0)} \ar@{->}[dr]_{\sm{\bse_1(p-1) \\ 0}} & T^0 \ar@{->}[r]^-{\partial_0} \ar@{->}[d]^{\sm{\beta_0 \\ 0}} & T^1 \ar@{->}[r]^{\partial_1} \ar@{->}[d]^{\beta_1} & T^2 \ar@{->}[d]^{\sm{\beta_{2,0}, & \beta_{2,p}}} \\
	& T^{2p-2} \oplus \Tbar^{p-2} \ar@{->}[r]_{\sm{\ve_{2p-2}, & \partialbar_{p-2}}} & \Tbar^{p-1} \ar@{->}[r]_{\sm{\vebar_{p-1} \\ \partialbar_{p-1}}} & T^0 \oplus \Tbar^p
	}}
	\]
First, $\beta_0$ is a morphism $\beta_0: \bsb_p(1)^0 \to \bsb_p(1)^{2p-2}$ such that $\beta_0 \circ \bse_1(0) = \bse_1(p-1)$. For this we can take the restriction of the superalgebra homomorphism $\bss(\rho^{p-1} \otimes 1) : \bss(\Sha_1 \otimes -) \to \bss(\Sha_1 \otimes -)$. Equivalently, $\beta_0$ corresponds as in \eqref{eq:SYoneda} to the morphism $\gamma_p(\rho^{p-1}) \in \bsg^p \Hom_k(\Sha_1,\Sha_1)$.

Next, $\beta_1$ is a morphism $\beta_1: \bsb_p(1)^1 \to \bsbbar_p(1)^{\binom{p}{2}}$ such that $\ve_{2p-2} \circ \beta_0 = \beta_1 \circ \partial = \beta_1 \circ d(1)$. Given a matrix unit $\alpha_{i,j}$ as in Section \ref{subsec:epsilon-prime}, one has $\alpha_{i,j} \circ \rho = \alpha_{i,j-1}$ and $\rhobar \circ \alpha_{i,j} = \alpha_{i+1,j}$, where by convention the matrix unit is zero if either $j-1 < 0$ or if $i+1 \geq p$. Then from \eqref{eq:ve-prime-2p-2} we get
	\[
	\ve_{2p-2} \circ \beta_0 = {}^\pi \left[ \alpha_{p-1,p-1} \cdot \alpha_{p-2,p-1} \cdots \alpha_{0,p-1} \right] \circ \gamma_p(\rho^{p-1}) = {}^\pi \left[ \alpha_{p-1,0} \cdot \alpha_{p-2,0} \cdots \alpha_{0,0} \right],
	\]
so under the correspondence \eqref{eq:HomB-to-Bbar} we can take $\beta_1 = {}^\pi \left[ \alpha_{p-1,0} \cdot \alpha_{p-2,0} \cdots \alpha_{1,0} \cdot \alpha_{0,1} \right]$. Given indices $0 \leq i,j < p$, let $\alphabar_{i,j} \in \Hom_k(\Shabar_1,\Sha_1)$ and $\kappa_{i,j} \in \Hom_k(\Sha_1,\Sha_1)$ be defined by $\alphabar_{i,j}(\shabar_\ell) = \delta_{j,\ell} \cdot \sha_i$ and $\kappa_{i,j}(\sha_\ell) = \delta_{j,\ell} \cdot \sha_i$, respectively. Then from \eqref{eq:ve-prime-p-1} we get
	\begin{align*}
	\vebar_{p-1} = (-1)^{p(p-1)/2} \cdot {}^\pi \left[ \alphabar_{0,0} \cdot  \alphabar_{0,1} \cdots \alphabar_{0,p-1} \right] = {}^\pi \left[ \alphabar_{0,p-1} \cdots \alphabar_{0,1} \cdot \alphabar_{0,0} \right],
	\end{align*}
and composing morphisms in $\bsg^p \bsv$ one gets
	\begin{align*}
	\vebar_{p-1} \circ \beta_1 &= \left( \alphabar_{0,p-1} \cdots \alphabar_{0,1} \cdot \alphabar_{0,0} \right) \circ \left( \alpha_{p-1,0} \cdots \alpha_{1,0} \cdot \alpha_{0,1} \right) \\
	&= (-1)^{0+1+\cdots+(p-1)} \cdot (\alphabar_{0,p-1} \circ \alpha_{p-1,0}) \cdots (\alphabar_{0,1} \circ \alpha_{1,0}) \cdot (\alphabar_{0,0} \circ \alpha_{0,1}) \\
	&= (-1)^{p(p-1)/2} \cdot \gamma_1(\kappa_{0,0}) \cdots \gamma_1(\kappa_{0,0}) \cdot \gamma_1(\kappa_{0,1}) \\
	&= (-1)^{p(p-1)/2} \cdot \gamma_1(\kappa_{0,0})^{p-1} \cdot \gamma_1(\kappa_{0,1}) \in \bsg^p \Hom_k(\Sha_1,\Sha_1).
	\end{align*}
The factor $(-1)^{0+1+\cdots+(p-1)}$ at the second equals sign arises because in the process of composing morphisms, each odd linear map $\alphabar_{0,i}$ passes over $(i-1)$ many other odd linear maps.

Finally, take $\beta_{2,0} = (-1)^{p(p-1)/2} \cdot \gamma_p(\kappa_{0,1}) \in \bsg^p \Hom_k(\Sha_1,\Sha_1)$. Arguing by induction, one can show for each $1 \leq i < p$ that $\gamma_p(\kappa_{0,1}) \circ d(1)^i = \gamma_1(\kappa_{0,0})^i \cdot \gamma_{p-i}(\kappa_{0,1})$. Then
	\[
	\beta_{2,0} \circ \partial_1 = \beta_{2,0} \circ d(1)^{p-1} = (-1)^{p(p-1)/2} \cdot \gamma_1(\kappa_{0,0})^{p-1} \cdot \gamma_1(\kappa_{0,1}) = \vebar_{p-1} \circ \beta_1,
	\]
as desired. Now consider the composite morphism $\whbse_1(p-1) \circ \bse_1(1): \bsionezero \to T^0 \oplus \Tbar^p$. Note that
	\[
	\Hom_{\bsp}(\bsionezero,\Tbar^p) \subseteq \Hom_{\bsp}(\bsionezero,\bsbbar_p(1)) \cong \bsionezero(\Shabar_1) = 0,
	\]
so we only need to consider the composite $\beta_{2,0} \circ \bse_1(1): \bsionezero \to T^0$, which one immediately verifies is equal to $(-1)^{p(p-1)/2} \cdot \bse_1(p)$.
\end{proof}

The following theorem summarizes the multiplicative structure of the ring $\Ext_{\bsp}^\bullet(\bsir,\bsir)$:

\begin{theorem} \label{theorem:Ext(Ir,Ir)-multiplicative}
The ring $\Ext_{\bsp}^\bullet(\bsir,\bsir)$ is generated as a $k$-algebra by the (even superdegree) extension classes
\[
\left.
\begin{aligned}
\bse_r(p^i) &\in \Ext_{\bsp}^{2p^i}(\bsi_0^{(r)},\bsi_0^{(r)}) \\
\bserpi(p^i) &\in \Ext_{\bsp}^{2p^i}(\bsi_1^{(r)},\bsi_1^{(r)})
\end{aligned}
\right\} \text{ for $0 \leq i < r$, and }
\left\{
\begin{aligned}
\bsc_r &\in \Ext_{\bsp}^{p^r}(\bsi_1^{(r)},\bsi_0^{(r)}), \\
\bscrpi &\in \Ext_{\bsp}^{p^r}(\bsi_0^{(r)},\bsi_1^{(r)}),
\end{aligned}
\right.
\]
subject only to the relations imposed by the matrix ring decomposition \eqref{eq:matrixring} and:
	\begin{enumerate}
	\item For each $0 \leq i < r$, $\bse_r(p^i) \circ \bsc_r = \bsc_r \circ \bserpi(p^i)$ and $\bserpi(p^i) \circ \bscrpi = \bscrpi \circ \bse_r(p^i)$.

	\item For each $0 \leq i \leq r -2$, $\bse_r(p^i)^p = [\bserpi(p^i)]^p = 0$.

	\item The subalgebras $\Ext_{\bsp}^\bullet(\bsirzero,\bsirzero)$ and $\Ext_{\bsp}^\bullet(\bsirone,\bsirone)$ are commutative.

	\item \label{item:pth-power} $\bse_r(p^{r-1})^p = \mu \cdot \bsc_r \circ \bscrpi$ and $[\bserpi(p^{r-1})]^p = \mu \cdot \bscrpi \circ \bsc_r$ for some common nonzero scalar $\mu$.
	\end{enumerate}
\end{theorem}

\begin{remark}
The only part of Theorem \ref{theorem:Ext(Ir,Ir)-multiplicative} for which we haven't given a standalone proof in this paper is part \eqref{item:pth-power} in the case $r > 1$. Note that if $\bse_r(p^{r-1})^p = \mu \cdot \bsc_r \circ \bscrpi$ for some $0 \neq \mu \in k$, then the conjugation action of $\bspi$ implies that $[\bserpi(p^{r-1})]^p = \mu \cdot \bscrpi \circ \bsc_r$ for the same scalar. Assuming that the field $k$ is perfect, so that $\mu^{1/p} \in k$, one can replace $\bse_r(p^{r-1})$ with the scalar multiple $\bse_r := \mu^{-1/p} \cdot \bse_r(p^{r-1})$, and then the relations become $\bse_r^p = \bsc_r \circ \bscrpi$ and $(\bserpi)^p = \bscrpi \circ \bsc_r$. This more closely matches the presentations given in \cite{Drupieski:2016,Drupieski:2019b}.
\end{remark}

\section{Calculations after Franjou, Friedlander, Scorichenko, and Suslin} \label{section:FFSS}

\subsection{Setup}

For $1 \leq j \leq r$ and $\ell \in \set{0,1}$, set
	\begin{align*}
	V_{j,\ell} &= \Ext_{\bsp}^\bullet(\bsirell,S_0^{p^{r-j}(j)}), & W_{j,\ell} &= \Ext_{\bsp}^\bullet(\bsirell,\Lambda_0^{p^{r-j}(j)}), \\
	\Vbar_{j,\ell} &= \Ext_{\bsp}^\bullet(\bsirell,S_1^{p^{r-j}(j)}), & \Wbar_{j,\ell} &= \Ext_{\bsp}^\bullet(\bsir_\ell,\Lambda_1^{p^{r-j}(j)}).
	\end{align*}
The conjugation action of $\bspi$ defines even isomorphisms $V_{j,\ell} \cong \Vbar_{j,\ell+1}$ and $W_{j,\ell} \cong \Wbar_{j,\ell+1}$. The calculation of these graded superspaces is given (either directly, or via conjugation by $\bspi$) by:

\begin{theorem}[{\cite[Theorem 4.5.1]{Drupieski:2016}}] \label{D2016-4.5.1}
For all $1 \leq j \leq r$, one has
	\begin{align*}
	\Ext_{\bsp}^s(\bsirzero,S_0^{p^{r-j}(j)}) \cong \Ext_{\bsp}^{s+p^{r-j}-1}(\bsirzero,\Lambda_0^{p^{r-j}(j)}) &\cong \begin{cases} k & \text{if $s \equiv 0 \mod 2p^{r-j}$ and $s \geq 0$,} \\ 0 & \text{otherwise.} \end{cases} \\
	\Ext_{\bsp}^s(\bsirone,S_0^{p^{r-j}(j)}) \cong \Ext_{\bsp}^{s+p^{r-j}-1}(\bsirone,\Lambda_0^{p^{r-j}(j)}) &\cong \begin{cases} k & \text{if $s \equiv p^r \mod 2p^{r-j}$ and $s \geq p^r$,} \\ 0 & \text{otherwise.} \end{cases}
	\end{align*}
\end{theorem}

Given a $\bsp$-coalgebra $A$ (i.e., a coalgebra object in the category $\bsp_{\ev}$) and a $\bsp$-algebra $B$ (i.e., an algebra object in the category $\bsp_{\ev}$), there exists a product operation on extension groups
	\begin{equation} \label{eq:cup-product}
	\Ext_{\bsp}^m(A,B) \otimes \Ext_{\bsp}^n(A,B) \to \Ext_{\bsp}^{m+n}(A \otimes A, B \otimes B) \to \Ext_{\bsp}^{m+n}(A,B)
	\end{equation}
that we call the cup product; see \cite[\S3.4]{Drupieski:2016} for details. We are interested in these products when $A = \Gamma_\ell^{(r)} = \Gamma \circ \bsirell$ and either $B = S_\ell^{(j)} = S \circ \bsijell$ or $B = \Lambda_\ell^{(j)} = \Lambda \circ \bsijell$. Our goal is to prove:

\begin{theorem} \label{theorem:FFSS-4.5}
Let $\ell \in \set{0,1}$. For all $d \geq 1$ and all $1 \leq j \leq r$, the cup product maps
	\begin{equation} \label{eq:cup-products}
	\begin{aligned}
	(V_{j,\ell})^{\otimes d} &\to \Ext_{\bsp}^\bullet(\Gamma_\ell^{d(r)},S_0^{dp^{r-j}(j)}), \qquad & (W_{j,\ell})^{\otimes d} &\to \Ext_{\bsp}^\bullet(\Gamma_\ell^{d(r)},\Lambda_0^{dp^{r-j}(j)}), \\
	(\Vbar_{j,\ell})^{\otimes d} &\to \Ext_{\bsp}^\bullet(\Gamma_\ell^{d(r)},S_1^{dp^{r-j}(j)}), \qquad & (\Wbar_{j,\ell})^{\otimes d} &\to \Ext_{\bsp}^\bullet(\Gamma_\ell^{d(r)},\Lambda_1^{dp^{r-j}(j)})
	\end{aligned}
	\end{equation}
factor to induce isomorphisms of graded vector spaces
	\begin{equation} \label{eq:cup-products-factored}
	\begin{aligned}
	S^d(V_{j,\ell}) &\cong \Ext_{\bsp}^\bullet(\Gamma_\ell^{d(r)},S_0^{dp^{r-j}(j)}), \qquad & \Lambda^d(W_{j,\ell}) &\cong \Ext_{\bsp}^\bullet(\Gamma_\ell^{d(r)},\Lambda_0^{dp^{r-j}(j)}), \\
	S^d(\Vbar_{j,\ell}) &\cong \Ext_{\bsp}^\bullet(\Gamma_\ell^{d(r)},S_1^{dp^{r-j}(j)}), \qquad & \Lambda^d(\Wbar_{j,\ell}) &\cong \Ext_{\bsp}^\bullet(\Gamma_\ell^{d(r)},\Lambda_1^{dp^{r-j}(j)}).
	\end{aligned}
	\end{equation}
\end{theorem}

The cup products factor through the indicated symmetric and exterior powers by \cite[Lemma 3.4.1]{Drupieski:2016}, so the content of Theorem \ref{theorem:FFSS-4.5} is in proving that the factored maps are isomorphisms. The proof of Theorem \ref{theorem:FFSS-4.5}, presented beginning in Section \ref{subsec:proof-j=1}, mimics the triple induction argument given in \cite[\S4]{Franjou:1999}, arguing first by induction on the number $j$ of Frobenius twists, then by induction on $d$, and lastly by induction on the cohomological degree. The induction step for $j$ is accomplished through the use of hypercohomology spectral sequences.

Let $C$ be a bounded below cochain complex in $\bsp_{\ev}$, and let $A \in \bsp$, considered as a chain complex concentrated in homological degree zero. Recall from \cite[\S3.6]{Drupieski:2016} that there exist two spectral sequences converging to the hypercohomology group $\bbExt_{\bsp}^\bullet(A,C)$:
	\begin{align}
	\rmI_1^{s,t} = \rmI_1^{s,t}(A,C) &= \Ext_{\bsp}^t(A,C^s) \Rightarrow \bbExt_{\bsp}^{s+t}(A,C), & \rmI(A,C) \label{eq:first-hyper} \\
	\rmII_2^{s,t} = \rmII_2^{s,t}(A,C) &= \Ext_{\bsp}^s(A,\opH^t(C)) \Rightarrow \bbExt_{\bsp}^{s+t}(A,C). & \rmII(A,C) \label{eq:second-hyper}
	\end{align}
These are the spectral sequences associated to applying the functor $\RHom_{\bsp}(A,-)$ to the cochain complex $C$. If $C$ is concentrated in cohomology degree $0$, then $\Ext_{\bsp}^\bullet(A,C) = \bbExt_{\bsp}^\bullet(A,C)$.

In the classical case of \cite[\S4]{Franjou:1999}, the spectral sequences are applied for $C = \Omega \circ I^{(j-1)}$, a Frobenius twist of the De Rham complex functor. For $j \geq 2$, the composition $\Omega \circ \bsi_0^{(j-1)}$ defines a strict polynomial superfunctor by the discussion of Section \ref{subsubsec:Frobenius}, and we can directly imitate the arguments from the classical case. For $j=1$, we could consider the super De Rham complex functor $\bso$, but its cohomology is spread over a wider range of cohomology degrees than in the classical case (see \cite[Remark 4.1.3]{Drupieski:2016}), which makes the second hypercohomology spectral sequence \eqref{eq:second-hyper} unweildy. So to handle the case $j=1$, we instead consider $C = T(\bss^{dp^{r-1}},1)$. This complex consists of injective objects, so the first hypercohomology spectral sequence \eqref{eq:first-hyper} collapses at the first page, and the cohomology of $T(\bss^{dp^{r-1}},1)$ is relatively simple, which makes the second hypercohomology spectral sequence \eqref{eq:second-hyper} manageable to deal with.

The induction argument on $j$ makes use of the following proposition:

\begin{proposition}[{\cite[Proposition 4.1]{Franjou:1999}}] \label{prop:FFSS-4.1}
Let $T_1^\bullet \in \bsp_{d_1 p^r}$ and $T_2^\bullet \in \bsp_{d_2 p^r}$ be bounded below cochain complexes. If $d = d_1 + d_2$, then the first (resp.\ second) spectral sequence associated to applying the functor $\RHom_{\bsp}(\Gamma_\ell^{d(r)},-)$ to the complex $T_1^\bullet \otimes T_2^\bullet$ is naturally isomorphic to the tensor product of the first (resp.\ second) spectral sequences associated to applying $\RHom_{\bsp}(\Gamma_\ell^{d_i(r)},-)$ to $T_i^\bullet$.
\end{proposition}

\begin{proof}
Same as for \cite[Proposition 4.1]{Franjou:1999}, using \cite[Theorem 3.4.2]{Drupieski:2016} in lieu of \cite[Theorem 1.7]{Franjou:1999}.
\end{proof}

Let $(\Omega,d)$ and $(\Kz,\kappa)$ be the De Rham and Koszul complex functors, respectively. So $\Omega_n^i = S^{n-i} \otimes \Lambda^i = \Kz_n^{-i}$, and $d$ and $\kappa$ are natural transformations
	\begin{gather}
	d: \Omega_n^i = S^{n-i} \otimes \Lambda^i \xrightarrow{\Delta \otimes 1} S^{n-i-1} \otimes I \otimes \Lambda^i \xrightarrow{1 \otimes m} S^{n-i-1} \otimes \Lambda^{i+1} = \Omega_n^{i+1}, \\
	\kappa: \Kz_n^{-i} = S^{n-i} \otimes \Lambda^i \xrightarrow{1 \otimes \Delta} S^{n-i} \otimes I \otimes \Lambda^{i-1} \xrightarrow{m \otimes 1} S^{n-i+1} \otimes \Lambda^{i-1} = \Kz_n^{-i+1} \label{eq:Koszul-complex}
	\end{gather}
defined in terms of the product and coproduct operations on $S$ and $\Lambda$. (Note that our indexing of $\Kz_n^\bullet$ makes it a bounded below, non-positively graded cochain complex, with $\kappa$ of cohomological degree $+1$. This differs slightly from the convention in \cite{Franjou:1999}.) Then for each $j \geq 1$, we can consider $\Omega \circ \bsijzero$ and $\Kz \circ \bsijzero$ as complexes of strict polynomial superfunctors.

\begin{convention}
We will simply write $\Omega^{(j)}$ and $\Kz^{(j)}$ rather than $\Omega \circ \bsijzero$ or $\Kz \circ \bsijzero$.
\end{convention}

\begin{corollary}[cf.\ {\cite[Corollary 4.2]{Franjou:1999}}] \label{cor:FFSS-4.2}
Let $1 \leq j \leq r$, and let $\ell \in \set{0,1}$.
\begin{enumerate}
\item Suppose $j \geq 2$. All differentials of the spectral sequence $\rmI(\Gamma_\ell^{d(r)},(\Omega_{p^{r-j+1}}^{\bullet (j-1)})^{\otimes d})$ are trivial.

\item Suppose $j \geq 2$ (resp.\ $j \geq 1$). In the spectral sequence
	\[
	\rmII(\Gamma_\ell^{d(r)},(\Omega_{p^{r-j+1}}^{\bullet (j-1)})^{\otimes d}) \cong \rmII(\Gamma_\ell^{d(r)},\Omega_{p^{r-j+1}}^{\bullet (j-1)})^{\otimes d}
	\]
(resp.\ in the spectral sequence $\rmI(\Gamma_\ell^{d(r)},(\Kz_{p^{r-j}}^{\bullet(j)})^{\otimes d}) \cong \rmI(\Gamma_\ell^{d(r)},\Kz_{p^{r-j}}^{\bullet(j)})^{\otimes d}$), all differentials except for $d_{p^{r-j}+1}$ (resp.\ except for $d_{p^{r-j}}$) are trivial. The only nontrivial differential sends a homogeneous element $u_1 \otimes \cdots \otimes u_d$ to
	\begin{equation} \label{eq:differential-derivation} \textstyle
	\sum_{i=1}^d (-1)^{\sigma(i)} u_1 \otimes \cdots \otimes \partial(u_i) \otimes \cdots \otimes u_d,
	\end{equation}
where $\partial$ is the only nontrivial differential in $\rmII(\bsirell,\Omega_{p^{r-j+1}}^{\bullet(j-1)})$ (resp.\ in $\rmI(\bsirell,\Kz_{p^{r-j}}^{\bullet(j)})$), and $\sigma(i)$ denotes the number of terms of odd total degree among $u_1,\ldots,u_{i-1}$.
\end{enumerate}
\end{corollary}

\begin{proof}
This follows directly from Proposition \ref{prop:FFSS-4.1} and the calculation of the nontrivial differentials in the spectral sequences $\rmI(\bsirell,\Omega_{p^{r-j+1}}^{\bullet (j-1)})$, $\rmII(\bsirell,\Omega_{p^{r-j+1}}^{\bullet(j-1)})$, and $\rmI(\bsirell,\Kz_{p^{r-j}}^{\bullet(j)})$. The differentials in $\rmI(\bsirell,\Omega_{p^{r-j+1}}^{\bullet (j-1)})$ were all determined to be zero in the proof of \cite[Lemma 4.5.2]{Drupieski:2016} (see the bottom of page 1407, noting that a shift in $j$ is required), and the result for $\rmII(\bsirell,\Omega_{p^{r-j+1}}^{\bullet(j-1)})$ follows from the observation in the middle of \cite[p.\ 1408]{Drupieski:2016} that $\rmII_2^{s,t} = 0$ unless $t = 0$ or $t = p^{r-j}$. Finally, since $\Kz^{(j)} \cong S^{(j)} \otimes \Lambda^{(j)}$, \cite[Theorem 3.4.3]{Drupieski:2016} implies that $\rmI_1^{s,t}(\bsirell,\Kz_{p^{r-j}}^{\bullet(j)}) = 0$ unless $s = -p^{r-j}$ or $s = 0$. Then the only nontrivial differential in $\rmI(\bsirell,\Kz_{p^{r-j}}^{\bullet(j)})$ is $d_{p^{r-j}}$.
\end{proof}

\begin{remark} \label{remark:Koszul-isomorphism}
Since the Koszul complex $\Kz$ is exact, one gets $\rmII_2^{s,t}(\bsirell,\Kz_{p^{r-j}}^{\bullet(j)}) = 0$, and hence $\rmI(\bsirell,\Kz_{p^{r-j}}^{\bullet(j)}) \Rightarrow 0$. This implies that the differential $d_{p^{r-j}}$ in $\rmI(\bsirell,\Kz_{p^{r-j}}^{\bullet(j)})$ defines isomorphisms
	\begin{equation} \label{eq:theta}
	\theta: \Ext_{\bsp}^t(\bsirell,\Lambda_0^{p^{r-j}(j)}) = \rmI_{p^{r-j}}^{-p^{r-j},t} \xrightarrow{\cong} \rmI_{p^{r-j}}^{0,t-(p^{r-j}-1)} = \Ext_{\bsp}^{t-(p^{r-j}-1)}(\bsirell,S_0^{p^{r-j}(j)}).
	\end{equation}
\end{remark}

In proving Theorem \ref{theorem:FFSS-4.5} for $j \geq 2$, we'll apply Proposition \ref{prop:FFSS-4.1} and Corollary \ref{cor:FFSS-4.2} while considering the natural product map $(\Omega_{p^{r-j+1}}^{\bullet (j-1)})^{\otimes d} \to \Omega_{d p^{r-j+1}}^{\bullet(j-1)}$ (which is a map of chain complexes) and the induced morphism of spectral sequences. To prove the case $j=1$ of Theorem \ref{theorem:FFSS-4.5}, we'll consider a morphism of complexes $T(\bss^{p^{r-1}},1)^{\otimes d} \to T(\bss^{d p^{r-1}},1)$ that arises from a construction by Touz\'{e} \cite{Touze:2010a}. Let $C$ and $D$ be two $p$-complexes whose $p$-differentials each raise the $\Z$-degrees of elements by $\alpha = 1$. Let $p(C,D) = \bigoplus_{i,j \geq 0} C^{pi} \otimes D^{pj}$, considered as a graded superspace with $C^{pi} \otimes D^{pj}$ in $\Z$-degree $2(i+j)$.

\begin{lemma}[{\cite[Lemma 2.2]{Touze:2010a}}] \label{lemma:Touze-2.2}
Let $C$ and $D$ be two $p$-complexes whose $p$-differentials each raise $\Z$-degrees by $\alpha = 1$. For all $i,j \in \N$, the object $C^{pi} \otimes D^{pj}$ appears exactly once, in cohomological degree $2(i+j)$, in each of the contracted complexes $(C \otimes D)_{[1,0]}$ and $C_{[1,0]} \otimes D_{[1,0]}$.
\end{lemma}

\begin{proposition}[{\cite[Proposition 2.4]{Touze:2010a}}] \label{prop:Touze-2.4}
Let $C$ and $D$ be two $p$-complexes whose $p$-differentials each raise $\Z$-degrees by $\alpha = 1$. There is a morphism of ordinary complexes $h_{C,D}: (C_{[1,0]} \otimes D_{[1,0]}) \to (C \otimes D)_{[1,0]}$ with the following properties:
\begin{enumerate}
\item $h_{C,D}$ is natural with respect to the $p$-complexes $C$ and $D$.
\item $h_{C,D}^0$ and $h_{C,D}^1$ are the identity maps.
\item There is a commutative diagram of graded objects:
	\[
	\xymatrix{
	(C_{[1,0]} \otimes D_{[1,0]})^\bullet \ar@{->}[r]^{h_{C,D}^\bullet} & (C \otimes D)_{[1,0]}^\bullet \\
	p(C,D)^\bullet \ar@{=}[r] \ar@{^(->}[u] & p(C,D)^\bullet. \ar@{^(->}[u]
	}
	\]
\end{enumerate}
\end{proposition}

Iterating Proposition \ref{prop:Touze-2.4}, one gets a morphism of complexes $(\bsb_{p^r}(1)_{[1,0]})^{\otimes d} \to (\bsb_{p^r}(1)^{\otimes d})_{[1,0]}$. The product morphism $\bsb_{p^r}(1)^{\otimes d} = \bss^{p^r}(\Sha_1 \otimes -)^{\otimes d} \to \bss^{d p^r}(\Sha_1 \otimes -) = \bsb_{dp^r}(1)$ is a morphism of $p$-complexes by Lemma \ref{lemma:convolution-components}\eqref{item:phiderivation}, and hence restricts to a morphism $(\bsb_{p^r}(1)^{\otimes d})_{[1,0]} \to \bsb_{dp^r}(1)_{[1,0]}$ between the contracted complexes. Composing these morphisms, one gets a chain map
	\begin{equation} \label{eq:chain-map-h}
	h: T(\bss^{p^{r-1}},1)^{\otimes d} = (\bsb_{p^r}(1)_{[1,0]})^{\otimes d} \to \bsb_{dp^r}(1)_{[1,0]} = T(\bss^{dp^{r-1}},1).
	\end{equation}

\subsection{The case \texorpdfstring{$j=1$}{j=1} of Theorem \ref{theorem:FFSS-4.5}} \label{subsec:proof-j=1}

The proof of Theorem \ref{theorem:FFSS-4.5} in the case $j=1$ is by induction on $d$. Assume that $d \geq 1$, and that the theorem is true for all $1 \leq d' < d$.

We first consider the hypercohomology spectral sequence $\rmI(\Gamma_\ell^{d(r)},T(\bss^{dp^{r-1}},1))$. The complex $T(\bss^{dp^{r-1}},1)$ consists of injectives, so $\rmI_1^{n,m} = 0$ for $m > 0$, and hence
	\begin{equation} \label{eq:I-collapse}
	\bbExt_{\bsp}^n(\Gamma_\ell^{d(r)},T(\bss^{dp^{r-1}},1)) \cong \Hom_{\bsp}(\Gamma_\ell^{d(r)}, T(\bss^{dp^{r-1}},1)^n) \cong \begin{cases} 
	\left[S^d(E_1^{(r-1)}) \right]^n & \text{if $\ell = 0$,} \\
	0 & \text{if $\ell = 1$.}
	\end{cases}
	\end{equation}
Here $E_1 = \bigoplus_{0 \leq i < p} k\subgrp{2i}$ is the graded superspace defined just before Theorem \ref{thm:Ext(Ir,Ir)-vector-space}, and $[-]^m$ denotes the component of $\Z$-degree $m$ of a graded superspace. The last identification in \eqref{eq:I-collapse} is by reasoning similar to that for \eqref{eq:Hom-iso-Shar-twist} and \eqref{eq:Hom-iso-Er}. One can check that \eqref{eq:I-collapse} is multiplicative in the sense that the morphism
	\begin{multline*}
	[E_1^{(r-1)}]^{\otimes d} \cong \bbExt_{\bsp}^\bullet(\bsirzero,T(\bss^{p^{r-1}},1))^{\otimes d} \cong \bbExt_{\bsp}^\bullet(\Gamma_0^{d(r)},T(\bss^{p^{r-1}},1)^{\otimes d}) \\
	\xrightarrow{h^*} \bbExt_{\bsp}^\bullet(\Gamma_\ell^{d(r)},T(\bss^{dp^{r-1}},1)) \cong S^d(E_1^{(r-1)})
	\end{multline*}
induced by the chain map \eqref{eq:chain-map-h} is a surjection.

Next consider the second hypercohomology spectral sequence $\rmII(\Gamma_\ell^{d(r)},T(\bss^{dp^{r-1}},1))$. Applying Corollary \ref{cor:T(Sn,r)-cohomology}, \cite[Theorem 3.4.2]{Drupieski:2016}, and the fact that the exponential superfunctor $\Gamma \circ \bsirell$ is zero in polynomial degrees not divisible by $p^r$, one gets
	\begin{multline} \label{eq:II2-iso}
	\rmII_2^{n,m} \cong \\
	\begin{cases} 
	\left[\Ext_{\bsp}^\bullet(\Gamma_\ell^{d-s(r)},S_0^{(d-s)p^{r-1}(1)}) \otimes \Ext_{\bsp}^\bullet(\Gamma_\ell^{s(r)},\Lambda_1^{sp^{r-1}(1)}) \right]^n & \text{if $m = s p^{r-1} (p-1)$, $0 \leq s \leq d$,} \\
	0 & \text{if $m \not\equiv 0 \mod p^{r-1}(p-1)$.}
	\end{cases}
	\end{multline}
This implies that the first nontrivial differential is $d_{p^{r-1}(p-1)+1}$, and $\rmII_2^{n,m} \cong \rmII_{p^{r-1}(p-1)+1}^{n,m}$.

Set $q = p^{r-1}(p-1)$, and suppose for the moment that $d = 1$. Then the only nonzero rows of $\rmII(\bsirell,T(\bss^{p^{r-1}},1))$ are $m = 0$ and $m = q$, and the only nontrivial differential is
	\begin{equation} \label{eq:nontrivial-differential}
	\Ext_{\bsp}^n(\bsirell,\Lambda_1^{p^{r-1}(1)}) \cong \rmII_{q+1}^{n,q} \xrightarrow{d_{q+1}} \rmII_{q+1}^{n+q+1,0} \cong \Ext_{\bsp}^{n+q+1}(\bsirell,S_0^{p^{r-1}(1)}).
	\end{equation}
One has $\rmII_\infty \cong E_1^{(r-1)} \cong \bigoplus_{0 \leq i < p} k\subgrp{2ip^{r-1}}$ as a graded superspace by \eqref{eq:I-collapse}, so a straightforward induction argument on the cohomological degree shows that \eqref{eq:nontrivial-differential} is an isomorphism for all $n \geq 0$. Then we can interpret \eqref{eq:nontrivial-differential} as a map of graded superspaces
	\[
	f_\ell: \Wbar_{1,\ell} \to V_{1,\ell}
	\]
that increases $\Z$-degrees by $q+1$, with $\ker(f_\ell) = 0$, and
	\[
	\coker(f_\ell) \cong C_{1,\ell} := \begin{cases}
	\left[ \bigoplus_{0 \leq s < 2p^r} \Ext_{\bsp}^s(\bsirell,S_0^{p^{r-1}(1)}) \right] \cong E_1^{(r-1)} & \text{if $\ell = 0$,} \\
	0 & \text{if $\ell = 1$.}
	\end{cases}
	\]
Note that Theorem \ref{D2016-4.5.1} implies that $V_{1,\ell} \cong \rmII_2^{\bullet,0}$ (resp.\ $\Wbar_{1,\ell} \cong \rmII_2^{\bullet,q}$) is concentrated in terms whose total degree is congruent to $\ell \mod 2$ (resp.\ congruent to $(\ell+1) \mod 2$).

Now suppose $d > 1$, and set $Q_d^s(f_\ell) = S^{d-s}(V_{1,\ell}) \otimes \Lambda^s(\Wbar_{1,\ell})$. Then the map $f_\ell : \Wbar_{1,\ell} \to V_{1,\ell}$ defines on $Q(f_\ell)$ a Koszul differential $\kappa(f_\ell)$ (of cohomological degree $-1$), and makes $Q(f_\ell)$ into a generalized Koszul complex in the sense of \cite[\S4]{Franjou:1999}.\footnote{For these generalized Koszul complexes, we follow the notational conventions of \cite{Franjou:1999}, rather than the indexing used in \eqref{eq:Koszul-complex} for ``the'' Koszul complex. This is inconsistent of us, but we are trying to follow the conventions of \cite{Franjou:1999} except where changes seem absolutely necessary.} Since $\ker(f_\ell) = 0$, we get from \cite[Lemma 4.3]{Franjou:1999} that $\opH^i(Q(f_\ell)) = 0$ for $i > 0$, and the inclusion $C_{1,\ell} \subseteq V_{1,\ell}$ induces an algebra isomorphism
	\begin{equation} \label{eq:generalized-Koszul-homology}
	\opH^0(Q(f_\ell)) \cong S(C_{1,\ell}),
	\end{equation}
which is compatible with the `internal' $\Z$-grading on $Q(f_\ell)$ inherited from the (cohomological) $\Z$-gradings on $V_{1,\ell}$ and $\Wbar_{1,\ell}$.

Returning to the spectral sequence $\rmII(\Gamma_\ell^{d(r)},T(\bss^{dp^{r-1}},1))$, and continuing to write $q = p^{r-1}(p-1)$, we claim that there is a commutative diagram
	\begin{equation} \label{eq:Q-II-diagram}
	\vcenter{\xymatrix{
	Q_d^d(f_\ell)^\bullet \ar@{->}[r]^-{\kappa_d} \ar@{->}[d] & Q_d^{d-1}(f_\ell)^\bullet \ar@{->}[r]^-{\kappa_{d-1}} \ar@{->}[d]^{\cong} & \cdots \ar@{->}[r]^-{\kappa_3} & Q_d^2(f_\ell)^\bullet \ar@{->}[r]^-{\kappa_2} \ar@{->}[d]^{\cong} & Q_d^1(f_\ell)^\bullet \ar@{->}[r]^-{\kappa_1} \ar@{->}[d]^{\cong} & Q_d^0(f_\ell)^\bullet \ar@{->}[d] \\
	\rmII_{q+1}^{\bullet,dq} \ar@{->}[r]^-{d_{q+1}} & \rmII_{q+1}^{\bullet,(d-1)q} \ar@{->}[r]^-{d_{q+1}} & \cdots \ar@{->}[r]^-{d_{q+1}} & \rmII_{q+1}^{\bullet,2q} \ar@{->}[r]^-{d_{q+1}} & \rmII_{q+1}^{\bullet,q} \ar@{->}[r]^-{d_{q+1}} & \rmII_{q+1}^{\bullet,0}.
	}}
	\end{equation}
Here $Q(f_\ell)^m$ denotes the component of internal degree $m$ in $Q(f_\ell)$. The vertical arrows in \eqref{eq:Q-II-diagram} are induced by the cup product maps
	\[
	S(V_{1,\ell}) \otimes \Lambda(\Wbar_{1,\ell}) \to \Ext_{\bsp}^\bullet(\Gamma_\ell^{*(r)},S_0^{*p^{r-1}(1)}) \otimes \Ext_{\bsp}^\bullet(\Gamma_\ell^{*(r)},\Lambda_1^{*p^{r-1}(1)}) \cong \Ext_{\bsp}^\bullet(\Gamma_\ell^{(r)},\bss^{(1)}).
	\]
By induction on $d$, the cup product $Q_d^s(f_\ell)^\bullet \to \rmII_{q+1}^{\bullet,s}$ is an isomorphism for $0 < s < d$. The map $\kappa_s: Q_d^s(f_\ell) \to Q_d^{s-1}(f_\ell)$ is defined as follows: Let $v_1,\ldots,v_{d-s} \in V_{1,\ell}$, and $w_1,\ldots,w_s \in \Wbar_{1,\ell}$. Then
	\[
	\kappa_s(v_1 \cdots v_{d-s} \otimes w_1 \wedge \cdots \wedge w_s) = \sum_{i=1}^s (-1)^{\tau_s(i)} v_1 \cdots v_{d-s} \cdot f_\ell(w_i) \otimes w_1 \wedge \cdots \wedge \wh{w}_i \wedge \cdots w_s,
	\]
where $\tau_s(i) = i-1$ if $\ell = 0$, and $\tau_s(i) = (d-s)+(i-1)$ if $\ell = 1$. So $\kappa_s$ is equal to the Koszul differential $\kappa(f_\ell)$ if $\ell = 0$, but $\kappa_s = (-1)^{d-s} \cdot \kappa(f_\ell)_s : Q_d^s(f_\ell) \to Q_d^{s-1}(f_\ell)$ if $\ell = 1$. The extra factor of $(-1)^{d-s}$ does not affect the homology of the top row of \eqref{eq:Q-II-diagram}. To see that this definition for $\kappa_s$ makes the diagram \eqref{eq:Q-II-diagram} commute, one can first argue as in Corollary \ref{cor:FFSS-4.2}, using \eqref{eq:nontrivial-differential}, to see that the only nontrivial differential in $\rmII(\Gamma_\ell^{d(r)},T(\bss^{p^{r-1}},1)^{\otimes d}) \cong \rmII(\Gamma_\ell^{d(r)},T(\bss^{p^{r-1}},1))^{\otimes d}$ is $d_{p^{r-1}(p-1)+1} = d_{q+1}$, with its action on homogeneous tensors given by the formula \eqref{eq:differential-derivation}. Next, the chain map \eqref{eq:chain-map-h} defines a morphism of spectral sequences $\rmII(\Gamma_\ell^{d(r)},T(\bss^{p^{r-1}},1)^{\otimes d}) \to \rmII(\Gamma_\ell^{d(r)},T(\bss^{dp^{r-1}},1))$. Then commutativity of \eqref{eq:Q-II-diagram} follows from considering the image of $v_1 \otimes \cdots \otimes v_{d-s} \otimes w_1 \otimes \cdots \otimes w_s$ under this morphism, and using \cite[Lemma 3.4.1]{Drupieski:2016} to see that the cup product relation $w \cdot v = v \cdot w$ if $\ell = 0$ (resp.\ $w \cdot v = - v \cdot w$ if $\ell = 1$) holds in $\Ext_{\bsp}^\bullet(\Gamma_\ell^{*(r)},\bss^{*(1)})$ for $v \in V_{1,\ell}$ and $w \in \Wbar_{1,\ell}$.

Commutativity of the right-most square in \eqref{eq:Q-II-diagram} gives rise to a sequence of maps
	\begin{equation} \label{eq:FFSS-lemma-4.8-map}
	S^d(C_{1,\ell}) = \opH^0(Q_d(f_\ell)) \to \rmII_{q+2}^{\bullet,0} \twoheadrightarrow \rmII_{\infty}^{\bullet,0} \hookrightarrow \bbExt_{\bsp}^\bullet(\Gamma_\ell^{d(r)},T(\bss^{dp^{r-1}},1)).
	\end{equation}

\begin{lemma}[cf.\ {\cite[Lemma 4.8]{Franjou:1999}}] \label{lemma:FFSS-4.8}
The composite homomorphism \eqref{eq:FFSS-lemma-4.8-map} is an isomorphism.
\end{lemma}

\begin{proof}
The homomorphism in question is induced by the composition
	\begin{equation} \label{eq:FFSS-Lemma4.8-composite}
	S^d(C_{1,\ell}) \hookrightarrow S^d(V_{1,\ell}) \to \Ext_{\bsp}^\bullet(\Gamma_\ell^{d(r)},S_0^{dp^{r-1}(1)}) \to \bbExt_{\bsp}^\bullet(\Gamma_\ell^{d(r)},T(\bss^{dp^{r-1}},1)),
	\end{equation}
in which the second arrow is the cup product map, and the third arrow is the map in (hyper)co\-homology induced by the morphism of cochain complexes $\opH^0(T(\bss^{dp^{r-1}},1)) \hookrightarrow T(\bss^{dp^{r-1}},1)$. For $d=1$, the composite \eqref{eq:FFSS-Lemma4.8-composite} is an isomorphism by the calculations in the paragraph containing \eqref{eq:nontrivial-differential}, where we completely described the spectral sequence $\rmII(\bsirell,T(\bss^{p^{r-1}},1))$. For $d > 1$, surjectivity of \eqref{eq:FFSS-Lemma4.8-composite} then follows by multiplicativity using the chain map \eqref{eq:chain-map-h}. This implies by dimension comparison that the composite \eqref{eq:FFSS-lemma-4.8-map} is an isomorphism.
\end{proof}

\begin{corollary} \label{cor:limit-zero}
In the spectral sequence $\rmII(\Gamma_\ell^{d(r)},T(\bss^{dp^{r-1}},1))$, one has $\rmII_{\infty}^{n,m} = 0$ for $m > 0$.
\end{corollary}

\begin{proof}
By the lemma, the inclusion $\rmII_{\infty}^{\bullet,0} \hookrightarrow \bbExt_{\bsp}^\bullet(\Gamma_\ell^{d(r)},T(\bss^{dp^{r-1}},1))$ is also a surjection, and hence an isomorphism. Then the $\rmII_{\infty}$-page must be concentrated in the row $m=0$.
\end{proof}

Now we argue by induction on $t$ to establish the following statements:
\begin{itemize}
\item[$1_{t,\ell}$.] \label{item:1t} The cup product map $[S^d(V_{1,\ell})]^t \to \Ext_{\bsp}^t(\Gamma_\ell^{d(r)},S_0^{dp^{r-1}(1)})$ is an isomorphism.

\item[$\ol{1}_{t,\ell}$.] \label{item:1tbar} The cup product map $[S^d(\Vbar_{1,\ell})]^t \to \Ext_{\bsp}^t(\Gamma_\ell^{d(r)},S_1^{dp^{r-1}(1)})$ is an isomorphism.

\item[$2_{t,\ell}$.] \label{item:2t} The cup product map $[\Lambda^d(W_{1,\ell})]^t \to \Ext_{\bsp}^t(\Gamma_\ell^{d(r)},\Lambda_0^{dp^{r-1}(1)})$ is an isomorphism.

\item[$\ol{2}_{t,\ell}$.] \label{item:2tbar} The cup product map $[\Lambda^d(\Wbar_{1,\ell})]^t \to \Ext_{\bsp}^t(\Gamma_\ell^{d(r)},\Lambda_1^{dp^{r-1}(1)})$ is an isomorphism.

\item[$3_{t,\ell}$.] \label{item:3t} In the spectral sequence $\rmII(\Gamma_\ell^{d(r)},T(\bss^{dp^{r-1}},1))$, one has $\rmII_{q+2}^{n,m} = \rmII_{\infty}^{n,m}$ for all $n \leq t-q-1$.
\end{itemize}
So let $t \geq 0$, and assume by way of induction that $1_{t',\ell}$, $\ol{1}_{t',\ell}$, $2_{t',\ell}$, $\ol{2}_{t',\ell}$, and $3_{t',\ell}$ are true for all $t' < t$. We will show that $1_{t,\ell}$, $\ol{1}_{t,\ell}$, $2_{t,\ell}$, $\ol{2}_{t,\ell}$ and $3_{t,\ell}$ are also true. Using the conjugation action of $\bsPi$, one can see that $1_{t,\ell}$ and $\ol{1}_{t,\ell+1}$ (resp., $2_{t,\ell}$ and $\ol{2}_{t,\ell+1}$) are logically equivalent, so it will suffice to establish one statement from each pair.

First, from \eqref{eq:Q-II-diagram} we get commutative diagrams of the form
	\begin{equation} \label{eq:FFSS-4.9.1}
	\vcenter{\xymatrix{
	Q_d^{s+1}(f_\ell)^{n-q-1} \ar@{->}[r]^-{\kappa_{s+1}} \ar@{->}[d] & Q_d^s(f_\ell)^n \ar@{->}[r]^-{\kappa_s} \ar@{->}[d] & Q_d^{s-1}(f_\ell)^{n+q+1} \ar@{->}[d] \\
	\rmII_{q+1}^{n-q-1,(s+1)q} \ar@{->}[r]^-{d_{q+1}} & \rmII_{q+1}^{n,sq} \ar@{->}[r]^-{d_{q+1}} & \rmII_{q+1}^{n+q+1,(s-1)q}.
	}}
	\end{equation}
The three vertical maps are all isomorphisms if either $1 < s < d-1$ (by the induction hypothesis on $d$) or if $n < t-q-1$ (by the induction hypothesis on $t$). This implies that
	\begin{equation} \label{eq:FFSS-4.9.2a}
	[\opH^s(Q_d(f_\ell))]^n \cong \rmII_{q+2}^{n,sq} \qquad \text{if $1 < s < d-1$ or if $n < t - q - 1$.}
	\end{equation}
In particular, since $\opH^s(Q_d(f_\ell)) = 0$ for $s > 0$, this implies that $\rmII_{q+2}^{n,sq} = 0$ (and hence $\rmII_{q+2}^{n,sq} = \rmII_{\infty}^{n,sq}$) if either $1 < s < d-1$, or if $s \geq 1$ and $n < t-q-1$.

%Note that the map on cohomology $[\opH^s(Q_d(f_\ell))]^n \to \rmII_{q+2}^{n,sq}$ is a monomorphism for all $n$ and $s$, either because $\opH^s(Q_d(f_\ell)) = 0$, while for $s=0$ the map $S^d(C_{1,\ell}) = \opH^0(Q_d(f_\ell)) \to \rmII_{q+2}^{\bullet,0}$ is an injection by Lemma \ref{lemma:FFSS-4.8}.

We claim that $\rmII_{q+2}^{t,0} = \rmII_{\infty}^{t,0}$ and $\rmII_{q+2}^{t-q-1,q} = \rmII_{\infty}^{t-q-1,q}$. Indeed, from the $(q+2)$-page onward, there are no nontrivial differentials originating at either $\rmII_{q+2}^{t,0}$ or $\rmII_{q+2}^{t-q-1,q}$, because $\rmII^{n,m} = 0$ for $m < 0$. On the other hand, the only differentials that could hit $\rmII_{q+2}^{t,0}$ come from terms of the form $\rmII^{n,m}$ with $n \leq t-q-2$, but the inductive assumption $3_{t-1,\ell}$ implies that there are no nontrivial differentials from terms of that form. Similarly, there are no nontrivial differentials that can hit $\rmII_{q+2}^{t-q-1,q}$. Thus, $\rmII_{q+2}^{t,0} = \rmII_{\infty}^{t,0}$ and $\rmII_{q+2}^{t-q-1,q} = \rmII_{\infty}^{t-q-1,q}$. Combined with Lemma \ref{lemma:FFSS-4.8}, the first of these two equalities implies that the map $[S^d(C_{1,\ell})]^t \to \rmII_{q+2}^{t,0}$ appearing in \eqref{eq:FFSS-lemma-4.8-map} is an isomorphism. And together with Corollary \ref{cor:limit-zero}, the second equality implies that $\rmII_{q+2}^{t-q-1,q} = 0$.

Now consider the following diagram obtained by extending \eqref{eq:Q-II-diagram} to the right:
	\begin{equation} \label{eq:FFSS-4.9.1-extended}
	\vcenter{\xymatrix{
	Q_d^2(f_\ell)^{t-2q-2} \ar@{->}[r]^-{\kappa_2} \ar@{->}[d]^{\alpha} & Q_d^1(f_\ell)^{t-q-1} \ar@{->}[r]^-{\kappa_1} \ar@{->}[d]^{\beta} & Q_d^0(f_\ell)^{t} \ar@{->>}[r]^-{\text{can}} \ar@{->}[d]^{\gamma} & S^d(C_{1,\ell})^t \ar@{->}[r] \ar@{->}[d]^{\delta} & 0 \ar@{->}[d] \\
	\rmII_{q+1}^{t-2q-2,2q} \ar@{->}[r]^-{d_{q+1}} & \rmII_{q+1}^{t-q-1,q} \ar@{->}[r]^-{d_{q+1}} & \rmII_{q+1}^{t,0} \ar@{->}[r]^-{\text{can}} & \rmII_{q+2}^{t,0} \ar@{->}[r] & 0
	}}
	\end{equation}
The horizontal arrows labeled ``can'' are the canonical cokernel maps. The induction hypothesis on $t$ implies that the vertical maps labeled $\alpha$ and $\beta$ are isomorphisms, while $\delta$ was observed to be an isomorphism in the previous paragraph. The top row of \eqref{eq:FFSS-4.9.1-extended} is exact by \eqref{eq:generalized-Koszul-homology}, while exactness of the bottom row follows from the observation that $\rmII_{q+2}^{t-q-1,q} = 0$. Then by the Five Lemma, the vertical map $\gamma$ is an isomorphism. In other words, the statement $1_{t,\ell}$ is true.

Next, taking $n = t-q-1$ in \eqref{eq:FFSS-4.9.1}, the induction hypotheses on $d$ and $t$ (and the statement $1_t$) imply that the vertical maps are all isomorphisms, and hence $[\opH^s(Q_d(f_\ell))]^n \cong \rmII_{q+2}^{n,sq}$. Combined with \eqref{eq:FFSS-4.9.2a}, this gives $[\opH^s(Q_d(f_\ell))]^n \cong \rmII_{q+2}^{n,sq}$ for all $s$ provided that $n \leq t - q - 1$. In particular, if $s \geq 1$ and $n \leq t-q-1$, then $\rmII_{q+2}^{n,sq} = 0$, and hence $\rmII_{q+2}^{n,sq} = \rmII_{\infty}^{n,sq}$. Since also $\rmII_{q+2}^{n,0} = \rmII_{\infty}^{n,0}$ for all $n$, by Lemma \ref{lemma:FFSS-4.8}, this implies that statement $3_{t,\ell}$ is true.

Now to show that $2_{t,\ell}$ is true, we consider the spectral sequence $\rmI(\Gamma_\ell^{d(r)},C)$, with $C = \Kz_{dp^{r-1}}^{\bullet(1)}$ if $\ell = 0$, and $C = \Kz_{dp^{r-1}}^{\bullet(1)}\subgrp{dp^{r-1}}$ if $\ell = 1$. The shift in cohomological degree when $\ell =1$ is to ensure that certain signs (that arise from the total degrees of elements in the spectral sequence) work out appropriately. We'll describe first what happens when $\ell = 0$, and then indicate what changes are required when $\ell = 1$.

Set $q = p^{r-1}$. First, arguing as for \eqref{eq:II2-iso}, the $\rmI_1$-page of $\rmI(\Gamma_0^{d(r)},\Kz_{dq}^{\bullet(1)})$ has the form
	\begin{equation} \label{eq:I1-iso}
	\rmI_1^{-m,n} \cong
	\begin{cases} 
	\left[\Ext_{\bsp}^\bullet(\Gamma_0^{d-s(r)},S_0^{(d-s)q(1)}) \otimes \Ext_{\bsp}^\bullet(\Gamma_0^{s(r)},\Lambda_0^{sq(1)}) \right]^n & \text{if $m = s q$, $0 \leq s \leq d$,} \\
	0 & \text{if $m \not\equiv 0 \mod q$.}
	\end{cases}
	\end{equation}
In particular, the first nontrivial differential of $\rmI(\Gamma_0^{d(r)},\Kz_{dq}^{\bullet(1)})$ is $d_q$, and by induction on $d$ one has $\rmI_1^{-sq,n} \cong \left[ S^{d-s}(V_{1,0}) \otimes \Lambda^s(W_{1,0}) \right]^n$ for $0 < s < d$ (and for $s = 0$, the isomorphism holds for $n \leq t$, by $1_{t,0}$ and the induction hypothesis on $t$). Next, arguing as for \eqref{eq:Q-II-diagram}, one can show that the generalized Koszul complex $Q(\theta)$ for the map $\theta$ of \eqref{eq:theta} fits into a commutative diagram
	\begin{equation} \label{eq:Q-I-diagram}
	\vcenter{\xymatrix{
	Q_d^d(\theta)^\bullet \ar@{->}[r]^-{\kappa(\theta)} \ar@{->}[d] & Q_d^{d-1}(\theta)^\bullet \ar@{->}[r]^-{\kappa(\theta)} \ar@{->}[d]^{\cong} & \cdots \ar@{->}[r]^-{\kappa(\theta)} & Q_d^2(\theta)^\bullet \ar@{->}[r]^-{\kappa(\theta)} \ar@{->}[d]^{\cong} & Q_d^1(\theta)^\bullet \ar@{->}[r]^-{\kappa(\theta)} \ar@{->}[d]^{\cong} & Q_d^0(\theta)^\bullet \ar@{->}[d] \\
	\rmI_q^{-dq,\bullet} \ar@{->}[r]^-{d_q} & \rmI_q^{-(d-1)q,\bullet} \ar@{->}[r]^-{d_q} & \cdots \ar@{->}[r]^-{d_q} & \rmI_q^{-2q,\bullet} \ar@{->}[r]^-{d_q} & \rmI_q^{-q,\bullet} \ar@{->}[r]^-{d_q} & \rmI_q^{0,\bullet}.
	}}
	\end{equation}
It is in verifying the commutativity of \eqref{eq:Q-I-diagram} that the choice of indexing for the Koszul complex becomes important: The indexing given in \eqref{eq:Koszul-complex} ensures that $V_{1,0}$ (resp. $W_{1,0}$) is concentrated in terms of even (resp.\ odd) total degree in the $\rmI_1$-page of $\rmI(\bsirzero,\Kz_q^{\bullet(1)})$, which is relevant for applying the formula \eqref{eq:differential-derivation}.\footnote{A reindexing of the Koszul complex similar to that employed here would also seem to be necessary in the classical situation of \cite[(4.9.5)]{Franjou:1999}, to get the signs to work out correctly.} Next, in lieu of applying the chain map \eqref{eq:chain-map-h}, one considers the morphism of spectral sequences corresponding to the product map $(\Kz_q^{\bullet(1)})^{\otimes d} \to \Kz_{dq}^{\bullet(1)}$. Finally, arguing as for \cite[Lemma 3.4.1]{Drupieski:2016}, one can check $v \in V_{1,0}$ and $w \in W_{1,0}$ that the cup product relation $w \cdot v = v \cdot w$ holds in $\Ext_{\bsp}^\bullet(\Gamma_0^{*(r)},\Kz_*^{(1)}) = \Ext_{\bsp}^\bullet(\Gamma_0^{*(r)},S_0^{*(1)} \otimes \Lambda_0^{*(1)})$.

We claim that $\rmI_{q+1}^{-dq,t} = \rmI_{\infty}^{-dq,t} = 0$. Indeed, if $k \geq q+1$ and if $d_k: \rmI_k^{-dq,t} \to \rmI_k^{-dq+k,t-k+1}$ is a nontrivial differential, then \eqref{eq:I1-iso} implies that $k = jq$ for some $2 \leq j \leq d$. But for $j \geq 2$, the induction hypotheses on $d$ and $t$ imply that the cup product maps are isomorphisms
	\[
	Q_d^{(d-j)}(\theta)^{t-jq+1} \cong \rmI_q^{-(d-j)q,t-jq+1} \quad \text{and} \quad Q_d^{(d-j-1)}(\theta)^{t-(j+1)q+2} \cong \rmI_q^{-(d-j-1)q,t-(j+1)q+2}.
	\]
The generalized Koszul complex $Q_d^\bullet(\theta)$ is exact because $\theta$ is an isomorphism. Then commutativity of \eqref{eq:Q-I-diagram} imples that $\rmI_{q+1}^{-(d-j)q,t-jq+1} = 0$. Thus, there are no nontrivial differentials originating at $\rmI^{-dq,t}$ from the $(q+1)$-page onward, so $\rmI_{q+1}^{-dq,t} = \rmI_{\infty}^{-dq,t}$, and $\rmI_{\infty}^{\bullet,\bullet} = 0$ by the exactness of the complex $\Kz^{(1)}$. By entirely similar reasoning, one deduces that $\rmI_{q+1}^{-(d-1)q,t-q+1} = \rmI_{\infty}^{-(d-1),t-q+1} = 0$. (Note that the only incoming differentials to $\rmI^{-(d-1)q,\bullet}$ occur on the $\rmI_q$-page.)

Now consider the following diagram obtained by extending \eqref{eq:Q-I-diagram} to the left:
	\begin{equation} %\label{eq:FFSS-4.9.1-extended}
	\vcenter{\xymatrix{
	0 \ar@{->}[r] \ar@{->}[d] & 0 \ar@{->}[r] \ar@{->}[d] & Q_d^d(\theta)^t \ar@{->}[r]^-{\kappa(\theta)} \ar@{->}[d]^{\alpha} & Q_d^{d-1}(\theta)^{t-q+1} \ar@{->}[r]^-{\kappa(\theta)} \ar@{->}[d]^{\beta} & Q_d^{d-2}(f_\ell)^{t-2q+2} \ar@{->}[d]^{\gamma} \\
	0 \ar@{->}[r] & 0 \ar@{->}[r] & \rmI_q^{-dq,t} \ar@{->}[r]^-{d_q} & \rmI_q^{-(d-1)q,t-q+1} \ar@{->}[r]^-{d_q} & \rmI_q^{-(d-2)q,t-2q+2}.
	}}
	\end{equation}
The top row is exact by the exactness of $Q_d^\bullet(\theta)$, and the bottom row is exact by the observations of the previous paragraph. The cup product maps $\beta$ and $\gamma$ are isomorphisms by the induction hypothesis on $d$ and $t$ (and possibly also by $1_{t,0}$, if $r=1$ and $d=2$). Then by the Five Lemma, $\alpha$ is also an isomorphism. In other words, the cup product map $[\Lambda^d(W_{1,0})]^t \to \Ext_{\bsp}^t(\Gamma_0^{d(r)},\Lambda_0^{dp^{r-1}(1)})$ is an isomorphism, so statement $2_{t,0}$ is true.

In the case $\ell = 1$ and $C = \Kz_{dp^{r-1}}^{\bullet(1)}\subgrp{dp^{r-1}}$, the changes to the argument are effectively just notational. The spectral sequence $\rmI(\Gamma_1^{d(r)},C)$ is now located in the first quadrant, with
	\[
	\rmI_1^{m,n} \cong
	\begin{cases} 
	\left[\Ext_{\bsp}^\bullet(\Gamma_1^{s(r)},S_0^{sq(1)}) \otimes \Ext_{\bsp}^\bullet(\Gamma_1^{d-s(r)},\Lambda_0^{(d-s)q(1)}) \right]^n & \text{if $m = s q$, $0 \leq s \leq d$,} \\
	0 & \text{if $m \not\equiv 0 \mod q$.}
	\end{cases}
	\]
In particular, $V_{1,1}$ (resp. $W_{1,1}$) is concentrated in terms of even (resp.\ odd) total degree in the $\rmI_1$-page of the spectral sequence $\rmI(\bsirone,\Kz_q^{\bullet(1)}\subgrp{p^{r-1}})$. The diagram \eqref{eq:Q-I-diagram} now takes the form
	\begin{equation} \label{eq:Q-I-diagram-ell-1}
	\vcenter{\xymatrix{
	Q_d^d(\theta)^\bullet \ar@{->}[r]^-{\kappa(\theta)} \ar@{->}[d] & Q_d^{d-1}(\theta)^\bullet \ar@{->}[r]^-{\kappa(\theta)} \ar@{->}[d]^{\cong} & \cdots \ar@{->}[r]^-{\kappa(\theta)} & Q_d^1(\theta)^\bullet \ar@{->}[r]^-{\kappa(\theta)} \ar@{->}[d]^{\cong} & Q_d^0(\theta)^\bullet \ar@{->}[d] \\
	\rmI_q^{0,\bullet} \ar@{->}[r]^-{d_q} & \rmI_q^{q,\bullet} \ar@{->}[r]^-{d_q} & \cdots \ar@{->}[r]^-{d_q} & \rmI_q^{(d-1)q,\bullet} \ar@{->}[r]^-{d_q} & \rmI_q^{dq,\bullet}.
	}}
	\end{equation}
One argues that $\rmI_{q+1}^{0,t} = \rmI_{\infty}^{0,t} = 0$ and $\rmI_{q+1}^{q,t-q+1} = \rmI_{\infty}^{q,t-q+1} = 0$, and then extends \eqref{eq:Q-I-diagram-ell-1} to the left by zeros and applies the Five Lemma to deduce that the cup product $Q_d^d(\theta)^t \to \rmI_q^{0,t}$ is an isomorphism, thus establishing that statement $2_{t,1}$ is true.

\subsection{The case \texorpdfstring{$j \geq 2$}{j > 1} of Theorem \ref{theorem:FFSS-4.5}} \label{subsec:proof-j>1}

Now suppose $j \geq 2$. As alluded to earlier in the proof of Corollary \ref{cor:FFSS-4.2}, the only nonzero rows in the spectral sequence $\rmII(\bsirell,\Omega_{p^{r-j+1}}^{\bullet(j-1)})$ are the rows $t=0$ and $t=p^{r-j}$, and the only nontrivial differential is
	\[
	\partial : \Ext_{\bsp}^{s-p^{r-j}-1}(\bsirell,\Lambda_0^{p^{r-j}(j)}) \cong \rmII_2^{s-p^{r-j}-1,p^{r-j}} \xrightarrow{d_{p^{r-j}+1}} \rmII_2^{s,0} \cong \Ext_{\bsp}^s(\bsirell,S_0^{p^{r-j}(j)}).
	\]
We interpret $\partial$ as a map of graded spaces $\partial: W_{j,\ell} \to V_{j,\ell}$ that increases $\Z$-degrees by $p^{r-j}+1$. Then by \cite[Theorem 4.6.1]{Drupieski:2016}, $\partial$ fits into an exact sequence
	\[
	0 \longrightarrow W_{j-1,\ell} \stackrel{\alpha}{\longrightarrow} W_{j,\ell} \stackrel{\partial}{\longrightarrow} V_{j,\ell} \stackrel{\beta}{\longrightarrow} V_{j-1,\ell} \longrightarrow 0,
	\]
where $\alpha$ increases $\Z$-degrees by $p^{r-j+1}-p^{r-j}$, and $\beta$ is the map in cohomology induced by the $p$-power map $S_0^{p^{r-j}(j)} \hookrightarrow S_0^{p^{r-j+1}(j-1)}$. Set $K_{j,\ell} = \ker(\partial) \cong W_{j-1,\ell}\subgrp{p^{r-j+1}-p^{r-j}}$, and set $C_{j,\ell} = \coker(\partial) \cong V_{j-1,\ell}$. Then $\partial$ defines a generalized Koszul differential on $Q(\partial) := S(V_{j,\ell}) \otimes \Lambda(W_{j,\ell})$, and by \cite[Lemma 4.3]{Franjou:1999} one has $\Hbul(Q(\partial)) \cong S(C_{j,\ell}) \otimes \Lambda(K_{j,\ell})$.

The case $j \geq 2$ of Theorem \ref{theorem:FFSS-4.5} is now handled by (essentially) a word-for-word repetition of the inductive argument given in \cite[p.\ 696--701]{Franjou:1999}, which we echoed already for the case $j=1$ in Section \ref{subsec:proof-j=1}. First, as in \cite[Lemma 4.6]{Franjou:1999}, one considers the spectral sequence $\rmI(\bsirell,\Omega_{dp^{r-j+1}}^{(j-1)})$, and argues by induction on $j$, using Corollary \ref{cor:FFSS-4.2}, to deduce that $\rmI_1 = \rmI_\infty$. Next, one considers the spectral sequences $\rmII(\Gamma_\ell^{d(r)},\Omega_{dp^{r-j+1}}^{\bullet(j-1)})$ and $\rmI(\Gamma_\ell^{d(r)},\Kz_{dp^{r-j}}^{\bullet(j)})$ to argue by induction on $t$ that the following statements hold:
\begin{itemize}
\item[$1_{t,\ell}$.] The cup product map $[S^d(V_{j,\ell})]^t \to \Ext_{\bsp}^t(\Gamma_\ell^{d(r)},S_0^{dp^{r-j}(j)})$ is an isomorphism.

\item[$2_{t,\ell}$.] The cup product map $[\Lambda^d(W_{j,\ell})]^t \to \Ext_{\bsp}^t(\Gamma_\ell^{d(r)},\Lambda_0^{dp^{r-j}(j)})$ is an isomorphism.

\item[$3_{t,\ell}$.] In the spectral sequence $\rmII(\Gamma_\ell^{d(r)},\Omega_{dp^{r-j+1}}^{\bullet(j-1)})$, one has $\rmII_{p^{r-j}+2}^{n,m} = \rmII_{\infty}^{n,m}$ for all $n \leq t-p^{r-j}-1$.
\end{itemize}
Finally, using the conjugation action of $\bsPi$, one deduces that the other two cup product maps in Theorem \ref{theorem:FFSS-4.5} are also isomorphisms.

\begin{remark} \label{rem:unstated-theorem}
In \cite[\S5]{Franjou:1999}, Franjou, Friedlander, Scorichenko, and Suslin compute all extension groups in $\calP$ from a Frobenius twist $X^{*(s)} = X^* \circ I^{(s)}$ of a `more projective' exponential functor $X^*$ to a Frobenius twist $Y^{*(t)} = Y^* \circ I^{(t)}$ of a `less projective' exponential functor $Y^*$. Replacing each of $I^{(s)}$ and $I^{(t)}$ with either the even or odd Frobenius twist functors of the category $\bsp$, the calculations in \cite[\S5]{Franjou:1999} ought to generalize via the same lines of reasoning to the category of strict polynomial superfunctors.
\end{remark}

\makeatletter
\renewcommand*{\@biblabel}[1]{\hfill#1.}
\makeatother

\bibliographystyle{eprintamsplain}
\bibliography{complexes-and-cohomology}

\end{document}